\let\latexarabic\arabic
\let\latexdocument\document
\let\latexenddocument\enddocument

\RequirePackage[thmmarks]{ntheorem}
\makeatletter
\renewtheoremstyle{plain} 
  {\item[\hskip\labelsep \theorem@headerfont ##1\ \textup{##2}\theorem@separator]} 
  {\item[\hskip\labelsep \theorem@headerfont ##1\ \textup{##2}\ (##3)\theorem@separator]}
\makeatother

\documentclass[
  supplementary,
  lineno
]{biometrika}

\nolinenumbers % this turns them off again

\let\document\latexdocument
\let\enddocument\latexenddocument
\AtEndDocument{\printhistory}
\let\arabic\latexarabic
\def\rm{}

\usepackage{array}
\usepackage{amsmath,amssymb}
\usepackage[pdftex]{graphicx}
\usepackage{float}
\usepackage{mdframed} 

\usepackage{times}
\usepackage{bm}
\usepackage{natbib}
\usepackage{hyperref}
\usepackage{dsfont}

\usepackage[plain,noend]{algorithm2e}
\usepackage{color}


\let\hat\widehat
\let\tilde\widetilde

\DeclareMathOperator{\diag}{{\rm diag}}
\DeclareMathOperator{\Bernoulli}{{\rm Bernoulli}}
\DeclareMathOperator{\Poisson}{{\rm Poisson}}
\DeclareMathOperator{\op}{{\rm op}}

\DeclareMathOperator*{\argmin}{arg\,min}

\newcolumntype{M}[1]{>{\centering\arraybackslash}m{#1}}

\begin{document}

\jname{Biometrika}
%% The year, volume, and number are determined on publication
\jyear{?}
\jvol{?}
\jnum{?}
%% The \doi{...} and \accessdate commands are used by the production team
%\doi{10.1093/biomet/asm023}
\accessdate{~}

%% These dates are usually set by the production team
%\received{2 January 2022}
%\revised{1 April 2022}

%% The left and right page headers are defined here:
\markboth{K. Z. Lin \and J. Lei}{Dynamic SBM}

%% Here are the title, author names and addresses
\title{Dynamic clustering for heterophilic stochastic block models with time-varying node memberships}

\author{K. Z. LIN}
\affil{Department of Biostatistics, University of Washington
\email{kzlin@uw.edu}}

\author{J. LEI}
\affil{Department of Statistics \& Data Science, Carnegie Mellon University \email{jinglei@andrew.cmu.edu}}

\maketitle

\begin{abstract}
We consider a time-ordered sequence of networks stemming from stochastic block models where nodes gradually change their memberships over time, and no network at any single time point contains sufficient signal strength to recover its community structure. To estimate the time-varying community structure, we develop KD-SoS (kernel debiased sum-of-squares), a method that performs spectral clustering after a debiased sum-of-squared aggregation of adjacency matrices. Our theory demonstrates, via a novel bias-variance decomposition, that KD-SoS achieves consistent community detection in each network, even when heterophilic networks do not require smoothness in the time-varying dynamics of between-community connectivities. We also prove the identifiability of aligning community structures across time based on how rapidly nodes change communities, and develop a data-adaptive bandwidth tuning procedure for KD-SoS. We demonstrate the utility and advantages of KD-SoS through simulations and a novel analysis of the time-varying dynamics in gene coordination in the human developing brain system.
\end{abstract}

\begin{keywords}
Gene co-expression network, human brain development, network analysis,  non-parametric analysis, single-cell RNA-seq, time-varying model
\end{keywords}

\section{Introduction} \label{sec:introduction}
Longitudinal analyses of a network reveal insights into how communities of nodes are lost or created over time.
Due to the complexity of most networks, statistical methods are necessary to uncover these broad dynamics.
Simply put, suppose we observe a time-ordered sequence of networks among the same $n$ nodes represented as symmetric binary matrices $A^{(0)}, \ldots, A^{(1)} \in \{0,1\}^{n\times n}$, where for time $t \in [0,1]$, the $(i,j)$-entry of $A^{(t)}$ denotes the presence or absence of interaction between two nodes at time $t$.
Due to the non-Euclidean nature of the data, it is often challenging to determine if larger-scale community structures have changed over time and, if so, which specific nodes are changing communities at what rate.
\cite{sarkar2006dynamic} developed one of the first methods to investigate these time-varying dynamics.
However, research on the statistical properties of such estimators is recent by comparison \citep{han2015consistent}.
See \cite{kim2018review,pensky2019spectral} for a comprehensive overview. 
Our goal in this paper is to provide a theoretically new method that is both computationally efficient and capable of handling a wide range of network dynamics.

In this work, we focus on understanding the dynamics of gene coordination during human brain development; however, our methods are applicable more broadly to investigate any time-ordered sequence of networks.
Consider the single-cell RNA-seq (scRNA-seq) dataset initially published in \cite{trevino2021chromatin}, where the authors delineated a specific set of 18,160 cells representing how cycling progenitors (orange) develop into numerous types of maturing glutamatergics (shades of teal).
The authors annotated these cells and discovered a set of 993 genes associated with their development.
This data can be visualized through a UMAP \citep{mcinnes2018umap}, a non-linear dimension-reduction method (Figure \ref{fig:cell-umap}A). 
Using typical tools in the single-cell analysis toolbox such as Slingshot \citep{street2018slingshot}, we can order the cells in this lineage from the youngest to oldest cells (Figure \ref{fig:cell-umap}B) and visualize how the gene expression evolves across this lineage (Figure \ref{fig:cell-umap}C). 
However, while this simple analysis reveals apparent dynamics of the mean gene expression across pseudotime, the evolution of gene coordination patterns remains unknown.
Do the genes tightly coordinated at the beginning of development remain tightly coordinated at the end of development, and are there tightly coordinated genes that are not highly expressed?

\begin{figure}[tb]
  \centering
  \includegraphics[width=400px]{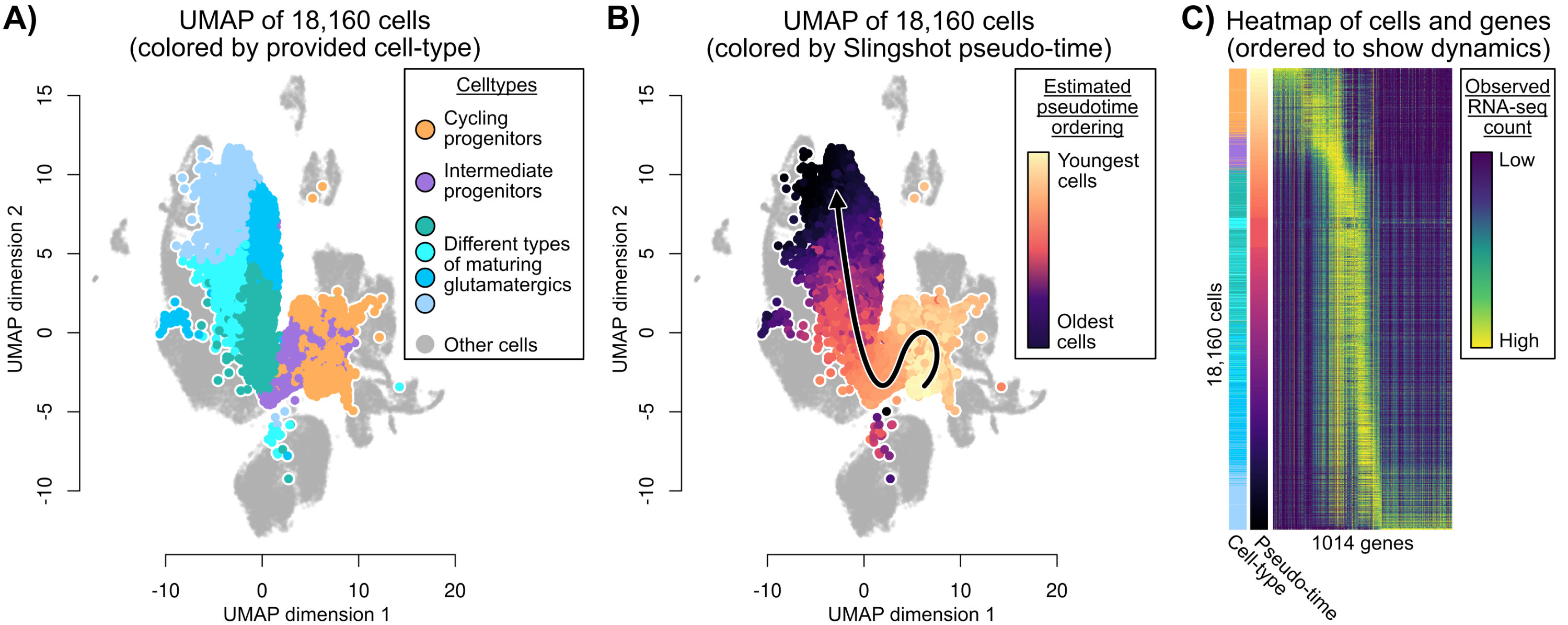}
  \caption
   { 
A) UMAP of the cells among the human developing brain, highlighting the 18,160 cells relevant to our analysis. These denote cell types, such as cycling progenitors (orange) and maturing glutamatergic neurons (shades of teal). 
B) The 18,160 cells colored based on their estimated pseudotime using Slingshot \citep{street2018slingshot}, colored from 
youngest (bright yellow) to oldest (dark purple).
C) Heatmap ordering the cells based on their estimated pseudotime, and ordering the 993 relevant genes for this development.
The gene expression for each cell is colored based on its expression (high as yellow, low as dark blue).
   }
    \label{fig:cell-umap}
\end{figure}

As reviewed in \cite{kim2018review}, many statistical models exist for time-varying networks.
This work focuses on time-varying stochastic block models (SBMs).
SBMs \citep{holland1983stochastic} are a class of prototypical networks that reveal insightful theory while being flexible enough to model many networks in practice.
Broadly speaking, an SBM represents each node as part of $K$ (unobserved) communities, and the presence of an edge between two nodes is determined solely by the community labels of the nodes.
Previous work has proven that there is a fundamental limit on how sparse the SBM can be before recovering the communities is impossible \citep{abbe2017community}. 
However, this fundamental limit could become even sparser when there is a collection of SBMs.
This has led to many different lines of work.
For example, one line of work studies the fixed community structure, where $T$ SBMs are observed with all the same community structure \citep{lei2020consistent,bhattacharyya2020general,paul2020spectral,arroyo2021inference,lei2020bias}. 
A variant is that no temporal structure is imposed across the $T$ networks, but instead, each network slightly deviates from a common community structure at random \citep{chen2020global}.
Another line of work is when $T$ time-ordered SBMs are observed, but there is a changepoint -- all the networks before or all the networks after the changepoint share the same community structure \citep{liu2018global,wang2021optimal}.

Despite the abundance of aforementioned SBM models equipped with rigorous theory, they only partially apply to our intended analysis of the human developing brain.
To provide the reader with a scope of the analysis, we plot the correlation network among the 993 genes for three different time points in Figure \ref{fig:graphs_unlabeled}.
These networks were constructed from 12 non-overlapping partitions of cells across the estimated time, and we observe potentially gradual changes in community structure over time.  
See the Appendix for more details on the preprocessing. 
Hence, we turn towards time-varying SBM models, where the community structure changes slowly over time.
To date, \cite{pensky2019spectral} and \cite{keriven2020sparse} are among the only works that study this setting.
This difficulty arises from the simple observation that changes in community structure are discrete, which prevents typical non-parametric techniques from being easily applied.
However, as discussed later, we take a different theoretical approach to analyze this problem and prove consistent estimation of each network's community under broader assumptions. 
We briefly note that, beyond time-varying SBMs, there are works on time-varying latent-position networks \citep{gallagher2021spectral, athreya2022discovering}. 
Latent-position networks are more general than SBMs, as they do not impose a community structure. 
In this work, we focus on SBMs as they are more applicable to understanding the gene coordination dynamics in the developing brain.

\begin{figure}[tb]
  \centering
  \includegraphics[width=400px]{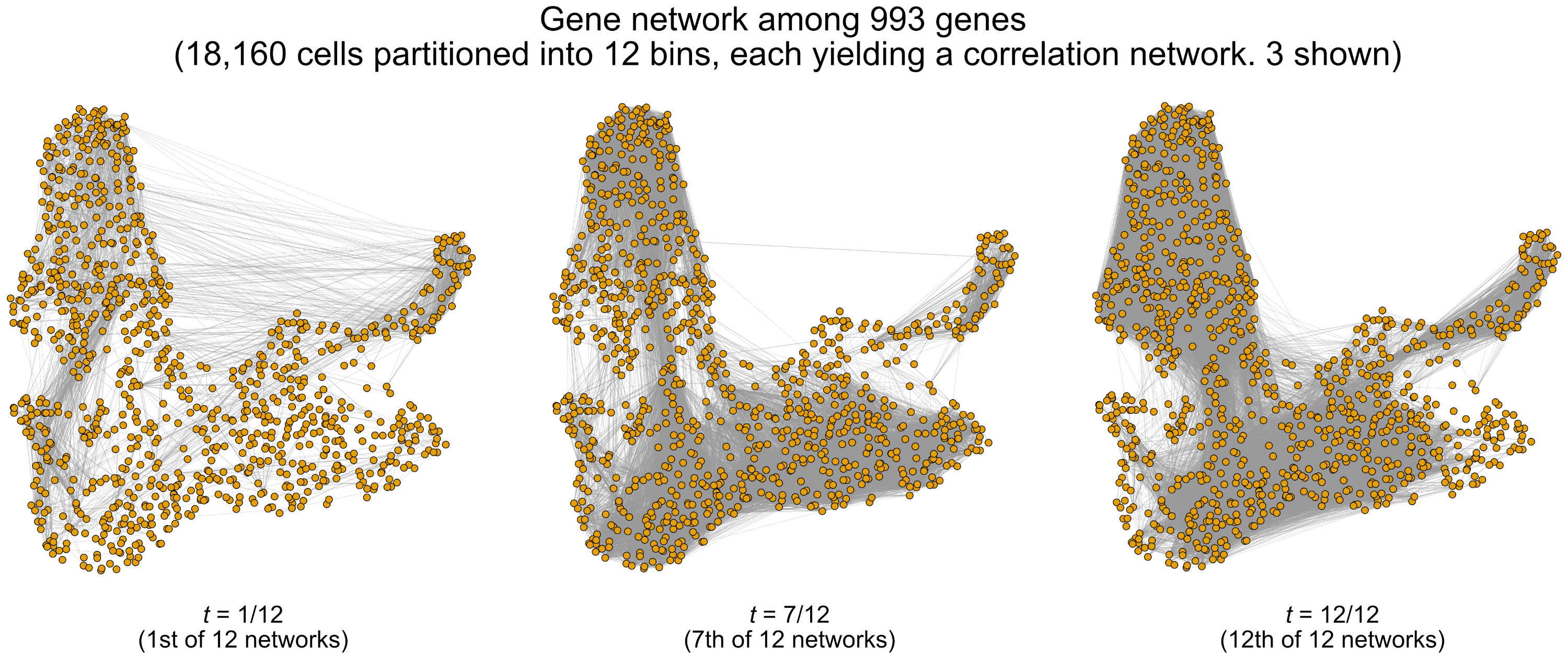}
  \caption
   { 
Three of twelve networks, for $t=1/12$ (i.e., gene network among the youngest cells), $t=7/12$, and $t=12/12$ (i.e., gene network among the oldest cells). These are constructed based on thresholding the correlation matrix among the 993 genes. The visual position of each gene is fixed for each network, but the edges among the genes vary.
   }
    \label{fig:graphs_unlabeled}
\end{figure}

The main contribution of this paper is a novel and computationally efficient method equipped with theoretical guarantees regarding community estimation in temporal SBMs with a time-varying community structure.
Our method is inspired by \cite{lei2020bias}, where a debiased sum-of-squared estimator was proven to consistently estimate communities for fixed-community multi-layer networks, allowing for both homophilic and heterophilic networks.
We adapt this to the time-varying setting by introducing a kernel smoother and prove, through a novel bias-variance decomposition, that it can consistently estimate the time-varying communities, holding all other assumptions the same.
In particular, while the nodes are gradually changing communities, we impose almost no conditions on the connectivity patterns except the positivity of the locally averaged squared connectivity matrix.
We also formalize the information-theoretic relation between the number of networks and the rate at which nodes change communities as an identifiability condition.

Our second contribution is a tuning procedure for an appropriate kernel bandwidth that also does not impose restrictions on how the community relations change across networks.
Leave-one-out tuning procedures designed for other matrix applications \citep{yang2020estimating}, where the network at time $t$ is predicted using temporally surrounding networks, are inappropriate since these procedures require community relations to change smoothly over time.
This also precludes Lepskii-based procedures \citep{pensky2019spectral}.
In contrast, our procedure is designed based on the cosine distance between eigenspaces -- for the network at time $t$, the cosine distance is computed between the eigenspaces of kernel-weighted networks for a time less than $t$ and of kernel-weighted networks for a time greater than $t$, respectively.
The bandwidth that minimizes this distance, averaged over all $t$, is deemed the most appropriate. 
We show through simulation studies and a thorough investigation of the scRNA-seq data that this procedure selects a desirable bandwidth. 

\section{Dynamic stochastic block model} \label{sec:model}

Let $n$ denote the number of nodes, and $m^{(0)} \in \{1,\ldots, K\}^n$ denote the initial membership vector, where $K$ is a fixed number of communities. 
That is, $m^{(0)}_i = k$ for $k\in \{1,\ldots,K\}$ if node $i \in \{1,\ldots,n\}$ starts in community $k$.
We posit that each of the $n$ nodes changes communities according to a $\Poisson(\gamma)$ process with $\gamma >0$, independent of all other nodes. 
This means node $i$ changes communities at random times $0 < x_{i,1} < x_{i,2} < \ldots < 1$ where the expected difference between consecutive times is $1/\gamma$, and the node changes to one of the $K-1$ other communities with some probability.
We place no assumptions about the specifics of how nodes are assigned to new communities.  
Instead, our assumptions only mandate the frequency of nodes changing communities and the independence between nodes.
This node-switching process generates membership vectors $m^{(t)}$ for $t \in [0,1]$.

Although each node can potentially change communities multiple times throughout $t\in[0,1]$, we assume that only $T$ networks at fixed time points are observed for  
\[
\mathcal{T} = \Big\{\frac{1}{T},\; \frac{2}{T},\; \ldots,\;1\Big\}.
\] 
The generative model for a specific graph $A^{(t)} \in \{0,1\}^{n \times n}$ for a time $t \in \mathcal{T}$ is as follows.
Let $B^{(t)} \in [0,1]^{K \times K}$ be a symmetric matrix that denotes the connectivity matrix among the $K$ communities for a fixed positive integer $K$, and let the sequence of matrices of $B^{(t)}$'s for $t \in [0,1]$ be deterministic.
Let $m^{(t)}$ be the random membership vector based on the above $\Poisson(\gamma)$ process.
Each membership vector $m^{(t)}$ can be encoded as one-hot membership matrix $M^{(t)} \in \{0,1\}^{n\times K}$ where $M_{ik}^{(t)} = 1$ if and only if node $i$ is in community $k$, and 0 otherwise. 
Then, the probability matrix $Q^{(t)} \in [0,1]^{n\times n}$ is defined as
\begin{equation} \label{eq:model:qt}
Q^{(t)} = \rho_n \cdot M^{(t)} B^{(t)} \big(M^{(t)}\big)^\top,
\end{equation}
for a network density parameter $\rho_n \in (0,1)$, and $P^{(t)} = Q^{(t)} - \diag(Q^{(t)})$.
The observed graph $A^{(t)}$ for time $t \in \mathcal{T}$ is then sampled according to
\begin{equation} \label{eq:adjacency_matrix}
A^{(t)}_{ij} = \begin{cases}
\Bernoulli\big(P_{ij}^{(t)}\big), &\quad \text{if } i > j,\\
0 &\quad \text{if }i =j,\\
A^{(t)}_{ji} &\quad \text{otherwise}.
\end{cases}
\end{equation}
This implies the following relation: 
\[
\mathbb{E}\big(A^{(t)}\big) = P^{(t)}  = Q^{(t)}- \diag\big(Q^{(t)}\big).
\]

For two membership matrices $M$, $M'$, define their confusion matrix $C(M, M')$ as
\begin{equation} \label{eq:confusion_population}
C_{k\ell}(M, M')= \Big|\big\{i \in \{1,\ldots,n\} : M_{ik} = 1 \text{ and } M'_{i\ell} =1 \big\}\Big|.
\end{equation}
Since the outputs of most clustering algorithms do not distinguish label permutations, to match the label permutation between $M$ and $M'$, we solve the following assignment problem,
\begin{equation} \label{eq:hungarian_assignment_population}
R(M, M') =\arg\max_{R \in \mathbb{Q}_K} \big\| \diag\{C(M,M')R\}\big\|_1,
\end{equation}
where $\mathbb{Q}_K$ is the set of $K\times K$ permutation matrices.
This can be formulated as an \emph{Hungarian assignment problem}, which can be solved via linear programming. 
Equipped with $C(M,M')$ and $R(M,M')$, we define $L(M, M')$ to be the relative Hamming distance between the two membership matrices $M$ and $M'$,
\begin{equation}\label{eq:hamming}
L(M, M') =1- \frac{1}{n} \big\|\diag\{C(M,M')R(M,M')\}\|_1\,,
\end{equation}
or, in other words, the total proportion of mis-clustered nodes after optimal alignment.
Furthermore, we define a square matrix $X\in\mathbb{R}^{K \times K}$ to be diagonally dominant if $X_{kk} > \sum_{\ell:\ell \neq k}|X_{k\ell}|$ for each $k\in\{1,\ldots,K\}$. 
If $C(M,M')R(M,M')$ and $C(M',M)R(M',M)$ are both diagonally dominant, we say that the two membership matrices $M, M'$ are \emph{alignable}.
This means there is an unambiguous mapping of the $K$ communities in $M$ to those in $M'$.

Our theoretical goal is to show the interplay between the number of nodes $n$, the number of observed networks $T$, the community switching rate $\gamma$, and the network-sparsity parameter $\rho_n$ needed to estimate the $T$ membership matrices across time consistently. 
The existing theory of single-layer SBMs has already shown that if $n\rho_n \gtrsim \log(n)$ for a single network, spectral clustering can asymptotically recover the community structure. 
At the same time, no method can achieve exact recovery if $n\rho_n \lesssim \log(n)$ \citep{bickel2009nonparametric,lei2015consistency,abbe2017community}. 
We are primarily interested in the latter setting, hoping that the temporal structure can enhance the signal for estimation.
Some previous methods and theoretical analyses for this setting require strict assumptions on connectivity matrices $\{B^{(t)}\}$ \citep{pensky2019spectral,keriven2020sparse} -- these matrices are required to vary across time smoothly and have strictly positive eigenvalues, i.e., cannot display patterns of heterophily where edges between communities are more frequent than edges within communities. 
We aim to develop a method that does not require these assumptions, building upon the work in \cite{lei2020consistent} and \cite{lei2020bias} to extend the line of work to temporal SBMs with varying communities. 
This requires a careful analysis of a 'bias'' term that bounds the impact of averaging over adjacency matrices with slight variations of the true community structure on spectral clustering.
Additionally, we wish to study the regime of community switching rate $\gamma$ that enables the researcher to align the community structure at one time point with that at the next. 
This quality is vital for interpreting the temporally dynamic network community structure, and is an aspect not studied in \cite{lei2020bias}, \cite{keriven2020sparse}, and other related work.

\section{Debiasing and kernel smoothing} 
\subsection{Estimator} \label{ss:algorithm}
Our estimator, the kernel debiased sum-of-squared (KD-SoS) spectral clustering, is motivated by \cite{lei2020bias}, where we adopt the debiased sum of squared adjacency matrices to handle heterophilic networks. 
We describe our method using the box kernel for simplicity, but the method and theory can be extended to any kernels that are bounded, continuous, symmetric, non-negative, and integrate to 1.
The estimation procedure consists of two phases: estimating the communities for each time $t$ by smoothing across time, and aligning the communities across time.

Provided a bandwidth $r \in [0,1]$ and a number of communities $K$, our estimator applies the following procedure for any $t \in \mathcal{T}$. 
First, compute the debiased sum of squared adjacency matrices, where the summation is over all networks within a bandwidth $r$,
\begin{equation}\label{eq:z_box}
Z^{(t;r)} = \sum_{s \in \mathcal{S}(t;r)}\left\{(A^{(s)})^2 - D^{(s)}\right\}, \quad \text{where}\quad \mathcal{S}(t;r) =  \mathcal{T} \cap [t-r, t+r],
\end{equation}
and $D^{(t)} \in \mathbb{R}^{n\times n}$ is the (random) diagonal matrix encoding the degrees of the $n$ nodes, i.e.,
\[
\big[D^{(t)}\big]_{ii} = \sum_{j=1}^{n}A^{(t)}_{ij}, \quad \text{for all }i \in \{1,\ldots,n\}.
\]
Second, compute eigen-decomposition of $\hat{Z}^{(t;r)}$,
\begin{equation}\label{eq:z_box2}
\hat{Z}^{(t;r)} = \hat{U}^{(t;r)} \hat{\Lambda}^{(t;r)} (\hat{U}^{(t;r)})^\top,
\end{equation}
where the diagonal entries of $\hat{\Lambda}^{(t;r)}$ are in descending order, and lastly, apply K-means clustering row-wise on the first $K$ columns of $\hat{U}^{(t;r)}$.
This yields the estimated memberships $\tilde{m}^{(t)} \in \{1,\ldots,K\}^n$. 
This debiased sum-of-squared estimator is proven in \cite{lei2020bias} to consistently estimate communities under the fixed-community setting, where the squaring of adjacency matrices enable the population connectivity matrices $\{B^{(t)}\}$ to be semidefinite, and the debiasing corrects for the additive noise incurred by this squaring. 
This completes the estimation for each individual time point.

After estimating the communities for all $T$ time points, we align the estimated communities across time. 
Specifically, initialize $\hat{M}^{(1/T)}$ as the one-hot membership matrix of $\tilde{m}^{(1/T)}$. 
Let $\delta = 1/T$. 
Then, suppose the aligned membership $\hat{M}^{(t)}$ has been obtained, and we want to align the membership for $\tilde{M}^{(t+\delta)}$, the one-hot membership matrix for $\tilde{m}^{(t+\delta)}$. 
Define the confusion matrix
\begin{equation} 
\tilde{C}^{(t,t+\delta)} = C(\hat M^{(t)},\tilde M^{(t+\delta)})\,,
\end{equation}
according to the definition in \eqref{eq:confusion_population}, and solve the following assignment problem,
\begin{equation} 
\hat{R}^{(t,t+\delta)} = R(\hat M^{(t)},\tilde M^{(t+\delta)}).
\end{equation}
As mentioned in \eqref{eq:hungarian_assignment_population}, this is called a Hungarian assignment problem and can be solved in practice via linear programming.
Then, we align $\tilde M^{(t+\delta)}$ with $\hat M^{(t)}$ by using 
\[
\hat M^{(t+\delta)}=\tilde M^{(t+\delta)}\hat{R}^{(t,t+\delta)}\,.
\]
Let the estimated memberships for time $t$ to be $\hat{m}^{(t)}$ where $\hat{m}^{(t)} = k$ if and only if $\tilde{m}^{(t)} = \ell$ and $\hat{R}^{(t,t+\delta)}_{\ell k} = 1$. 
Finally, we return the final estimated memberships  $\hat{m}^{(t)}$ for $t\in\mathcal{T}$. We document the pseudocode of KD-SoS in the Appendix.

Optionally, we can compute if $\tilde{C}^{(t,t+\delta)}\hat{R}^{(t,t+\delta)}$ and $(\tilde{C}^{(t,t+\delta)}\hat{R}^{(t,t+\delta)})^\top$ are both diagonally dominant for all $t\in\mathcal{T}\backslash \{1\}$. 
If so, we say that the entire sequence of communities in $\mathcal{T}$ is \emph{alignable}, which means we can track the evolution of specific nodes and communities across time.

\subsection{Bias-variance tradeoff for spectral clustering}

We first describe the bias-variance decomposition foundational to our work.
Let $n_1^{(t)},\ldots,n_K^{(t)}$ denote the number of nodes in each community at time $t$, and $n_{\min}^{(t)} = \min\{n^{(t)}_1,\ldots,n_K^{(t)}\}$.
Let $\Delta^{(t)} \in \mathbb{R}^{K\times K}$ denote the diagonal matrix where 
\[
\diag\big(\Delta^{(t)}\big) = \big\{n_1^{(t)},\ldots, n_K^{(t)}\big\}.
\]
Let  $\Pi^{(t)} = M^{(t)} (\Delta^{(t)})^{-1}M^{(t)})^\top$ be the projection matrix of the column subspace of $M^{(t)}$. Additionally, define the noise matrix $X^{(t)} = P^{(t)} - A^{(t)}$.
Observe the following bias-variance decomposition.
\begin{lemma} \label{lem:decomposition}
Given the model in Section \ref{sec:model}, the following deterministic equality holds,
\begin{align}
& \sum_{s \in \mathcal{S}(t;r)}(A^{(s)})^2 - D^{(s)}=\underbrace{\Big[  \sum_{s \in \mathcal{S}(t;r)}(Q^{(s)})^2 - \Pi^{(t)} (Q^{(s)})^2 \Pi^{(t)}\Big]}_{I}  \label{eq:ideal_decomposition}\\
&\qquad+ \underbrace{\Big[  \sum_{s \in \mathcal{S}(t;r)} \big[\diag\{Q^{(t)}\}\big]^2 - 
Q^{(t)}\diag(Q^{(t)}) - \diag(Q^{(t)})Q^{(t)}\Big]}_{II}
 \nonumber\\
&\qquad+\underbrace{\Big[  \sum_{s \in \mathcal{S}(t;r)}X^{(s)}P^{(s)}+ P^{(s)}X^{(s)} \Big]}_{III} + \underbrace{\Big[  \sum_{s \in \mathcal{S}(t;r)}(X^{(s)})^2- D^{(s)}\Big]}_{IV}
\nonumber \\
&\qquad+\underbrace{\Big[ \sum_{s \in \mathcal{S}(t;r)}\Pi^{(t)} (Q^{(s)})^2 \Pi^{(t)}\Big]}_{V}.\nonumber
\end{align}
\end{lemma}
We deem this decomposition as the \emph{bias-variance decomposition} for dynamic SBMs since term $I$ represents the deterministic bias dictated by nodes changing communities, term $II$ represents the deterministic diagonal bias, term $III$ represents a random error term centered around 0, term $IV$ represents the random variance term, and term $V$ represents the deterministic signal matrix containing the community information.
We note that this decomposition differs from those used in \cite{pensky2019spectral} and \cite{keriven2020sparse}, which instead yield a decomposition that requires smoothness assumptions in $\{B^{(t)}\}$ to derive community-consistency. 

\subsection{Consistency of time-varying communities} \label{ss:consistency}

In the following, we discuss the assumptions and theoretical guarantees for KD-SoS. 
We define the following notation.
For two sequences $a_n$ and $b_n$, we define $a_n = O(b_n)$, $a_n= o(b_n)$, and $a_n = \omega(b_n)$ to denote $a_n$ is asymptotically bounded above by $b_n$ by a constant, $\lim a_n/b_n = 0$, or $\lim a_n/b_n = \infty$ respectively.
For a symmetric matrix $X$, let $\lambda_{\min}(X)$ denote its smallest eigenvalue in absolute value.

\vspace{.5em}
\begin{assumption}[Asymptotic regime] \label{ass:asymptotic}
Assume a sequence where $n$ and $T$ are increasing, $n,T\geq 3$, and $T \log(T)/n = o(1)$.
Additionally, $\rho_n$ and $\gamma$ can vary with $n$ and $T$, but there exists a constant $c_1$ such that $n\rho_n \leq c_1$. 
Furthermore, assume $K$ is fixed.
\end{assumption}
\vspace{.5em}

We codify the membership dynamics described in Section \ref{sec:model} with the following assumption.
\vspace{1em}
\begin{assumption}[Independent Poisson community changing rate] \label{ass:poisson}
Assume for a given community switching rate  $\gamma\geq 0$, each node changes memberships at random times between $[0,1]$ according to a $\Poisson(\gamma)$ process, independent of all other nodes.
\end{assumption}

\vspace{1em}
\begin{assumption}[Stable community sizes] \label{ass:cluster_size}
Assume that across all $t \in [0,1]$ and all communities $k\in\{1,\ldots,K\}$, there exists a constant $c_2$ independent of $n,T,\gamma,\rho_n$ satisfying $1 \leq c_2$ such that 
\[
\mathbb P\left\{n_k^{(t)} \in \Big[\frac{1}{c_2 K}\cdot n, \;\frac{c_2}{K}\cdot n\Big], \quad \text{for all } k\in\{1,\ldots,K\},\;t\in \mathcal T\right\}\ge 1-\epsilon_{c_2,n}.
\]
for some $\epsilon_{c_2,n}\rightarrow 0$.
\end{assumption}

\vspace{1em}
\begin{assumption}[Minimum eigenvalue of aggregated connectivity matrix] \label{ass:connectivity_min_eig}
Assume that the sequence $\{B^{(t)}\}$ from $t\in[0,1]$ is fixed and is an integrable process across each $(i,j)\in\{1,\ldots,K\}^2$ coordinate. Additionally, for a chosen $\delta>0$, we define
\[
c_{\delta} = \min_{\substack{t_1, t_2 \in [0,1],\\t_2-t_1 \geq 2\delta}} \lambda_{\min}\Big\{\frac{1}{t_2-t_1}\int_{s=t_1}^{t_2}(B^{(s)})^2ds\Big\} \geq 0.
\]
\end{assumption}

\vspace{1em}
\begin{assumption}[Alignability] \label{ass:alignability}
Assume that along the sequence of $\gamma$ and $T$,
\begin{equation} \label{eq:alignability}
\gamma/T = o(1).
\end{equation}
\end{assumption}

\vspace{1em}
\begin{remark}[Additional remark for Assumption \ref{ass:cluster_size}] \label{remark_cluster_size}
Assumption \ref{ass:cluster_size} extends the balanced community size condition from a single time point to a uniform version across all time points. 
Notably, this assumption only restricts how nodes are assigned to new communities through the community sizes. 
It precludes the scenario where communities vanish during the time interval $[0,1]$.  
This assumption serves two purposes: 
First, this is needed to control the error bound at each time point. 
Second, when combined with Assumption \ref{ass:alignability}, it guarantees the alignability of estimated communities across time.  
The exact relationship between $c_2$ and $\epsilon_{c_2,n}$ depends on the switching rate $\gamma$, as well as the transition probabilities between communities when a node changes membership. 
We provide a concrete example in Section \ref{sec:identifiability} of how specific transition mechanisms can satisfy Assumption \ref{ass:cluster_size} with high probability.
\end{remark}

\vspace{1em}
\begin{remark}[Additional remark for Assumption \ref{ass:connectivity_min_eig}] \label{remark_eigenvalue1}
Assumption \ref{ass:connectivity_min_eig} states that column space of the matrices $\{(B^{(t)})^2\}$ should span enough of $\mathbb{R}^K$  in an average sense among all $t \in [t_1,t_2]$. 
That is, $B^{(t)}$ can be rank deficient for any particular $t \in [t_1,t_2]$, but as long as $\delta$ is large enough, the average of $\{(B^{(t)})^2\}$ is full rank. 
As we will discuss later, $c_{\delta}$ has a nuanced relation with our bandwidth $r$ and the consistency of our estimator -- estimating the community structure consistently for each time $t$ will be difficult if we choose a bandwidth $r=\delta$ where $c_{\delta} \approx 0$. 
\end{remark}

\vspace{1em}
\begin{remark}[Additional remark for Assumption \ref{ass:alignability}]
As we will show later in Section \ref{sec:identifiability}, Assumption \ref{ass:alignability} is a label permutation identifiability assumption. 
Without it, KD-SoS can still estimate each network's community structure. However, it would be difficult to align the communities across time, where ``alignability'' will be defined later as the main focus of Section \ref{sec:identifiability}.
Recall that since each node changes memberships independently of one another according to the Poisson($\gamma$) process, the expected number of nodes to change memberships within a time interval of $1/T$ (i.e., the time elapsed between two consecutively observed networks) is roughly $n\gamma/T$ if $\gamma/T\lesssim 1$. 
Combined with Assumption \ref{ass:cluster_size}, a more explicit equivalent statement of \eqref{eq:alignability} is
\[
n\gamma/T = o(n/K).
\]
This demonstrates the intuition that the networks' communities are alignable across time if the number of changes between consecutive networks is less than the smallest community size.
\end{remark}

\vspace{1em}
Provided these assumptions, KD-SoS's estimated communities have the following point-wise relative Hamming estimation error for the network at time $t \in \mathcal{T}$. 
Let the function $(x)_+$ denote $\min\{0,x\}$.

\begin{theorem} \label{thm:consistency}
Given Assumptions \ref{ass:asymptotic}, \ref{ass:poisson},  \ref{ass:cluster_size}, and \ref{ass:connectivity_min_eig} for the model in Section \ref{sec:model}, for a bandwidth $r \in [0,1]$ satisfying $(rT+1)^{1/2}n\rho_n \geq c_3 \log^{1/2}(rT+ n+1)$ for some constant $c_3> 1$, then at any particular $t \in \mathcal{T}$,
\begin{equation} \label{eq:theorem_rate}
L\big(M^{(t)}, \hat{M}^{(t)}\big) \leq c \cdot \frac{1}{\{1-(\gamma r+\log(n)/n)^{1/2}\}^2_+}\cdot\Big\{\gamma r + \frac{\log(n)}{n} + \frac{1}{n^2} + \frac{\log(rT+n)}{rTn^2\rho_n^2}\Big\},
\end{equation}
with probability at least $1-O\{(rT+n)^{-1}\}-\epsilon_{c_2,n}$ for some constant $c>0$ that depends on $c_1$, $c_2$, $c_3$, $c_\delta$, and $K$. 
\end{theorem}

The proof of Theorem \ref{thm:consistency} relies on the bias-variance decomposition stated in Lemma \ref{lem:decomposition}, where techniques developed in  \cite{lei2020bias} are used to bound the ``variance'' terms while a new, detailed analysis tracks how membership changes affect the ``bias'' term.
Observe that if $\gamma r$ is close to 1 or larger, then our bound in Theorem \ref{thm:consistency} is vacuously true since $L(M^{(t)}, \hat{M}^{(t)})$ has to be less than 1, see \eqref{eq:hamming}.

\vspace{1em}
\begin{remark}[Explicit relation between $r$ and minimal eigenvalue in Assumption \ref{ass:connectivity_min_eig}]
We expand upon Remark \ref{remark_eigenvalue1}.
In Theorem \ref{thm:consistency}, we state the bandwidth $r$ separately from the bandwidth $\delta$ used to define the minimum eigenvalue $c_{\delta}$ stated in Assumption \ref{ass:connectivity_min_eig} for simplicity of exposition.
We can derive a similar theorem where both bandwidths are the same, i.e., $r=\delta$.
This is because the minimum eigenvalue $c_{\delta}$ only appears in the denominator when applying Davis-Kahan.
Hence, we can rewrite RHS of \eqref{eq:theorem_rate} to explicitly include the dependency on $c_{\delta}$, which would result in an upper bound proportional to
\[
\frac{1}{c_{\delta}^2\cdot\{1-(\gamma r + \log(n)/n)^{1/2}\}^2_+} \cdot\Big\{\gamma r + \frac{\log(n)}{n} + \frac{1}{n^2} + \frac{\log(rT+n)}{rTn^2\rho_n^2}\Big\}.
\]
If $c_{\delta} = 0$, the above equation would equal infinity, yielding a vacuously true upper bound.
\end{remark}

\vspace{1em}
\begin{remark}[Extension of Theorem \ref{thm:consistency} to be uniform over time] \label{rem:time_uniform}
Using the same assumptions as Theorem \ref{thm:consistency}, if $c_3$ is large enough, one can derive the same upper bound of $L(M^{(t)}, \hat{M}^{(t)})$ for each $t \in \mathcal{T}$ that holds with probability $1-O\{(rT+n)^{-c_4}\} - \epsilon_{c_2,n}$ where $c_4$ is a linear and increasing function of $c_3$ and can be chosen to be larger than $1$.
We do not display the proof due to its repetitive nature; however, the key insight is that all the probabilistic bounds underlying Theorem 1 have an exponential rate (i.e., Bernstein's inequality and Theorem 3 in \cite{lei2020bias}). 
Hence, a union bound over all $T$ time points yields a uniform bound that holds with probability $1-O\{T(rT+n)^{-c_4}\} - \epsilon_{c_2,n}$. The term $O\{T(rT+n)^{-c_4}\}$ is $o(1)$ since $c_4>1$ and $T \log(T)/n = o(1)$, as stated in Assumption \ref{ass:asymptotic}.
This means that Theorem \ref{thm:consistency} can hold uniformly over time at the same rate in certain asymptotic settings.
\end{remark}

\vspace{1em}
\begin{remark}[Relation with constant network density]
Theorem \ref{thm:consistency} assumes that network density $\rho_n$ itself does not depend on $T$ for simplicity of theoretical exposition.
For settings where the density $\rho_n$ varies with $T$ itself, the techniques to demonstrate how to adapt the weight of each network based on the local density from \cite{levin2022recovering} may be applicable.
\end{remark}

\vspace{1em}
We now derive an upper bound for the relative Hamming error when we use the near-optimal bandwidth $r$.

\begin{corollary}[Near-optimal bandwidth] \label{cor:bandwidth}
Consider the setting in Theorem \ref{thm:consistency} with the bandwidth
\[
r^* = \min\Big\{c \cdot \frac{1}{(\gamma T)^{1/2}n\rho_n},1\Big\},
\]
for some constant $c>0$ that depends on $c_1$, $c_2$, $c_3$, $c_\delta$, and $K$.
If the asymptotic setting satisfies
\[
\gamma r^* = \Big(\frac{\gamma}{T}\Big)^{1/2} \cdot \frac{1}{n\rho_n} \ll 1,
\]
then the bandwidth $r^*$ minimizes the rate in Theorem \ref{thm:consistency} up to logarithmic factors.
\end{corollary}

Observe that $r^*$ in Corollary \ref{cor:bandwidth} captures an intuitive behavior. 
If the number of nodes $n$ or network density $\rho_n$ increases, then there is more signal in each network, reducing the bandwidth $r^*$. 
If the community switching rate $\gamma$ increases, there is less incentive to aggregate across networks, reducing $r^*$. 
Loosely speaking, the box kernel roughly averages over $O(Tr^*)$ networks, meaning that the number of networks relevant for computing the community structure of network at time $t$ is approximately $O(T^{1/2})$ if $\gamma$ and $n\rho_n$ (the expected number of edges per node) are held constant.
This means the bandwidth grows more slowly than the total number of networks $T$, which is reasonable. 
Next, we state the resulting relative Hamming error bound stemming from this choice of bandwidth $r^*$. 
In particular, we are interested in two regimes based on whether $r^*\rightarrow 1$ (i.e., averaging across all $T$ networks asymptotically) or $r^*\rightarrow 0$ (i.e., averaging across a smaller and smaller proportion of the $T$ networks asymptotically).

\vspace{1em}
\begin{corollary}[Slow community-changing regime] \label{cor:slow}
Given Assumptions \ref{ass:asymptotic}, \ref{ass:poisson}, \ref{ass:cluster_size}, and \ref{ass:connectivity_min_eig} for the model in 
Section \ref{sec:model},
and bandwidth $r^*$
defined in  Corollary \ref{cor:bandwidth},
consider an asymptotic sequence of $\{n,T,\gamma,\rho_n\}$ where
\begin{equation}\label{eq:cor:slow_membership_regime}
\gamma \rightarrow 0, \quad \text{and}\quad T^{1/2}n\rho_n = \omega\big\{\log^{1/2}(T+n)\big\}.
\end{equation}
In this setting, $r^* \rightarrow 1$ and KD-SoS has a relative Hamming error upper bound of
\[
L\Big(M^{(t)}, \hat{M}^{(t)}\Big) = O\Big\{\gamma + \frac{\log(n)}{n} + \frac{1}{n^2} + \frac{\log(T+n)}{T(n\rho_n)^2}\Big\} \rightarrow 0,
\]
with probability $1-O((T+n)^{-1})-\epsilon_{c_2,n}$ for any particular $t \in \mathcal{T}$.
\end{corollary}

\vspace{1em}
\begin{corollary}[Fast community-changing regime] \label{cor:fast}
Given Assumptions \ref{ass:asymptotic}, \ref{ass:poisson}, \ref{ass:cluster_size}, and \ref{ass:connectivity_min_eig},  for the model in 
Section \ref{sec:model}, and bandwidth $r^*$ defined in  Corollary \ref{cor:bandwidth}, consider an asymptotic sequence of $\{n,T,\gamma,\rho_n\}$  where
\begin{equation}\label{eq:cor:fast_membership_regime}
\gamma = \omega(1)\,, \quad \text{and} \quad \gamma = o\Big\{\frac{T(n\rho_n)^2}{\log(T+n)}\Big\}.
\end{equation}
In this setting, $r^*\rightarrow 0$ and KD-SoS has a relative Hamming error of
\[
L\Big(M^{(t)}, \hat{M}^{(t)}\Big) = O \Big\{\frac{\gamma^{1/2}}{T^{1/2}n\rho_n} + \frac{\log(n)}{n} + \frac{1}{n^2} + \frac{\gamma^{1/2}\log\big(T^{1/2}/(\gamma^{1/2}n\rho_n) + n\big)}{T^{1/2}n\rho_n}\Big\} \rightarrow 0,
\]
with probability $1-O[\{\log^{1/2}(T+n)/(n^2\rho_n^2) + n\}^{-1}] - \epsilon_{c_2,n}$. for any particular $t \in \mathcal{T}$.
\end{corollary}

\vspace{.5em}
Observe that the two conditions \eqref{eq:cor:slow_membership_regime} and \eqref{eq:cor:fast_membership_regime} dichotomize the settings in a ``slow community switching regime'' and a ``fast community switching regime'' respectively. 
In the former setting, the nodes become less and less likely to change communities along the asymptotic sequence of $\{n,T,\gamma,\rho_n\}$, eventually resulting in KD-SoS averaging over all $T$ networks. 
In this regime Corollary \ref{cor:slow} concurs with the recent results in static multi-layer SBM \citep{lei2020consistent,lei2020bias,lei2024computational}, which imply that $T^{1/2} n\rho_n\gg \log^{1/2}(T+n)$ is nearly necessary up to a logarithm factor for consistent community estimation.  
In the latter setting, the bandwidth converges to 0 because the nodes change communities too quickly relative to the other parameters $(T, n,\rho_n)$.
Observe that if $n\rho_n = \log^{1/2}(T+n)$, then \eqref{eq:cor:fast_membership_regime} is equivalent to $\gamma/T=o(1)$, which is the requirement posed in Assumption \ref{ass:alignability}. 
This further upper bounds how often nodes can change communities relative to the total number of networks, $T$. 
As we will show in the next section, however, this requirement is not only for the consistent estimation of a network's community structure but also for ensuring the alignability of the communities across the $T$ networks.

\subsection{Identifiability bound for aligning communities across time} \label{sec:identifiability}

While Theorem \ref{thm:consistency} proves consistent estimation of the community structure at each time $t$, we now turn our attention towards proving that the estimated community structure at each time $t$ can be aligned with those at the previous time $s=t-1/T$. 
This is an important but separate concern from the consistency proven in Theorem \ref{thm:consistency} as we strive to track how individual communities evolve over time. 
This aspect has not been studied in \cite{lei2020bias} and \cite{keriven2020sparse}.
Our estimator uses the Hungarian assignment \eqref{eq:hungarian_assignment_population} to align communities across time since the K-means algorithm returns unordered memberships.
For this section, we will work under the pretense that for a sequence of membership matrices $M^{(1/T)},M^{(2/T)},\ldots,M^{(1)}$, we have already applied Hungarian assignment to each consecutive pair of membership matrices to optimally permute the column order.
Our discussion of alignability here will demonstrate that even after this column permutation, there may still be detrimental ambiguity in tracking individual communities over time.
As alluded to in Section \ref{ss:consistency}, we prove how the alignability of communities across time is related to Assumption \ref{ass:alignability}. We define it formally below.

\begin{definition}[Alignability of memberships across time] \label{def:alignability}
Let $M^{(1/T)}$, $M^{(2/T)}$, $\ldots,~M^{(1)}$ denote the $T$ membership matrices. We say the sequence of memberships is \emph{alignable} if 
\[
C(M^{(t)}, M^{(t+1/T)}) \quad \text{and}\quad 
C(M^{(t+1/T)}, M^{(t)}) \quad \text{are both diagonally dominant}
\]
for all $t \in \{1/T, 2/T, \ldots, (T-1)/T\}$, where the confusion matrices $C(M^{(t)}, M^{(t+1/T)})$ are defined in \eqref{eq:confusion_population}.
\end{definition}

We view $M^{(1/T)},M^{(2/T)},\ldots,M^{(1)}$ as the ``true'' membership matrices that encode the time-varying community structure that we wish to estimate, even though these are technically random matrices.
From the data-generative point of view, alignability implies that $R^{(t,t+1/T)}=I_K$ defined in \eqref{eq:hungarian_assignment_population} for all $t$.
Indeed, for times $t$ and $t+1/T$, if the optimal assignment between the unobserved communities $M^{(t)}$ and $M^{(t+1/T)}$ is not identity, then there is no hope of recovering the alignment of the estimated communities consistently.
Hence, intuitively, alignability requires that nodes do not switch memberships too quickly, relative to the amount of time between consecutive networks, $1/T$.

Below, we first prove that when $\gamma$ is in a regime that violates Assumption \ref{ass:alignability}, there always exists a non-vanishing probability that $T$ networks cannot be aligned. 
Later, we prove that when $\gamma$ is in a regime that satisfies Assumption \ref{ass:alignability} for specifically a two-community model, then all $T$ networks are alignable with high probability.  
Since tracking the community sizes over time under Assumption \ref{ass:poisson} involves specifying the transition probabilities and the number of times such transition occurs in a single time interval, to simplify the discussion in this subsection, we will consider an alternative discrete approximation of Assumption \ref{ass:poisson}.

\begin{assumption}[Discrete approximation of Assumption \ref{ass:poisson}]\label{ass:bernoulli}
 For each $t\in\mathcal T\backslash\{1\}$, each node changes its community membership from time $t$ to $t+1/T$ independently with probability $\gamma/T$.
\end{assumption}

\begin{proposition}[Lack of alignability]\label{prop:no-alignability}
Given Assumptions \ref{ass:asymptotic} and \ref{ass:bernoulli} for the model in Section \ref{sec:model},
if
\[
\gamma \geq T \cdot \log\bigg[
\bigg\{
1 - \frac{1}{2}\cdot\Big(\frac{(2n)^{1/2}}{T-1}\Big)^{1/n}
\bigg\}^{-1}
\bigg],
\]
then the probability that the set of random membership matrices $M^{(1/T)},M^{(2/T)},\ldots,M^{(1)}$ is not alignable is strictly bounded away from 0.
\end{proposition}

Observe that as $n$ and $T$ tend to infinity, the relation in Proposition \ref{prop:no-alignability} simplifies to 
\[
\gamma \geq T \cdot \log(2) \approx T \cdot 0.693,
\]
and when $\gamma /T = 0.693$, each node has roughly a 50\% probability of switching communities between each consecutive pair of observed networks.
The proof of the lack of alignability first revolves around the observation that if more than $n/2$ nodes change memberships between consecutive times $t$ and $t+1/T$, i.e.,
\begin{equation} \label{eq:no-alignability}
\big\|M^{(t)} - M^{(t+1/T)}\big\|_0 > n,
\end{equation}
where $\|x\|_0$ denotes the number of non-zero elements in $x$, then, deterministically, the Hungarian assignment between the unobserved membership matrices $M^{(t)}$ and $M^{(t+1/T)}$ will not be the identity matrix. 
This means the two membership matrices are not alignable.
The proof shows that the event \eqref{eq:no-alignability} occurs with non-vanishing probability.

In contrast, to show that $\gamma/T = o(1)$ ensures alignability, our proof strategy is more delicate, as we need to ensure alignability between time $t$ and $t+1/T$ for each $t \in \mathcal{T}\backslash \{1\}$. 
First, we discuss a deterministic condition that ensures alignability among all the community structures.

\begin{proposition}[Deterministic condition for alignability] \label{prop:yes-alignability-deterministic}
Assume any fixed sequence of membership matrices $M^{(1/T)},M^{(2/T)},\ldots,M^{(1)}$.
For this sequence, if the number of nodes that change memberships between time $t$ and $t+1/T$ is less than half of the smallest community size at time $t$ for each pair of consecutive time points, meaning
\[
\big\|M^{(t)} - M^{(t+1/T)}\big\|_0 < \min_{k\in\{1,\ldots,K\}}\sum_{i=1}^{n}M^{(t)}_{ik}, 
\quad \text{for some time } t\in\mathcal{T}\backslash\{1\},
\]
then deterministically this sequence of matrices $M^{(1/T)},M^{(2/T)},\ldots,M^{(1)}$ is alignable.
\end{proposition}

Proposition \ref{prop:yes-alignability-deterministic} highlights that alignability is guaranteed if not many nodes change communities relative to the smallest community size. 
Next, the following proposition ensures that if $\gamma/T = o(1)$, this event occurs with high probability, specifically focusing on a two-community model (i.e., $K=2$), where each community initially has equal sizes.

 \begin{proposition}[Alignability in a two-community model]\label{prop:yes-alignability-probabilistic}
Given Assumptions \ref{ass:asymptotic} and \ref{ass:bernoulli} for the model in Section \ref{sec:model} for a two-community model (i.e., $K=2$) initialized at $t=0$ to have equal community sizes, if $\gamma/T = o(1)$, then with probability at least $1-2/T$, the set of random membership matrices $M^{(1/T)},M^{(2/T)},\ldots,M^{(1)}$ is alignable.
\end{proposition}

This proof involves a novel recursive martingale argument since we need to ensure that alignability holds for the entire sequence of membership matrices across each pair of consecutive time points.  
We expect the argument to hold for more general settings under mild conditions, provided that more careful bookkeeping is employed.
As an aside, our proof shows that the community sizes stay close to $n/2$ for all time points, demonstrating that Assumption \ref{ass:cluster_size} can be satisfied with high probability.

\vspace{1em}
\begin{remark}[No assumptions on the specific node-switching mechanism]
Both the negative and positive results in Propositions \ref{prop:no-alignability} and  \ref{prop:yes-alignability-deterministic}, respectively, formalize the conditions under which alignability is possible without assuming any specific mechanism for how nodes change memberships or knowledge of the connectivity matrices $\{B^{(t)}\}$'s. 
The only assumption needed is Assumption \ref{ass:bernoulli}.
As our simulations suggest, this means that the nodes can change their memberships according to a time-varying Markov transition matrix.
\end{remark}

\vspace{1em}
\begin{remark}[Investigating the behavior of individual nodes]
Consider the two-community setting of Proposition \ref{prop:yes-alignability-probabilistic} where the ratio $\gamma/T$ is small enough that the distinction between the community changing mechanism in Assumption \ref{ass:poisson} and Assumption \ref{ass:bernoulli} is negligible.
With a time-uniform community recovery error bound and alignability, we can track the community trajectory of each node.  
Let $\hat{m}_i^{(t)}$ be the estimated and aligned group membership of node $i$ at time $t$, then Remark \ref{rem:time_uniform} regarding time-uniform estimation error and Proposition \ref{prop:yes-alignability-probabilistic} regarding alignability jointly imply that with high probability, $\hat{m}_i^{(t)}$ correctly track most of the nodes, 
\[
\sup_{t\in\mathcal{T}}\frac{1}{n}\sum_{i=1}^n \mathds{1}(\hat{m}_i^{(t)} \neq m_i^{(t)})\leq c \cdot \frac{1}{\{1-(\gamma r+\log(n)/n)^{1/2}\}^2_+}\cdot\Big\{\gamma r + \frac{\log(n)}{n} + \frac{1}{n^2} + \frac{\log(rT+n)}{rTn^2\rho_n^2}\Big\},
\]
for a universal constant $c$, where $\mathds{1}(\cdot)$ denotes the indicator function.
\end{remark}

\section{Numerical experiments}
In this section, we describe the tuning procedure for choosing $r$ in a data-adaptive manner, as the optimal bandwidth in Corollary \ref{cor:bandwidth} involves nuisance parameters.
Our simulations demonstrate that 1) the tuning procedure reflects the oracle bandwidth, and 2) KD-SoS and the tuning procedure combined outperform other estimators for time-varying SBMs.

\subsection{Tuning procedure} \label{sec:tuning}
We design the following procedure to tune the bandwidth $r$ in practice. 
Observe that typical tuning procedures for time-varying scalar- or matrix-valued data often rely on the local smoothness of the observed data across time. 
For example, this may be predicting the network $A^{(t)}$ using all other networks $\{A^{(s)}\}$ for $s \in \mathcal{S}(t;r)\backslash\{t\}$ for $\mathcal{S}(t;r)$ defined in \eqref{eq:z_box}, but such a procedure would necessarily require additional smoothness assumptions on the connectivity matrices $\{B^{(t)}\}$ on top of our weaker integrability assumption in Assumption \ref{ass:connectivity_min_eig}. 
Since our estimation theory in Theorem \ref{thm:consistency} does not require these additional assumptions, we seek to design a tuning procedure that also does not.

Recall that while Theorem \ref{thm:consistency} does not require smoothness across $\{B^{(t)}\}$, we assume that the community structure is gradually changing via a Poisson($\gamma$) process where $\gamma/T = o(1)$ (Assumption \ref{ass:alignability}). 
Our theory also demonstrates that changes to the community structure are reflected in the eigenspaces of the probability matrices $\{P^{(t)}\}$.
This inspires our method -- for a particular time $t \in \mathcal{T}$ and choice of bandwidth $r$, we kernel-average the networks earlier than $t$ (i.e., $\{A^{(s)}\}$ for $s < t$) and compute its leading eigenspace.
We then compute the $\sin \theta$ distance (defined below) of this eigenspace from the kernel-average of the networks later than $t$ (i.e., $\{A^{(s)}\}$ for $s > t$). 
A small $\sin \Theta$ distance for an appropriate choice of the bandwidth $\hat{r}$ would be indicative of two aspects, relative to other choices of $r$: 1) the community structure among the networks in $\mathcal{S}(t;\hat{r})\backslash [0,t)$ are not too dissimilar to those in networks in $\mathcal{S}(t;\hat{r})\backslash (t,1]$, and 2) $\hat{r}$ is large enough to produce stably estimated eigenspaces among the networks in  $\mathcal{S}(t;\hat{r})\backslash [0,t)$ or $\mathcal{S}(t;\hat{r})\backslash (t,1]$.
Reflecting on our bias-variance decomposition in \eqref{eq:ideal_decomposition}, the first regards the bias caused by community dynamics, and the second regards the variance due to sparsely observed networks.

Recall that for two orthonormal matrices $U,V \in [-1,1]^{n\times K}$, the $\sin \Theta$ distance (measured via Frobenius norm) is defined as,
\begin{equation} \label{eq:sin_theta}
\big\|\sin \Theta(U, V)\big\|_F = \big(K - \big\|U^\top V\big\|^2_F\big)^{1/2}.
\end{equation}
(See references such as \cite{stewart1990matrix} and \cite{cai2018rate}.)
Formally, our procedure is as follows. Suppose a grid of possible bandwidths $r_1, \ldots, r_m$ are provided, in addition to the observed networks $\{A^{(t)}\}$.
\begin{enumerate}
\item For each bandwidth $r \in \{r_1,\ldots,r_m\}$, compute the score of the bandwidth $\theta(r)$ in the following way.
\begin{enumerate}
\item For each time $t \in \mathcal{T}$, compute the leading eigenspaces of $\sum_{s \in \mathcal{S}}(A^{(s)})^2 - D^{(s)}$, where $\mathcal{S}$ is either $\mathcal{S}(t; c\cdot r)\backslash [0,t)$ or $\mathcal{S}(t;c \cdot r)\backslash (t,1]$ for $\mathcal{S}(t;c \cdot r)$ defined in \eqref{eq:z_box}. Then, compute the $\sin \Theta$ distance between these two eigenspaces via \eqref{eq:sin_theta}, denoted as $\theta(t;r)$.
\item Average $\theta(t;r)$ over $t$. That is, $\theta(r) = \sum_t \theta(t;r)/T$.
\end{enumerate}
\item Choose the optimal bandwidth with the smallest score, i.e., $\hat{r} = \argmin_{r \in \{r_1,\ldots,r_m\}} \theta(r)$.
\end{enumerate}
Observe the presence of a small adjustment factor $c>0$ when deploying the above tuning strategy.
This is to account for the fact the size of the sets $\mathcal{S}(t; c\cdot r)\backslash [0,t)$ and $\mathcal{S}(t;c \cdot r)\backslash (t,1]$ are both roughly $c\cdot rT$, while the usage of $\hat{r}$ in KD-SoS would use $\mathcal{S}(t; \hat{r})$, a set of size roughly $2\cdot\hat{r}T+1$. 
Hence, the adjustment factor $c$ scales the bandwidths when tuning to reflect its performance when used by KD-SoS. We have found $c=2$ to be a reasonable choice in practice.

\subsection{Simulation} \label{sec:simulation}
We provide numerical experiments that demonstrate that our estimator described in Section \ref{ss:algorithm}  is equipped with a tuning procedure which: 1) selects the bandwidth based on data that mimics the oracle that minimizes the Hamming error bound, and 2) improves upon other methods designed to estimate the community structure for the model \eqref{eq:adjacency_matrix}.
Consider $T=50$ networks, each consisting of a network among $n=500$ nodes partitioned into $K=3$ communities. 
The first layer set 200 nodes to the first community, 50 nodes to the second community, and 250 nodes to the third community. 
We describe our simulation setup, which consists of two major components: (1) how nodes change memberships between two time points, and (2) the connectivity matrix at each time point.
First, for each consecutive layer, the nodes switch communities according to the following Markov transition matrix,
\begin{equation} \label{eq:simulation:stationary_markov}
\begin{bmatrix}
1-\gamma & 0 & \gamma\\
0 & 1-\gamma & \gamma \\
\frac{4 \gamma}{5} & \frac{\gamma}{5} & 1-\gamma
\end{bmatrix}.
\end{equation}
Observe that $100\cdot (1-\gamma)$ percent of the nodes change communities between any two consecutive layers in expectation, and for the given initial community partition, this transition matrix ensures that the community sizes are stationary in expectation.
Note that we simulate nodes switching communities via a Markov transition matrix for simplicity.
As alluded to in Remark \ref{remark_cluster_size}, our theorems do not specifically require a Markov transition, and we illustrate KD-SoS on other transition mechanisms in the Appendix.
Second, for a particular time $t$, the connectivity matrix is set to alternate between two possible matrices,
\[
B^{(t)} = \begin{cases}
B^{(\text{odd})} &\quad \text{if } t \cdot T \bmod 2 = 1,\\
B^{(\text{even})} &\quad \text{otherwise}
\end{cases} \quad \text{ for } t \in \mathcal{T},
\]
where 
\begin{equation}\label{eq:simulation:connectivity}
B^{(\text{odd})} = 
\begin{bmatrix}
0.62 & 0.22 & 0.46\\
0.22 & 0.62 & 0.46\\
0.46 & 0.46 & 0.85
\end{bmatrix}, \quad \text{and}\quad
B^{(\text{even})} = 
\begin{bmatrix}
0.22 & 0.62 & 0.46\\
0.62 & 0.22 & 0.46\\
0.46 & 0.46 & 0.85
\end{bmatrix}.
\end{equation}
Then, the observed data is generated according to the model \eqref{eq:adjacency_matrix}, for the desired network density $\rho_n$ (varying between sparse networks with $\rho_n=0.05$ to dense networks with $\rho_n=1$) and the nodes' community switching transition matrix \eqref{eq:simulation:stationary_markov} for a given rate $\gamma$ (varying between stable communities with $\gamma=0$ to rapidly-changing communities with $\gamma=0.1$). 
By considering connectivity matrices $B^{(t)}$ of the form \eqref{eq:simulation:connectivity}, the networks alternate between being either homophilic or heterophilic.
Since not all networks are homophilic in this simulation, certain methods, such as those in \cite{pensky2019spectral}, which we compare against, may not perform well.
We also present simulation settings more favorable to such methods in the Appendix.

\begin{figure}[tb]
  \centering
  \includegraphics[width=400px]{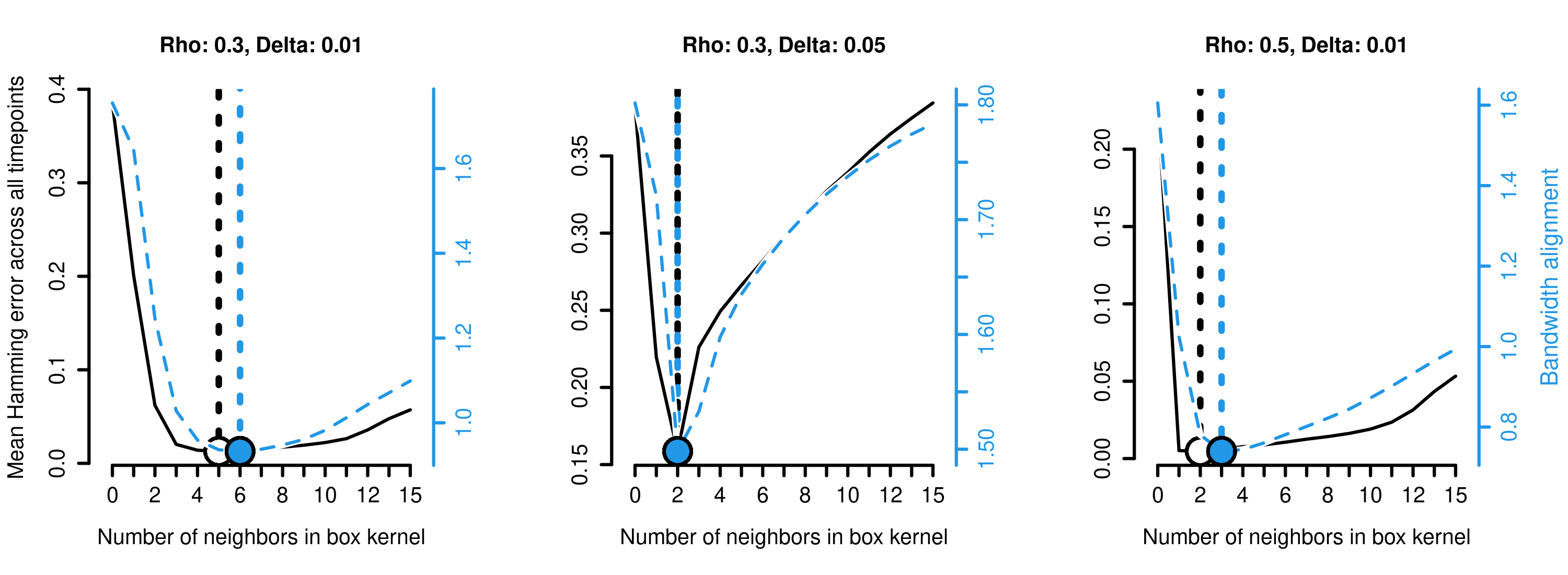}
  \caption
   { 
    Simulation across three different settings of the community switching rate $\gamma$ and network density $\rho_n$, demonstrating KD-SoS's performance for different bandwidths $r$'s.
The Hamming error  \eqref{eq:hamming} or the bandwidth score measured via $\sin \Theta$ \eqref{eq:sin_theta} are averaged across 25 trials for each $r$ (black and blue respectively), and the the vertical dotted lines denote the oracle minimizer of the Hamming error (black) and the chosen bandwidth $\hat{r}$ using the tuning procedure (blue).
 }
    \label{fig:simulation_stationary}
\end{figure}

We first demonstrate that our tuning procedure selects an appropriate bandwidth $r$ of the box kernel, as shown in Figure \ref{fig:simulation_stationary}. 
In the left panel, we fix $\rho_n=0.3$ and $\gamma=0.01$ and plot the mean Hamming error across all networks as a function of applying our estimator with the bandwidth $r$ (black line) and the bandwidth alignment used to tune $r$ (blue line), both averaged across 25 trials. 
A dot of their respective color marks the minimum of both curves.
We make two observations. 
First, the Hamming error follows a classical U-shape as a function of $r$.
This demonstrates that although a single network does not contain information to accurately estimate the communities (i.e., $r=0$), pooling information across too many networks is not ideal either since the community structures vary too much among the networks (i.e., $r=15/50 \approx 0.3$, meaning 15 networks are involved in the box kernel at each side of $t$).
Second, while a bandwidth of $r=5/50$ achieves the minimum Hamming error, our tuning procedure selects $r=6/50$ on average, and the degradation in Hamming error is not substantial.
We also vary $\rho_n$ and $\gamma$. When we set $\gamma$ to 0.05 instead of 0.1, we observe that the bandwidth becomes smaller, indicating that fewer neighboring networks are relevant for estimating a particular network's community structure.
Alternatively, when we set $\rho_n$ to 0.5 instead of 0.3, we observe that the minimized bandwidth becomes smaller. 
However, as implied by the mean Hamming error on the y-axis, this is because more information is contained within each denser network, lessening the need to pool information across networks.

\begin{figure}[tb]
  \centering
  \includegraphics[width=200px]{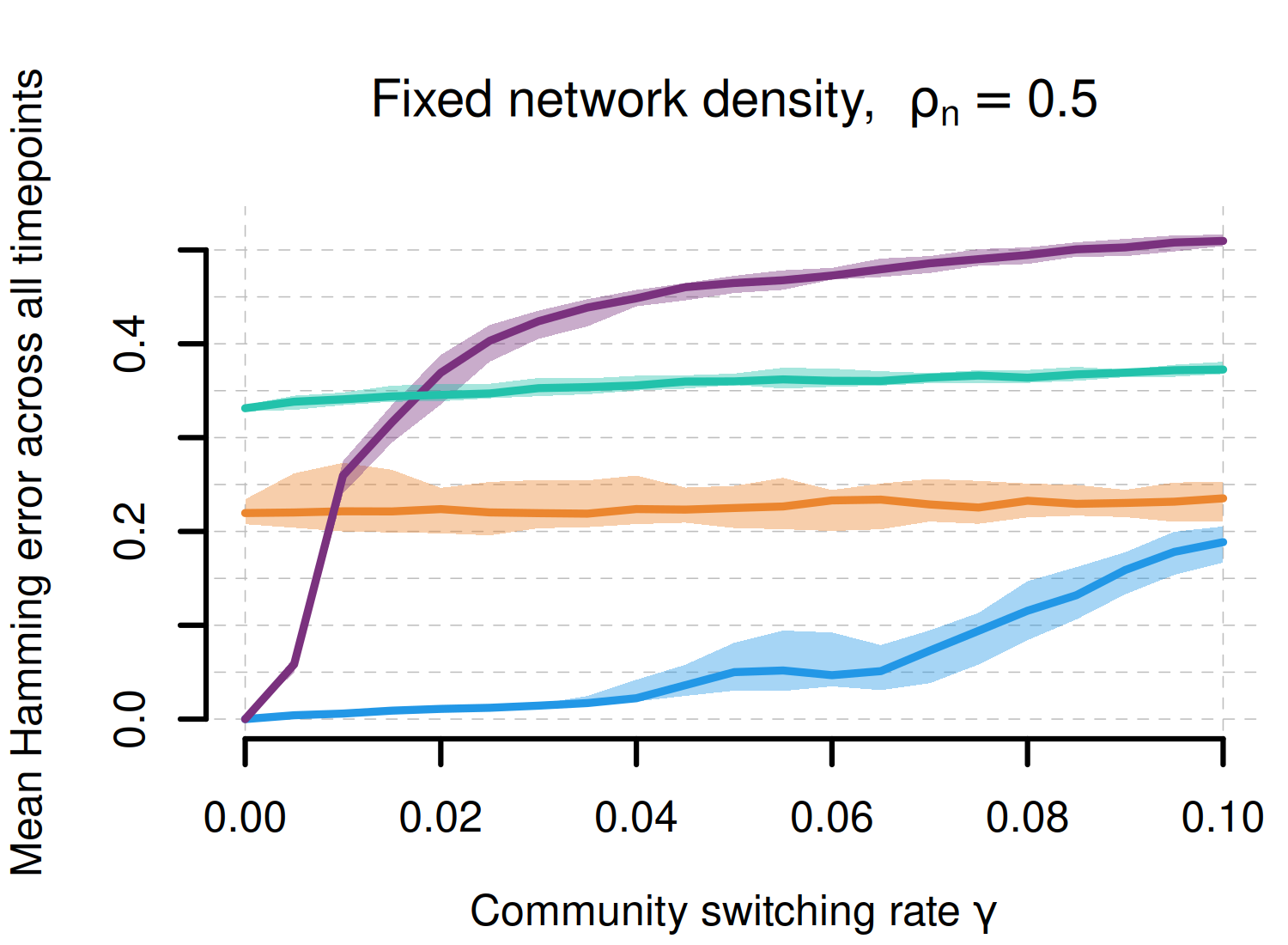}
  \quad
   \includegraphics[width=200px]{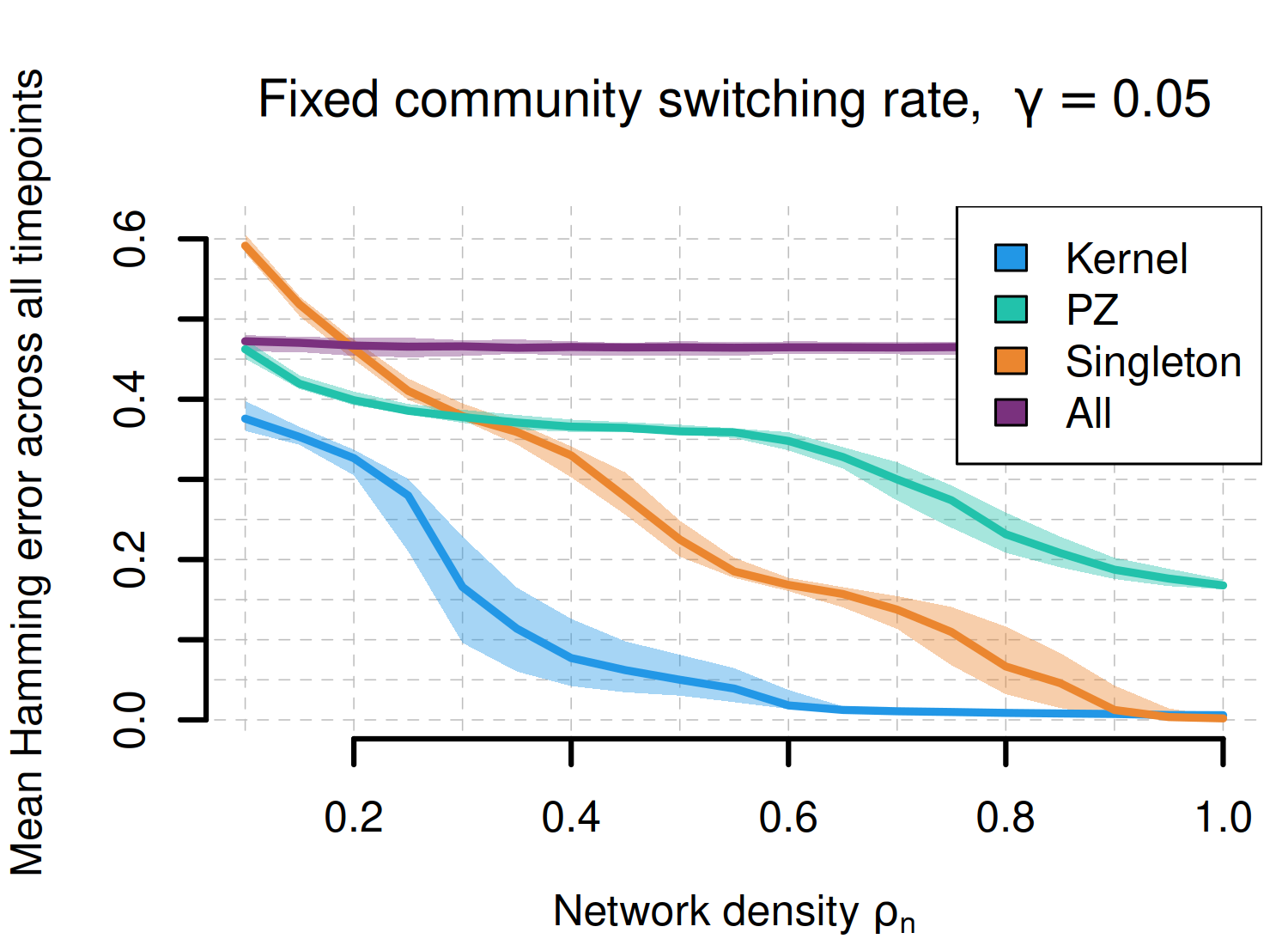}
  \caption
   { 
  Simulation suite across various settings of 
  the community switching rate $\gamma$ (left) or the network density $\rho_n$ (right),
  demonstrating KD-SoS with the bandwidth tuning procedure's performance (``Kernel,'' blue) compared to
  applying spectral clusterings to only one network at a time (``Singleton,'' orange), aggregating across all networks, akin to 
  \cite{lei2020bias} (``All,'' purple), or smooth over a bandwidth of networks without squaring or debiasing the networks, akin to \cite{pensky2019spectral} (``PZ,'' green). The smaller the value on the y-axis is, the better the method performs. The solid lines denote the median over 25 trials, while the bands denote the 10\% to 90\% quantile.
  The simulation setting is more statistically challenging when the community switching rate $\gamma$ is larger or the network density $\rho_n$ is smaller.
   }
    \label{fig:simulation_stationary_all-cleaned}
\end{figure}

We now compare our method against other methods designed to estimate communities for the model \eqref{eq:adjacency_matrix}. Two natural candidates are our debiasing-and-smoothing method where the bandwidth is set to be $r=0$ (i.e., ``Singleton,'' where each network's community is estimated using only that network) and $r=1$ (i.e., ``All,'' where each network's community is estimated by equally weighting all the networks), analogous to \cite{lei2020bias}. We also compare KD-SoS to a method that uses a bandwidth selection procedure to aggregate information across layers by summing the corresponding networks. 
This is analogous to the method proposed by  \cite{pensky2019spectral} (henceforth called the ``PZ'' method).
In this simulation study, we use the same bandwidth selection for KD-SoS and PZ to demonstrate the clear impact of debiasing the sum of squared adjacency matrices.
We measure the performance of each of the three methods by computing the relative Hamming distance between $\hat{M}^{(t)}$ and $M^{(t)}$, averaged across all time $t \in \mathcal{T}$ (i.e., a smaller metric implies better performance).
Our results are shown in Figure \ref{fig:simulation_stationary_all-cleaned}.
In the first simulation suite, we hold the network density $\rho_n=0.5$ but vary the community switching rate $\gamma$ from $0$ to $0.1$ (i.e., from stable communities to rapidly changing communities). Across the 50 trials for each value of $\gamma$, we see that KD-SoS (blue) can retain a small Hamming error below 0.2 across a wide range of $\gamma$.
In contrast, observe that Singleton (orange) exhibits relatively stable performance, which is intuitive since the time-varying structure does not affect this method.
Meanwhile, All (purple) and PZ (green) degrade in performance as $\gamma$ increases due to aggregating among all the networks despite large differences in community structure.
In the second simulation suite, we hold the community switching rate $\gamma = 0.05$ but vary the network density $\rho_n$ from $0.2$ to $1$ (i.e., sparse networks to dense networks). 
Across the 50 trials for each value of $\rho_n$, we see that KD-SoS (blue) performs better as $\rho_n$ increases, which is uniformly better than the Singleton (orange) and PZ (green). 
This is sensible, as KD-SoS with an appropriately chosen bandwidth $r$ aggregates information across networks more effectively than Singleton and PZ.
Meanwhile, all (purple) does not change in performance as $\rho_n$ increases because the time-varying community structure obstructs good performance regardless of network sparsity.

In the Appendix, we present the results of four additional simulations that further demonstrate KD-SoS's performance in other settings. This includes a ``pure'' homophilic setting where methods like PZ can outperform KD-SoS, a setting where $K$ is misspecified, a setting where the Markov transition matrix changes as a function of time $t$, and a setting where the network density changes as a function of $t$.

\subsection{Application to gene co-expression networks along developmental trajectories} \label{sec:application}

We now return to the analysis of the developing brain introduced in Section \ref{sec:introduction}.
We first present descriptive summary statistics for these twelve networks, each comprising the same 993 genes. 
The median of the median degree across all twelve networks is 30.5 (range of 1 to 86, increasing with time), while the mean of the mean degree across all twelve networks is 52.8 (range of 4.6 to 121.9, also increasing with time).
The median overall network sparsity, defined as the number of observed edges divided by the total number of possible edges, across all twelve networks, is 5\% (range: 0.4\% to 12\%, increasing with time). 
Lastly, when analyzed separately, the median number of connected components is 97.5 (range: 34 to 452). 
However, if all the edges across all twelve networks are aggregated, there are two connected components (one with 981 genes, another with 12 genes).

We now present the results obtained by applying KD-SoS to the dataset. 
To encourage a smoother transition between the twelve time points, we use a Gaussian kernel,
i.e.,
\[
Z^{(t;r)} = \sum_{s\in \mathcal{T}} w(s,t;r) \cdot \Big[(A^{(s)})^2 - D^{(s)}\Big],
\quad \text{where }
w(s,t;r) = \exp\Big(\frac{-(t-s)^2}{r^2}\Big)
\]
instead of the aggregation used in \eqref{eq:z_box}. 
Although our theoretical developments in Theorem \ref{thm:consistency} do not use this estimator, our techniques can be applied similarly to such estimators.
Based on a scree plot among $\{A^{(t)}\}$, we chose $K=10$ as the dimensionality and number of communities.
We further document the choice of $K=10$ and the impact of the Gaussian kernel over the box kernel in the Appendix.
The bandwidth is chosen using our procedure in Section \ref{sec:tuning}, among the range of bandwidths $r$ that yielded alignable membership matrices as defined by Definition \ref{def:alignability}.
The membership results for three of the twelve networks are shown in Figure \ref{fig:graphs_labeled}, where nodes of different colors are in different communities. 
Already, we can see gradual shifts in communities within these three networks. 
For example, both the purple and red communities grow in size as time progresses. Meanwhile, genes starting in the olive community eventually become part of the pink or white community.

\begin{figure}[tb]
  \centering
  \includegraphics[width=400px]{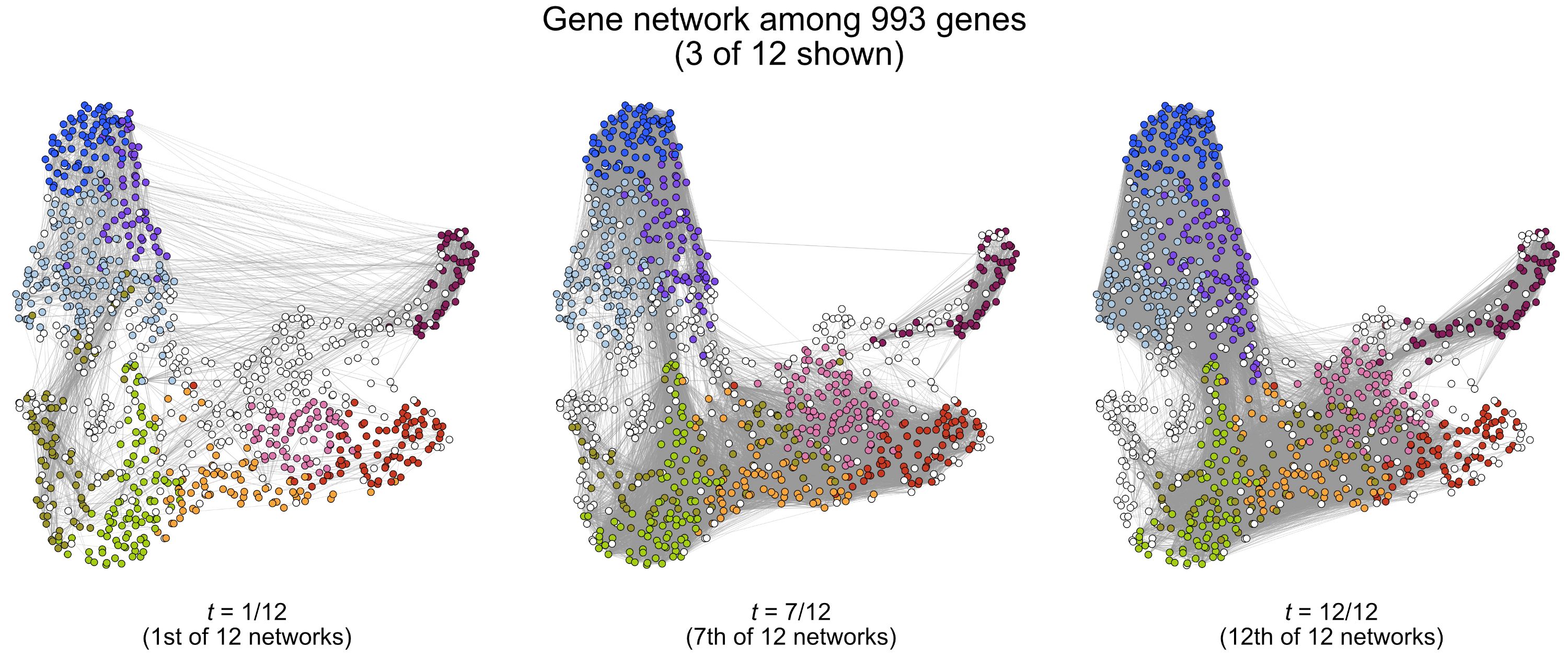}
  \caption
   { 
     Three networks, as displayed in Figure \ref{fig:graphs_unlabeled}, but with genes colored by the $K=10$ different communities
     via $K$ different colors as estimated by KD-SoS and the bandwidth tuning procedure.
   }
    \label{fig:graphs_labeled}
\end{figure}

It is hard to discern the broad summary of how communities are related across time from Figure \ref{fig:graphs_labeled}. 
Hence, we plot the percentage of genes that exit from one community to join a different community between the first three networks in Figure \ref{fig:alluvial}. 
Our tuning bandwidth procedure chooses an $r$ that yields relatively stable communities across time.
Meanwhile, Figure \ref{fig:alluvial} also visualizes the latent 10-dimensional embedding among all 993 genes
for the first three networks. 
We observe that: 1) the SBM model is appropriate for modeling the dataset at hand since the heatmaps demonstrate strong block structure, and 2) a choice of $K=10$ is deemed appropriate via diagnostics based on the scree plot and percent of variance captured. 
Plots demonstrating these diagnostics are shown in the Appendix.
Furthermore, as seen by the visualizations of the latent dimensions and adjacency matrices in the Appendix, none of the 10 communities visually represent sub-communities based on the 10 latent dimensions.

\begin{figure}[tb]
  \centering
  \includegraphics[width=300px]{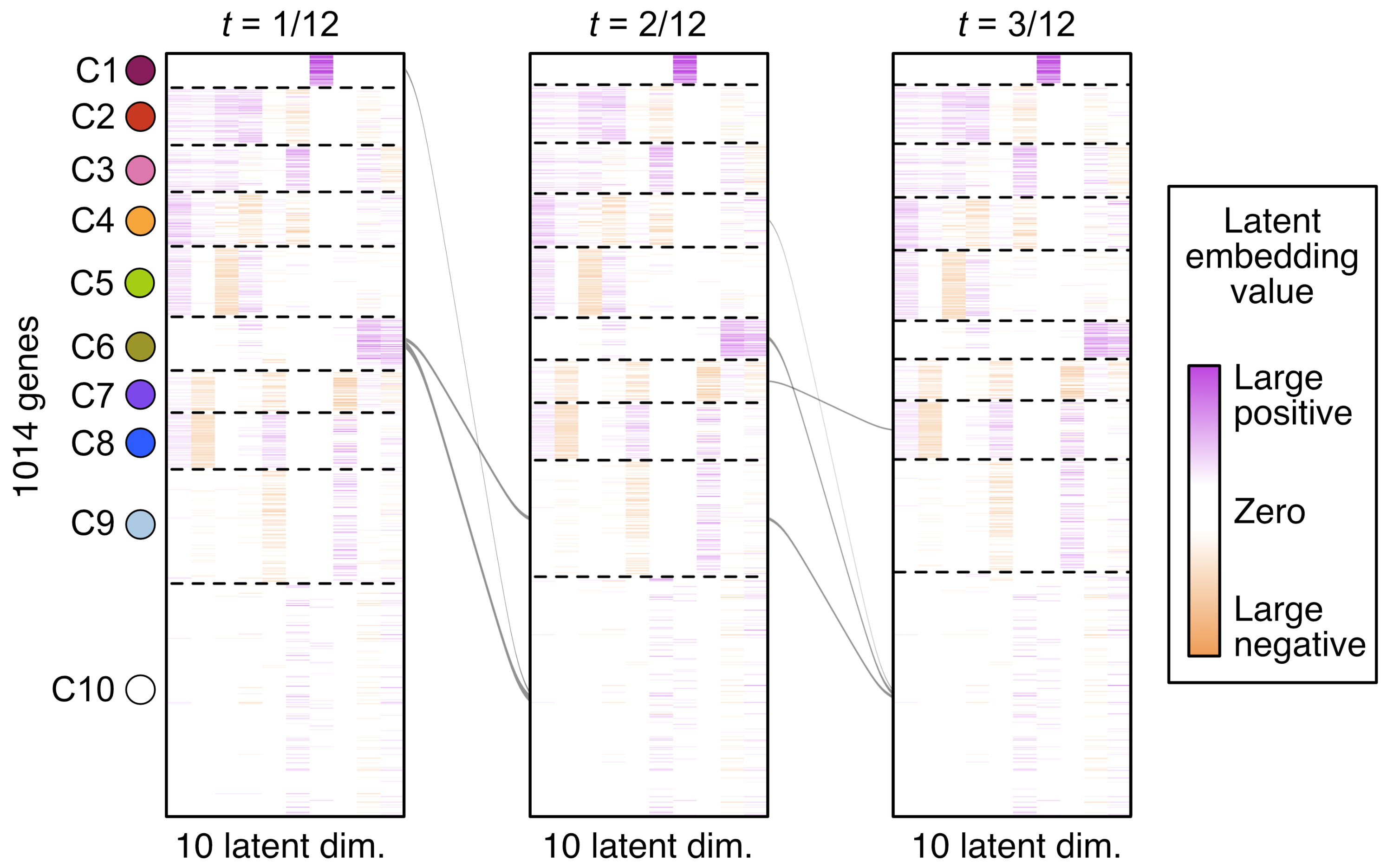}
  \caption
   { 
   The heatmap of the first three networks' leading $K=10$ eigenvectors, where the 993 genes are ordered based on their assigned communities, with their colors (left) corresponding to those in Figure \ref{fig:graphs_labeled}. 
    Let $s,t\in[0,1]$ denote two consecutive two times where $s<t$.
    The size of the arrow connecting two different communities, one at $s$ and another at $t$, denotes the percentage
    of genes that leave the community at time $s$ to a different community at time $t$, ranging from 1\% of the genes
    in the community (thin arrow) to 10\% (thick arrow).
       }
    \label{fig:alluvial}
\end{figure}

Now that we have investigated the appropriateness of the time-varying SBM model, we now address the motivating biological questions asked in Section \ref{sec:introduction} -- what new insights about the glutamatergic development that we could investigate based on the dynamic network structure that we couldn't have inferred based on only analyzing the mean? 
We focus specifically on the fifth and twelfth networks here.
Starting with the fifth network (Figure \ref{fig:mean-correlation}A), we present the enriched Gene Ontology (GO) terms for the selected communities in Table \ref{tab:table1} to investigate the functionality of each set of genes. 
For example, community 2 (red) is highly enriched for coordinated genes related to neurogenesis, despite these genes not yet having high mean expression.
In contrast, community 6 (olive) contains genes related to nervous system development with high gene expression, but these genes are not as well-coordinated.
Meanwhile, community 8 (blue) is highly enriched for coordinated and highly expressed genes related to cellular component biogenesis. 
Likewise, in the twelfth network (Figure \ref{fig:mean-correlation}A and Table \ref{tab:table2}), community 1 (burgundy) is highly enriched for coordinated genes related to cell cycle, despite these genes not yet having high mean expression.
In contrast, community 2 (red) remains highly enriched for genes related to neurogenesis (similar to the fifth network), but these genes are now highly expressed but not coordinated.
Lastly, community 7 (purple) is highly enriched for genes related to the metabolic process that are both coordinated and highly expressed.
Altogether, these results demonstrate that investigating the dynamics of gene coordination can give an alternative perspective on brain development.

\begin{figure}[tb]
  \centering
  \includegraphics[width=400px]{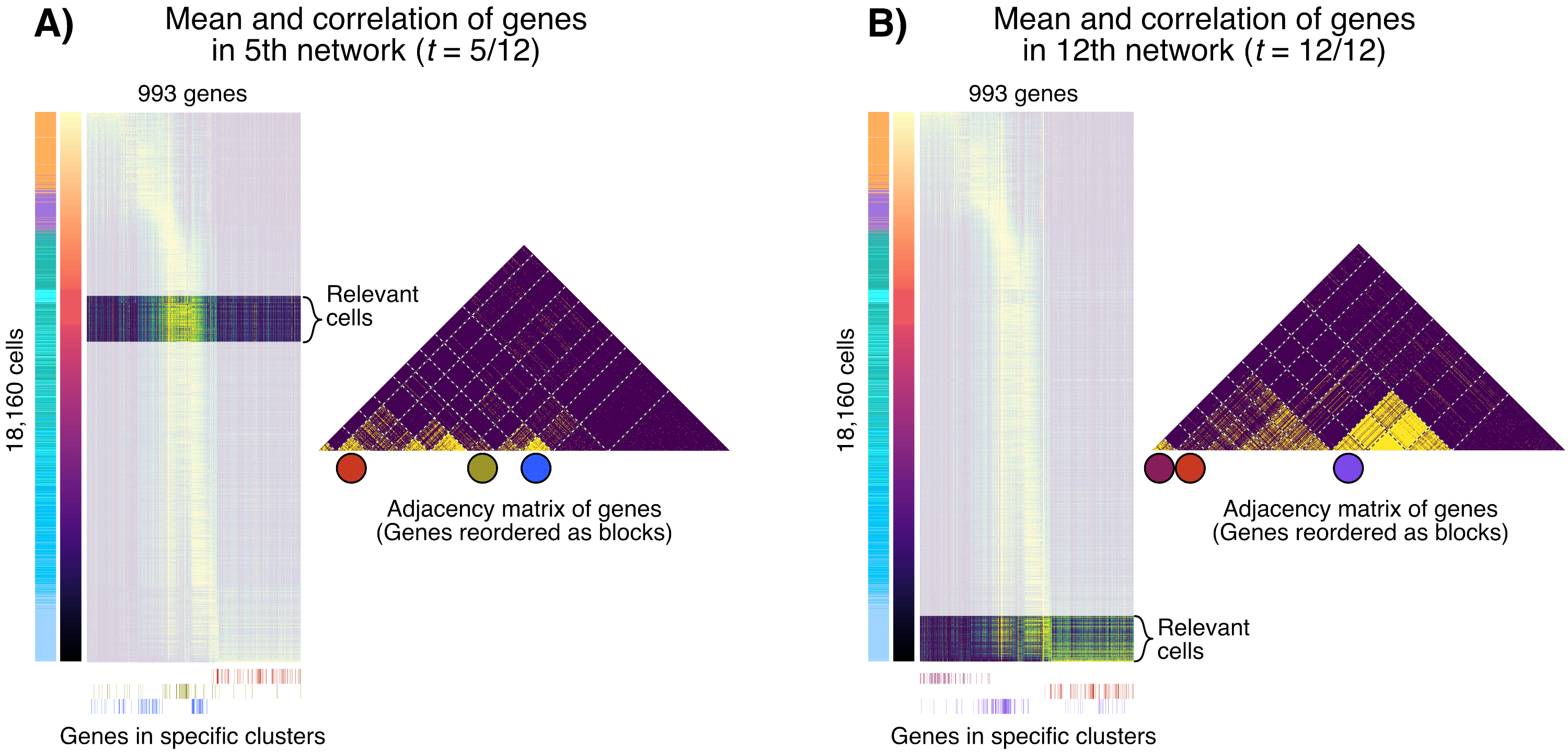}
  \caption
   { 
Correlation networks for the second (A) or twelfth (B) time points, where the cells corresponding to the respective bin of pseudotimes are highlighted via the cell-gene heatmap (left) and the corresponding adjacency matrix among 993 genes where the genes are organized based on their estimated memberships for the respective time point (right). The cell-gene heatmaps are the same as in Figure \ref{fig:cell-umap}. Below, the heatmaps mark the genes (i.e., columns) that are part of specifically highlighted communities, corresponding to the marked entries of the adjacency matrices. 
   }
    \label{fig:mean-correlation}
\end{figure}

\begin{table}
\def~{\hphantom{0}}
\tbl{Description of select gene communities for network $t=5/12$}{%
\begin{tabular}{lcM{0.75in}c|M{0.75in}cc}
 \\
 &\multicolumn{3}{c}{Summary stat.} & \multicolumn{3}{c}{Gene set enrichment} \\ \hline
& \# genes & Mean value (std.) & Connectivity & GO term & \% of community & FDR p-value\\[5pt]
Community 2 (Red) & 92 & 0.06 (0.09) & 0.72 & GO:0022008 (Neurogenesis) &  25\% & $1.81\times 10^{-6}$ \\ \hline
Community 6 (Olive) & 55 & 0.47 (0.36) & 0.16 & GO:0007399 (Nervous system development) & 49\% & $3.83\times 10^{-4}$ \\ \hline
Community 8 (Blue) & 75 & 0.55 (0.27) & 0.88 & GO:0044085 (Cellular component biogenesis) & 39\% & $6.39 \times 10^{-5}$ \\ \hline
\end{tabular}}
\label{tab:table1}
\begin{tabnote}
{Select gene communities for network $t=5/12$, depicting (from left to right) the number of genes in the community, the mean gene expression value and standard deviation among all the cells in this partition (after each gene is standardized across all 18,160 cells), the percent of edges among the genes in the community, an enriched GO term among these genes, the percentage of genes in this community that are in this GO term, and the GO term's FDR value.}
\end{tabnote}
\end{table}

\begin{table}
\def~{\hphantom{0}}
\tbl{Description of select gene communities for network $t=12/12$}{%
\begin{tabular}{lcM{0.75in}c|M{0.75in}cc}
 \\
 &\multicolumn{3}{c}{Summary stat.} & \multicolumn{3}{c}{Gene set enrichment} \\ \hline
& \# genes & Mean value (std.) & Connectivity & GO term & \% of community & FDR p-value\\[5pt]
Community 1 (burgundy) & 56 & 0.01 (0.03) & 0.66 & GO:0007049 (Cell cycle) &  66\% & $1.09\times 10^{-33}$ \\ \hline
Community 2 (Red) & 71 & 0.52 (0.32) & 0.13 & GO:0022008 (Neurogenesis)  & 28\%  & $3.75\times 10^{-5}$ \\ \hline
Community 7 (Purple) & 89 & 0.43 (0.36) & 0.77 & GO: 0008152 (Metabolic process) & 61\% & $8.10\times 10^{-3}$ \\ \hline
\end{tabular}}
\label{tab:table2}
\begin{tabnote}
{Select gene communities for network $t=12/12$, displayed in the same layout as Table 1.}
\end{tabnote}
\end{table}

Additional plots corresponding to networks not shown in Figures \ref{fig:graphs_labeled} through \ref{fig:mean-correlation} 
as well as additional visualizations of the time-varying dynamics are included in the Appendix. 

\section{Discussion}

We establish a bridge between time-varying network analysis and non-parametric analysis in this paper, demonstrating that smoothness across the connectivity matrices $\{B^{(t)}\}$ is not required for consistent community detection. 
We achieve this through a novel bias-variance decomposition, whereby we project networks close to time $t$ onto the leading eigenspace of the network at time $t$. 

While our paper has demonstrated how to relate the discrete changes in nodes' communities to the typically continuous, non-parametric theory, there are four major theoretical directions in which our work can aid future research. 
The first is refining this relation between time-varying networks and non-parametric analyses. 
While previous work for time-varying networks such as \cite{pensky2019spectral} and \cite{keriven2020sparse} derived rates reliant on the smoothness across  $\{B^{(t)}\}$, it is unclear from a minimax perspective how the community estimation rates improve
as $\{B^{(t)}\}$ evolve according to a smoother process. 
Additionally, there have been major historical developments in non-parametric analysis, including the use of local polynomials and trend filtering. These address the so-called boundary bias typical in non-parametric regression and construct estimators that inherently adapt to the data's smoothness. 
We wonder if there are analogies for these estimators for the time-varying SBM setting. 
Secondly, as with any non-parametric estimator, there are unanswered questions about how to tune KD-SoS optimally. 
As we described in Section \ref{sec:tuning}, tuning procedures that rely on prediction, such as cross-validation, are unlikely to be fruitful in the setting we study. 
However, recent ideas using leave-one-out analysis or sharp $\ell_{2\rightarrow \infty}$ estimation bounds for the leading eigenspaces have successfully derived cross-validation-like approaches in other network settings. 
We believe that these ideas can be applied similarly in our setting, where $\{B^{(t)}\}$ is not assumed to be positive definite or smoothly varying. 
Third, we are curious about the optimality of our Hamming estimation bound in this dynamic setting. 
While we are not aware of any optimality results for the setting discussed in this paper, \cite{lei2024computational} discusses the optimal rate from both statistical and computational perspectives for the multi-layer two-community network setting, where community memberships persist across all layers.
Incorporating a smoothness assumption on $\{B^{(t)}\}$ and adjusting KD-SoS's procedure could yield a faster convergence rate.
See \cite{lei2017generic} for a theoretical analysis on how to estimate $B^{(t)}$ itself when incorporating a smoothness assumption.
Lastly, we are interested in determining the optimal $K$ in this dynamic network setting.
While Section \ref{sec:tuning} documents a novel procedure to select the kernel bandwidth, we currently do not have a fully data-driven way to choose the most appropriate number of communities $\hat{K}$. 
Works about goodness-of-fit for a single SBM, such as \cite{chen2018network}, \cite{li2016network}, and \cite{lei2016goodness}, could potentially be extended to our dynamic setting in future work.

\bibliographystyle{biometrika}
\bibliography{bib.bib}

\newpage

\section*{Acknowledgement}
We thank David Choi, Bernie Devlin, and Kathryn Roeder for useful comments when developing this method.  Jing Lei's research is partially supported by NSF grants DMS-2015492 and DMS-2310764.

\section*{Data and code reproducibility}
The human brain development dataset \citep{trevino2021chromatin} was downloaded from \url{https://www.ncbi.nlm.nih.gov/geo/query/acc.cgi?acc=GSE162170}, specifically the \texttt{GSE162170\_rna\_counts.tsv.gz} and \texttt{GSE162170\_rna\_cell\_metadata.txt} files. (Alternatively, the data can also be accessed via \url{https://github.com/GreenleafLab/brainchromatin}.) We use the author's clustering information derived from the Supplementary Information of \cite{trevino2021chromatin}, Table S1 (file: \texttt{1-s2.0-S0092867421009429-mmc1.xlsx}, Sheet F), and genes from Table S1 and Table S3 (files: \texttt{1-s2.0-S0092867421009429-mmc1.xlsx}, Sheet G and \texttt{1-s2.0-S0092867421009429-mmc3.xlsx}, Sheet A).
The code for the KD-SoS as well as all simulations and analyses (including the details on how we preprocessed the single-cell RNA-seq data) is in \url{https://github.com/linnykos/dynamicGraphRoot}.

\section*{Supplementary material}
\label{SM}
In the supplementary materials, we include the pseudocode of KD-SoS,
the proofs of Lemma \ref{lem:decomposition}, 
Theorem \ref{thm:consistency},
Corollary \ref{cor:bandwidth},
Corollary \ref{cor:slow}, Corollary \ref{cor:fast},
Proposition \ref{prop:no-alignability},
Proposition \ref{prop:yes-alignability-deterministic}, and 
Proposition \ref{prop:yes-alignability-probabilistic}.
We also include additional simulations and preprocessing details and more supplemental results in the scRNA-seq analysis from Section \ref{sec:application}.

\appendix

\appendixone

\section{Pseudocode of KD-SoS}

The following provides a high-level pseudocode of our proposed Kernel Debiased Sum-of-Squares (KD-SoS).
\begin{mdframed}
\begin{algorithm}[H] \label{algorithm}
\DontPrintSemicolon
\SetNlSty{textbf}{}{\quad}  
\KwIn{Adjacency matrices $\{A^{(t)}\}_{t=1}^T$; bandwidth $r$; communities $K$.}
\KwOut{Membership matrices $\{\hat{M}^{(t)}\}_{t=1}^T$.}

\For{$t \gets 1$ \KwTo $T$}{
   $\bullet$ $\mathcal{S}(t;r) \gets \{\,s:|t-s|\le rT\}$

   $\bullet$ $Z^{(t)} \gets \sum_{s \in \mathcal{S}(t;r)} \!\bigl[(A^{(s)})^2 - D^{(s)}\bigr]$ 

   $\bullet$ $U^{(t)} \gets \text{top-$K$ eigenvectors of }Z^{(t)}$

   $\bullet$ $\tilde{m}^{(t)} \gets \textsc{KMeans}\bigl(U^{(t)},K\bigr)$ 

   $\bullet$ $\tilde{M}^{(t)} \gets \textsc{OneHot}\bigl(\tilde{m}^{(t)}\bigr)$
}

$\bullet$ $\hat{M}^{(1)} \gets \tilde{M}^{(1)}$

\For{$t \gets 2$ \KwTo $T$}{
   $\bullet$ $C \gets \textsc{Confusion}\bigl(\hat{M}^{(t)}, \tilde{M}^{(t+1)}\bigr)$

   $\bullet$ $R \gets \textsc{Hungarian}(C)$ 

   $\bullet$ $\hat{M}^{(t+1)} \gets \tilde{M}^{(t+1)}\cdot R$ 
}
\Return $\{\hat{M}^{(t)}\}_{t=1}^T$
\end{algorithm}
\end{mdframed}

Here, the \textsc{KMeans} step refers to clustering the rows of $U^{(t)}$ into $K$ clusters via K-means clustering, the \textsc{OneHot} step refers to converting a memberships vector $\tilde{m}^{(t)} \in \{1,\ldots,K\}^n$ into a membership matrix $M \in \{0,1\}^{n \times K}$, the \textsc{Confusion} step refers to computing confusion matrix between the two membership matrices $\hat{M}^{(t)}$ and $\tilde{M}^{(t+1)}$ as in Equation \ref{eq:confusion_population} in the main text, and the \textsc{Hungarian} step refers to computing the optimal permutation of labels for the memberships at time $t+1$ via Hungarian assignment as in Equation \ref{eq:hungarian_assignment_population} in the main text.

\section{Proofs}

\subsection{Proof for bias-variance tradeoff}

\textbf{Proof of Lemma \ref{lem:decomposition}.}
\begin{proof}
The proof is straightforward after observing for any $t\in \mathcal{T}$,
\[
(A^{(t)})^2 = (P^{(t)}+X^{(t)})^2 = (P^{(t)})^2 + P^{(t)}X^{(t)} + X^{(t)}P^{(t)} + (X^{(t)})^2.
\]
and furthermore,
\[
 (P^{(t)})^2 = (Q^{(t)})^2 + \big\{\diag(Q^{(t)})\big\}^2 - 
Q^{(t)}\diag(Q^{(t)}) - \diag(Q^{(t)})Q^{(t)}.
\]

\end{proof}

\vspace{1em}
\textbf{Proof of Theorem \ref{thm:consistency}.}
\begin{proof}
Let $c$ be a constant that can vary from term to term, depending only on the constants $c_1$, $c_2$, $c_3$, $c_\delta$, and $K$. 
Consider the decomposition in Lemma \ref{lem:decomposition}, where we focus on the time $t \in \mathcal{T}$.
We start with the membership bias term (i.e., term $I$). Let $\|\cdot\|_{\op}$ denote the operator norm (i.e., largest singular value). 
For $h =  c\cdot (\gamma r + \log(n)/n)$ for a bandwidth of length $r$, consider the event that 
\begin{align} 
\mathcal{E} 
&= \underbrace{\Big\{
\max_{s \in \mathcal{S}(t;r)} L\big(M^{(s)}, M^{(t)}\big) \leq h
\Big\}}_{\mathcal{E}_1} \bigcap  \underbrace{\Big\{n_k^{(t)} \in \Big[\frac{1}{c K} \cdot n, \frac{c}{K}\cdot n\Big], \quad \text{for all }k \in \{1,\ldots,K\},\;t \in \mathcal{T} \Big\}}_{\mathcal{E}_2}.  \label{eq:membership_change_size_bound}
\end{align}
Lemma \ref{lem:membership_changing_bound} shows that the event $\mathcal{E}_1$ happens with probability at least $1-1/n$, and the event $\mathcal{E}_2$ is controlled by Assumption \ref{ass:cluster_size}. 
Hence, by union bound,
this means event $\mathcal{E}$ happens with probability at least $1-1/n-\epsilon_{c_2,n}$.

The remainder of our analysis will be done in the intersection with event $\mathcal{E}$. 
We start by analyzing the minimum eigenvalue of the target term (i.e., term $V$ in \eqref{eq:ideal_decomposition}).
We define $\tilde{M}^{(t)} = M^{(t)} \big(\Delta^{(t)}\big)^{-1/2}$ as well as 
\begin{equation} \label{eq:definition_tildeb}
\tilde{B}^{(t)} = \big(\Delta^{(t)}\big)^{1/2} B^{(t)} \big(\Delta^{(t)}\big)^{1/2},
\end{equation}
so that $Q^{(t)} = \rho_n \cdot M^{(t)} B^{(t)} (M^{(t)})^\top = \rho_n \cdot \tilde{M}^{(t)} \tilde{B}^{(t)} (\tilde{M}^{(t)})^\top$.
Also recall that the definition of the projection matrix $\Pi^{(t)} = \tilde{M}^{(t)}(\tilde{M})^\top$.
We start with the observation that
\begin{align}
&\Big\|\sum_{s\in \mathcal{S}(t;r)}\Pi^{(t)}(Q^{(s)})^2\Pi^{(t)}\Big\|_{\op}=  \Big\|\sum_{s\in \mathcal{S}(t;r)}\tilde{M}^{(t)}(\tilde{M}^{(t)})^\top(Q^{(s)})^2\tilde{M}^{(t)}(\tilde{M}^{(t)})^\top\Big\|_{\op} \nonumber \\
&= \rho_n^2 \cdot \Big\|\sum_{s\in \mathcal{S}(t;r)}\underbrace{(\tilde{M}^{(t)})^\top
\tilde{M}^{(s)}}_{=U^{(t;s)}} (\tilde{B}^{(s)})^2 (\tilde{M}^{(s)})^\top \tilde{M}^{(t)}\Big\|_{\op}
\nonumber \\
&\overset{(i)}{\geq} \rho_n^2 \cdot \Big\|\sum_{s\in \mathcal{S}(t;r)} \sigma_{\min}^2(U^{(t;s)}) \cdot (\tilde{B}^{(s)})^2\Big\|_{\op} \geq 
\rho_n^2 \cdot \Big\|\Big[\min_{s\in \mathcal{S}(t;r)}\Big\{\sigma^2_{\min}\big(U^{(t;s)}\big)\Big\}\Big] \sum_{s\in \mathcal{S}(t;r)}(\tilde{B}^{(s)})^2\Big\|_{\op} \nonumber\\
&\overset{(ii)}{\geq} (1-ch^{1/2})\cdot  \rho_n^2 \cdot  \Big\|\sum_{s\in \mathcal{S}(t;r)}(\tilde{B}^{(s)})^2\Big\|_{\op}
\overset{(iii)}{\geq} c \cdot (1-ch^{1/2}) \cdot \tilde{T}\rho_n^2 n^2 \label{eq:consistency_1},
\end{align}
where $\tilde{T}=|\mathcal{S}(t;r)| = \min\{2rT+1,T\}$ denotes the number of networks with non-zero weights via the box kernel of bandwidth $r$. 
Here, $(i)$ holds by the variational characterization of eigenvalues (i.e., Rayleigh-Ritz theorem), $(ii)$ holds using Lemma \ref{lem:membership_minsig}, the definition of $h$ under the event $\mathcal{E}$ in \eqref{eq:membership_change_size_bound}, as well as $(1-x)^2 = 1-2x+x^2 \geq 1-2x$ for $x<1$, and $(iii)$ holds via Assumptions
\ref{ass:cluster_size} and \ref{ass:connectivity_min_eig} and the definition of $\tilde{B}$ in \eqref{eq:definition_tildeb}.

We now move to upper-bound relevant terms in \eqref{eq:ideal_decomposition}. Recall that $\sigma_{\min}(A)$ denote the smallest singular value of a matrix $A$. For term $I$, observe that
\begin{align}
&\Big\|(Q^{(s)})^2 - \Pi^{(t)} (Q^{(s)})^2 \Pi^{(t)}\Big\|_{\op} =
\Big\| \Pi^{(s)} (Q^{(s)})^2\Pi^{(s)} -\Pi^{(t)}  (Q^{(s)})^2 \Pi^{(t)}\Big\|_{\op} \nonumber \\
&\overset{(i)}{\leq} \underbrace{\Big(\Big\|\tilde{M}^{(s)}(\tilde{M}^{(s)})^\top\Big\|_{\op} + \Big\|\tilde{M}^{(t)}(\tilde{M}^{(t)})^\top\Big\|_{\op}\Big)}_{=2} \Big\|(Q^{(s)})^2 \Big\|_{\op} \Big\|\tilde{M}^{(s)}(\tilde{M}^{(s)})^\top - \tilde{M}^{(t)}(\tilde{M}^{(t)})^\top\Big\|_{\op} \nonumber \\
&\overset{(ii)}{\leq} c \rho_n^2 n^2 h^{1/2} \label{eq:rate_term_1}
\end{align}
where in $(i)$, we used $ADA^\top - BDB^\top  = ADA^\top  \pm ADB^\top - BDB^\top =
AD(A-B)^\top +(A-B)DB^\top$, and $(ii)$ holds using Lemma \ref{lem:membership_projection_op} and Lemma \ref{lem:probability_op} for some constant $c$ that depends polynomially on $c_2$ and $K$ (recalling the asymptotics in Assumption \ref{ass:cluster_size}).

For the remaining terms (i.e., terms $II$, $III$ and $IV$), since we are considering the regime where $\tilde{T}^{1/2}n\rho_n \geq c_3 \log^{1/2}(\tilde{T}+n)$, we invoke the techniques in Theorem 1 of \cite{lei2020bias}\footnote{
Specifically, 
\eqref{eq:consistency_2}, \eqref{eq:consistency_3}, and \eqref{eq:consistency_4} are analogous to the bound for the term $E_1$, $E_2$, and $E_3$ together with $E_4$ in Theorem 1's proof
in \cite{lei2020bias}, respectively.
},
\begin{align}
&\Big\|\sum_{s\in \mathcal{S}(t;r)} \big[\diag(Q^{(t)})\big]^2 - 
Q^{(t)}\diag(Q^{(t)}) - \diag(Q^{(t)})Q^{(t)}\Big\|_{\op} \leq \tilde{T}n\rho_n^2, \label{eq:consistency_2} \\
&\Big\|\sum_{s\in \mathcal{S}(t;r)}X^{(s)}P^{(s)}+ P^{(s)}X^{(s)} \Big\|_{\op} \leq
c\cdot \tilde{T}^{1/2}n^{3/2}\rho_n^{3/2}\log^{1/2}(\tilde{T}+n), \label{eq:consistency_3}\\
&\Big\|\sum_{s\in \mathcal{S}(t;r)}(X^{(s)})^2- D^{(s)}\Big\|_{\op} \leq \tilde{T}n\rho_n^2 + c\cdot \tilde{T}^{1/2}n\rho_n \log^{1/2}(\tilde{T}+n), \label{eq:consistency_4}
\end{align}
the second and third which hold with probability $1-O((\tilde{T}+n)^{-1})$.

Consider the eigen-decomposition,
\[
\Big\{\sum_{s\in\mathcal{S}(t;r)}\Pi^{(t)}(Q^{(s)})^2\Pi^{(t)}\Big\} = U^{(t;r)} \Lambda^{(t;r)} \big(U^{(t;r)} \big)^\top,
\]
and observe that the eigen-basis of $Q^{(t)}$ is also $U^{(t;r)}$ (i.e., there is a $K \times K$ orthonormal matrix $\Theta$ such that the eigen-basis of $Q^{(t)}$ is equal to $U^{(t;r)} \Theta$, see Lemma 2.1 of \cite{lei2015consistency}.
Recall that $\hat{U}^{(t;r)}$ is the eigen-basis estimated by KD-SoS.
Putting everything together and recalling that the product of two orthonormal matrices yields an orthonormal matrix, 
we see that with an application of Davis-Kahan (see Theorem 2 of \cite{yu2014useful}), there exists a unitary matrix $\hat{O}\in\mathbb{R}^{K\times K}$ such that
\begin{align*}
&\big\|\hat{U}^{(t;r)}\hat{O} - U^{(t;r)}\big\|_F \leq \frac{2^{3/2} K^{1/2} \big\|\big[\sum_{s\in \mathcal{S}(t;r)}(A^{(s)})^2 - D^{(s)}\big] - \big[\sum_{s\in \mathcal{S}(t;r)}\Pi^{(t)}(Q^{(s)})^2\Pi^{(t)}
\big]\big\|_{\op}}{\lambda_{\min}\big(\sum_{s\in \mathcal{S}(t;r)}\Pi^{(t)}(Q^{(s)})^2\Pi^{(t)}\big)}\\
&\overset{(i)}{\leq} c\cdot \frac{h^{1/2}\tilde{T}n^2\rho_n^2  + \tilde{T}n\rho_n^2 + \tilde{T}^{1/2}n^{3/2}\rho_n^{3/2}\log^{1/2}(\tilde{T}+n) + \tilde{T}n\rho_n^2 + \tilde{T}^{1/2}n\rho_n\log^{1/2}(\tilde{T}+n)}{(1-ch^{1/2})_+\cdot \tilde{T}n^2\rho_n^2}\\
&\overset{(ii)}{\leq} \frac{ch^{1/2}}{(1-ch^{1/2})_+} +\frac{2 c}{(1-ch^{1/2})_+\cdot n}+ \frac{c\log^{1/2}(\tilde{T}+n)}{(1-ch^{1/2})_+\cdot \tilde{T}^{1/2}n\rho_n},
\end{align*}
where $(i)$ holds with an application of Lemma \ref{lem:decomposition} as well as Equations \eqref{eq:consistency_1}, \eqref{eq:rate_term_1}, \eqref{eq:consistency_2}, and \eqref{eq:consistency_3},
 and $(ii)$ holds since $n\rho_n \leq c_1$ (due to Assumption \ref{ass:asymptotic}).

Lastly, we wish to convert a Frobenius norm bound between the true and estimated orthonormal matrices into a misclustering error rate.
To do this, from Lemma 2.1 of \cite{lei2015consistency}, we know the minimum Euclidean distance between distinct rows of $U^{(t;r)}$ is at least $c/n^{1/2}$.
Hence, by invoking Lemma D.1 of \cite{lei2020bias} (i.e., a simplification of Lemma 5.3 of \cite{lei2015consistency}),
 the number of misclustered nodes by spectral clustering is no larger than
\[
c\cdot \Big\{\frac{hn}{(1-ch^{1/2})_+^2} + \frac{1}{(1-ch^{1/2})_+^2\cdot n} + \frac{\log(\tilde{T}+n)}{(1-ch^{1/2})_+^2 \cdot \tilde{T}n\rho_n^2}\Big\}.
\]
We divide the above term by $n$ to obtain the percentage of misclustered nodes.
\end{proof}

\vspace{1em}
\textbf{Proof for Corollary \ref{cor:bandwidth}.}
\begin{proof}
Let $c$ be a constant that can vary from term to term, depending only on the constants $c_1$, $c_2$, $c_3$, $c_\delta$, and $K$. 
We seek to derive a the near-optimal bandwidth $r^*$.
Consider the rate in Theorem \ref{thm:consistency}. 
We will only consider the regime where
\[
\gamma r 
\ll 1,
\]
which would mean the leading term in the rate in Theorem \ref{thm:consistency} is upper-bounded by a constant, i.e.,
\[
\frac{1}{\{1-(\gamma r + \log(n)/n)^{1/2}\}^2_+} \ll c.
\]
This allows us to ignore this leading term when deriving the functional form of $r^*$.

Next, observe that if we only want to derive the optimal bandwidth $r^*$ up to logarithmic factors,
we can define
\[
r^* = \min_{r \in [0,1]}\underbrace{c\cdot\gamma r}_{=A(r)} + \underbrace{\frac{\log(T+n)}{rTn^2\rho_n^2}}_{=B(r)}.
\]
Setting the derivative of $A(r)+B(r)$ to be 0 yields,
\[
0 = c\cdot \gamma - \frac{1}{(r^*)^2Tn^2\rho_n^2} \quad \Longrightarrow \quad r^* = c \cdot \frac{1}{(\gamma T)^{1/2}n\rho_n},
\]
for some constant $c$ that depends on $c_1$, $c_2$, $c_3$, $c_\delta$, and $K$.
\end{proof}

\vspace{1em}
\textbf{Proof for Corollary \ref{cor:slow} and Corollary \ref{cor:fast}.}
\begin{proof}
The upper-bound of the relative Hamming distance depends on if $r^* \rightarrow 1$ or $r^* \rightarrow 0$ based on the asymptotic sequence of $n$, $T$, $\gamma$ and $\rho_n$.
Recall that by assumptions in Theorem \ref{thm:consistency}, we require 
\begin{align}
 (rT+1)^{1/2}n\rho_n = \omega\Big\{ \log^{1/2}(rT+n+1)\Big\}. \label{eq:consistency_requirement1}
\end{align}

\begin{itemize}
\item Based on Corollary \ref{cor:bandwidth}, the scenario $r^* \rightarrow 1$ occurs if 
\[
\frac{1}{(\gamma T)^{1/2}n\rho_n} \rightarrow \infty \quad \iff \quad (\gamma T)^{1/2}n\rho_n \rightarrow 0.
\]

We also require that $\gamma r^* \rightarrow 0$ as a necessary condition for
the relative Hamming distance in Theorem \ref{thm:consistency} to converge to 0.
To ensure this, we will require asymptotically
\begin{equation}
\gamma \rightarrow 0.
 \label{eq:consistency_requirement2a}
\end{equation}

Furthremore, the requirement \eqref{eq:consistency_requirement1} is satisfied if
\begin{equation}  \label{eq:consistency_requirement1_bandwidth1}
T^{1/2}n\rho_n = \omega \big\{\log^{1/2}(T+n)\big\}.
\end{equation}

To upper-bound the relative Hamming error, since $\gamma r^* \rightarrow 0$,
for any constant $c$, this means somewhere along this asymptotic sequence of $\{n,T,\gamma,\rho_n\}$, we are guaranteed $\gamma r + \log(n)/n \leq c$ for the remainder of the asymptotic sequence. Then,
\[
L\big(M^{(t)}, \hat{M}^{(t)}\big) =  O\Big\{\gamma + \frac{\log(n)}{n} + \frac{1}{n^2} + \frac{\log(T+n)}{Tn^2\rho^2_n}\Big\},
\]
By \eqref{eq:consistency_requirement2a} and \eqref{eq:consistency_requirement1_bandwidth1}, we are ensured that $L\big(M^{(t)}, \hat{M}^{(t)}\big)$ converges to 0.

\vspace{1em}
\item Based on Corollary \ref{cor:bandwidth}, the scenario $r^* \rightarrow 0$ occurs if 
\begin{equation}
\frac{1}{(\gamma T)^{1/2}n\rho_n} \rightarrow 0 \quad \iff \quad (\gamma T)^{1/2}n\rho_n \rightarrow \infty.
 \label{eq:consistency_requirement2b_part1}
\end{equation}

We also require that $\gamma r^* \rightarrow 0$ as a necessary condition for
the relative Hamming distance  in Theorem \ref{thm:consistency}  to converge to 0.
To ensure this, using the rate of $r^*$ derived in Corollary \ref{cor:bandwidth}, we require asymptotically
\begin{equation}
\quad \gamma r^* = \frac{\gamma^{1/2}}{T^{1/2}n\rho_n} \rightarrow 0 \quad \iff \quad \gamma = o\big\{T(n\rho_n)^2\big\},
 \label{eq:consistency_requirement2b_part2}
\end{equation}
which upper-bounds the maximum $\gamma$ before KD-SoS is no longer consistent.
Furthermore, the requirement \eqref{eq:consistency_requirement1} is satisfied based on the bandwidth $r^*$ in Corollary \ref{cor:bandwidth} if
\begin{equation} \label{eq:consistency_requirement2b_part3}
\Big(\frac{T}{\gamma}\Big)^{1/2}n\rho_n = \omega\big( \log^{1/2}(T+n)\big).
\end{equation}

An asymptotic regime that would satisfy \eqref{eq:consistency_requirement2b_part1}, \eqref{eq:consistency_requirement2b_part2}, and \eqref{eq:consistency_requirement2b_part3} is 
\begin{equation} \label{eq:consistency_requirement2b_part4}
\gamma \text{ is increasing and }
\gamma = o\Big\{\frac{T(n\rho_n)^2}{\log(T+n)}\Big\}.
\end{equation}

To upper-bound the relative Hamming error, since $\gamma r^* \rightarrow 0$,
for any constant $c$, this means somewhere along this asymptotic sequence of $\{n,T,\gamma,\rho_n\}$, we are guaranteed $\gamma r + \log(n)/n \leq c$ for the remainder of the asymptotic sequence. Then,
\begin{align*}
L\big(M^{(t)}, \hat{M}^{(t)}\big) = O\Big\{\frac{\gamma^{1/2}}{T^{1/2}n\rho_n} + \frac{\log(n)}{n} + \frac{1}{n^2} + \frac{\gamma^{1/2}\log(T^{1/2}/(\gamma^{1/2}n\rho_n)+n)}{T^{1/2}n\rho_n}\Big\}.
\end{align*}
By \eqref{eq:consistency_requirement2b_part4}, we are ensured that $L\big(M^{(t)}, \hat{M}^{(t)}\big)$ converges to 0.

The probability that the bound for $L\big(M^{(t)}, \hat{M}^{(t)}\big)$ holds is $1-O\{(rT + n)^{-1}\} - \epsilon_{c_2,n}$. Using the optimal bandwidth $r^*=c/(\gamma^{1/2}T^{1/2}n\rho_n)$ in Corollary \ref{cor:bandwidth} and $\gamma = o\{T(n\rho_n)^2/\log(T+n)\}$ as assumed in the asymptotic regime, we derive that
\[
r^* = \omega\Big\{\frac{\log^{1/2}(T+n)}{Tn^2\rho_n^2}\Big\}.
\]
This means the probability that the bound for $L\big(M^{(t)}, \hat{M}^{(t)}\big)$ holds is
\[
1-O\Big\{\Big(\frac{\log^{1/2}(T+n)}{n^2\rho_n^2} + n\Big)^{-1}\Big\} - \epsilon_{c_2,n}.
\]
\end{itemize}

Hence, we are done.
\end{proof}

\vspace{1em}
\textbf{Proof of Proposition \ref{prop:no-alignability}.}
\begin{proof}
We split the proof into two parts. 

\textbf{Deterministic component.}
Here, we prove if more than $n/2$ nodes change memberships between $M^{(t)}$ and $M^{(t+1/T)}$ for a particular $t \in \mathcal{T}\backslash \{1\}$, then $M^{(t)}$ and $M^{(t+1/T)}$ are not alignable. Then, by definition, the entire sequence of memberships is not alignable. 

Consider the confusion matrix $C \in \{0,\ldots,n\}^{K\times K}$ formed from $M^{(t)}$ and $M^{(t+1/T)}$.
Since more than $n/2$ nodes change memberships, then by definition, 
the sum of the off-diagonal entries in $C$ must be larger than $n/2$, and the sum of the diagonal entries in $C$ must be smaller than $n/2$. Hence, there must exist a diagonal entry in $C$ whereby it is smaller than its respective column-sum or row-sum. Hence, it must be the case that either $C$ or $C^\top$ is not diagonally dominant, and hence, $M^{(t)}$ and $M^{(t+1/T)}$ is not alignable.

\vspace{.5em}
\textbf{Probabilistic component.}
Here, we prove that if $\gamma$ is large relative to $T$, then there is a non-vanishing probability that more than $n/2$ nodes change memberships  between $M^{(t)}$ and $M^{(t+1/T)}$ for some time $t \in \mathcal{T}\backslash \{1\}$.

Towards this end, let $X^{(t)}$ denote the total number of instances when nodes change communities between time $t$ and $t+1/T$ based on Assumption \ref{ass:poisson}. (Note, this random variable is not a Poisson, since the Poisson process denotes the number of instances a node changes membership, not the number of unique nodes change membership.)
We are interested in when the probability $X^{(t)} \geq n/2$ for some $t\in\{1/T,\ldots,(T-1)/T\}$ is bounded away from 0.
That is,
\begin{align}
&\mathbb{P}\Big(X^{(t)} \geq n/2, \text{ for some } t \in\{1/T,\ldots,(T-1)/T\}\Big)\nonumber \\
&=1-\mathbb{P}\Big(X^{(t)} \leq n/2, \text{ for all } t \in\{1/T,\ldots,(T-1)/T\}\Big) \nonumber \\
&= 1-\mathbb{P}\Big(X^{(1/T)} \leq n/2\Big)^{T-1}=1-\Big\{1-\mathbb{P}\Big(X^{(1/T)} \geq n/2\Big)\Big\}^{T-1} \label{eq:information_eq1}
\end{align}

To lower-bound the RHS of \eqref{eq:information_eq1}, consider a probability $p$ that a node changes membership in a time interval of length $1/T$. Since each node changes memberships independently of one another, the total number of nodes that change memberships is modeled as $X^{(1/T)} = \text{Binomial}(n, p)$ for a $p$ to be determined, and we are interested the probability that $X^{(1/T)} \geq n/2$. Certainly, if $p = 1/2$, then the probability of $X^{(1/T)} \geq n/2$ is strictly bounded away from 0. Hence, we are interested in a $p$ less than $1/2$. 

Towards this end, invoking a lower-bound of the upper-tail of a Binomial (see Chernoff-Hoeffding bounds in references such as \cite{pelekis2016lower}), observe that 
\begin{equation} \label{eq:binomial_lower-bound}
\mathbb{P}\big( X^{(1/T)} \geq n/2\big) \geq \frac{1}{(2n)^{1/2}} \exp\Big\{-n D\big(\frac{1}{2} \;||\; p\big)\Big\},
\end{equation}
where 
\begin{align}
D\big(\frac{1}{2} \;||\; p\big) &= \frac{1}{2} \cdot \log \Big(\frac{1/2}{p}\Big) +  \frac{1}{2} \cdot \log \Big(\frac{1/2}{1-p}\Big) \nonumber \\
&= \frac{-1}{2} \cdot \log \big(2\cdot p\big) +  \frac{-1}{2} \cdot \log \big\{2\cdot(1-p)\big\} \nonumber \\
&= \log \Big[\big\{4 \cdot p \cdot (1-p)\big\}^{-1/2}\Big]\label{eq:bernoulli_divergence}.
\end{align}

For reasons we will shortly discuss, we are interested when \eqref{eq:binomial_lower-bound} is lower-bounded by $1/(T-1)$. Hence, combining \eqref{eq:binomial_lower-bound} with  \eqref{eq:bernoulli_divergence}, we are interested in $x$ such that
\begin{equation} \label{eq:binomial_lower-bound2}
\mathbb{P}\big(X^{(0)} \geq n/2\big) \geq 
\frac{1}{(2n)^{1/2}} \cdot \Big[\big\{4 \cdot p \cdot (1-p)\big\}^{n/2}\Big] \geq \frac{1}{T-1},
\end{equation}
which is equivalent to
\begin{equation} \label{eq:binomial_lower-bound3}
p \cdot (1-p) \geq \frac{1}{4}\cdot \Big\{\frac{(2n)^{1/2}}{T-1}\Big\}^{2/n}.
\end{equation}
Observe that if we assume that $p\leq 1/2$, then a value of $p$ that satisfies 
\begin{equation} \label{eq:binomial_lower-bound4}
p^2 \geq \frac{1}{4}\cdot \Big\{\frac{(2n)^{1/2}}{T-1}\Big\}^{2/n} 
\quad \iff \quad 
p \geq \frac{1}{2}\cdot \Big\{\frac{(2n)^{1/2}}{T-1}\Big\}^{1/n} 
\end{equation}
is ensured to satisfy \eqref{eq:binomial_lower-bound3}.

This means if $1/2 \cdot ((2n)^{1/2}/(T-1))^{1/n} \leq p \leq 1/2$, then there is at least probability $1/(T-1)$ that $X^{(1/T)} \geq n/2$.
Therefore, using this value of $p$, we infer from \eqref{eq:binomial_lower-bound2} that
\begin{equation} \label{eq:information_eq3}
\Big\{1-\mathbb{P}\Big(X^{(1/T)} \geq n/2\Big)\Big\}^{T-1}  \leq \big(1-\frac{1}{T-1}\big)^{T-1} \overset{(i)}{\leq}1/e \approx 0.37,
\end{equation}
where $(i)$ uses $\lim_{x\rightarrow \infty}(1-1/x)^x = 1/e$ from below.
Plugging \eqref{eq:information_eq3} back into \eqref{eq:information_eq1} shows 
for probability $p$ that a node changes membership within any time interval of length $1/T$, then
for any $T \geq 2$, 
\[
\mathbb{P}\Big(X^{(t)} \geq n/2, \text{ for some } t \in \{1/T, \ldots, (T-1)/T\}\Big) \geq 1- 1/e \approx 0.63.
\]

Lastly, we are now interested in the relation between $\gamma$ and $T$ such that there is at least a probability $p$ of a node changing memberships in a time interval of length $1/T$. By the Poisson process in Assumption \ref{ass:poisson}, the probability a node changes membership in such an interval is
\begin{align*}
1 - \exp(-\gamma/T) &\geq p =  \frac{1}{2}\cdot \Big(\frac{(2n)^{1/2}}{T-1}\Big)^{1/n} \\
\Longrightarrow \quad  \gamma&\geq T \cdot \log\bigg[\bigg\{1- \frac{1}{2}\cdot 
\Big(
\frac{2^{1/2}n^{1/2}}{T-1}
\Big)^{1/n}\bigg\}^{-1}\bigg]
\end{align*}
Hence, we are done. 
\end{proof}

\vspace{1em}
\textbf{Proof of Proposition \ref{prop:yes-alignability-deterministic}.}
\begin{proof}
Consider a particular time $t \in \mathcal{T}\backslash \{1\}$.
For any time $t$ and $t+1/T$, consider the confusion matrix $C^{(t,t+1/T)}$ formed between membership matrices
$M^{(t)}$ and $M^{(t+1/T)}$. Let $C = C^{(t,t+1/T)}$ for notational simplicity.
Let $m_{\min}$ denote the size of the smallest community at time $t$,
\[
m_{\min} = \min_{k \in \{1,\ldots,K\}} \sum_{i=1}^{n}M_{ik}^{(t)} = \min_{k \in \{1,\ldots,K\}} \sum_{\ell = 1}^{K} C_{k\ell}.
\]
Consider any community $k \in \{1,\ldots,K\}$. We first compare $C_{kk}$ to the sum of all the other elements in the row (i.e., the number of nodes that leave community $k$ between time $t$ and $t+1/T$). Let $z = \sum_{\ell:k\neq \ell}C_{k \ell}$. Since $C_{kk}+z$ equals the number of nodes in community $k$ at time $t$, and that the number of nodes that change is at most $m_{\min}/2$, we know
\[
m_{\min} \leq C_{kk}+z \quad \text{and}\quad z \leq m_{\min}/2 \quad \Rightarrow \quad C_{kk} \geq z.
\]

Next, we compare $C_{kk}$ to the sum of all the other elements in the column (i.e., the number of nodes that enter community $k$ between time $t$ and $t+1/T$). Let $y = \sum_{\ell:k\neq \ell}C_{\ell k}$. Since $C_{kk}+z$ equals the number of nodes in community $k$ at time $t$, and the number of nodes that change total is less than $m_{\min}/2$, we know
\[
m_{\min} \leq C_{kk}+z \quad \text{and}\quad z+y \leq m_{\min}/2  \quad \Rightarrow \quad C_{kk} \geq m_{\min}/2 + y \geq y,
\]
which completes the proof.
\end{proof}

Note that the above proof works for any number of communities, not necessarily only when $K=2$.

\vspace{1em}
\textbf{Proof of Proposition \ref{prop:yes-alignability-probabilistic}.}
\begin{proof}

Let $x^{(t)} = \|M^{(t)} - M^{(t+1/T)}\|_0$ and $y^{(t)} = \min_{k \in \{1,\ldots,K\}} \sum_{i=1}^{n}M^{(t)}_{ik}$. %Focusing on a specific $t \in \mathcal{T}$, o
Observe that we have the following relation in events,
\[
\Big\{x^{(t)} \geq y^{(t)}, \;\; \text{for some time }t\in \mathcal{T} \Big\}\Longrightarrow \underbrace{\Big\{x^{(t)} \geq \Delta, \;\; \text{for some time }t\in \mathcal{T}\Big\}}_{\mathcal{E}_1} \cup \underbrace{\Big\{\Delta \geq y^{(t)}, \;\; \text{for some time }t\in \mathcal{T}\Big\}}_{\mathcal{E}_2}.
\]
for any constant $\Delta > 0$. 
Hence, we wish to upper-bound the following undesirable event via a union bound,
\begin{equation} \label{eq:no-alignability1}
\mathbb{P}\Big(x^{(t)} \geq y^{(t)},\;\; \text{for some time }t\in \mathcal{T}\Big) \leq 
\mathbb{P}\big(\mathcal{E}_1 \big) +
\mathbb{P}\big(\mathcal{E}_2\big).
\end{equation}

We invoke Lemma \ref{lem:recursive_decomposition} to first upper-bound $\mathbb{P}(\mathcal{E}_2)$ by $1/T$ via a recursive decomposition and to pick the appropriate threshold $\Delta$, specifically,
\[
\Delta = \frac{n}{2} - c\cdot\max\big[\{n\gamma \log (T)\}^{1/2}, \log(T) \big]
\] 
for some universal constant $c$. (By Assumption \ref{ass:asymptotic} and  $\gamma/T=o(1)$, we are assured that $\max[\{n\gamma \log (T)\}^{1/2}, \log(T)] \ll n$.)
Using this threshold $\Delta$, we then invoke Lemma \ref{lem:tail_binomial} which shows that
\[
\mathbb{P}\Big\{x^{(t)} > \frac{5n\gamma}{T} + 4\log(T)\Big\}  \leq 1/T^2, \quad \text{for a particular time } t \in \mathcal{T}.
\]
Since Assumption \ref{ass:asymptotic} and  $\gamma/T=o(1)$ ensure that $5n\gamma/T + 4\log(T) \ll \Delta$, 
we have upper-bound shown $\mathbb{P}(\mathcal{E}_1) < 1/T$ via a union bound.
Therefore, altogether, we obtain the desired upper-bound when plugging these bounds into \eqref{eq:no-alignability1},
\[
\mathbb{P}\Big(\big\|M^{(t)} - M^{(t+1/T)}\big\|_0 \geq \min_{k\in\{1,\ldots,K\}}\sum_{i=1}^{n}M_{ik}^{(t)}, \text{ for some time }t \in \mathcal{T}\Big) \leq \frac{2}{T},
\]
or equivalently,
\[
\mathbb{P}\Big(\big\|M^{(t)} - M^{(t+1/T)}\big\|_0 < \min_{k\in\{1,\ldots,K\}}\sum_{i=1}^{n}M_{ik}^{(t)}, \text{ for all time }t \in \mathcal{T}\Big) \geq 1-\frac{2}{T},
\]
and complete the proof.
\end{proof}

\subsection{Helper lemmata}

We aim to probabilistically bound the relative Hamming distance between two membership matrices given the dynamics stated in Section \ref{sec:model}.

\begin{lemma}  \label{lem:membership_changing_bound}
Given the model in Section \ref{sec:model},
consider a particular $t, r \in [0,1]$.
Letting $\delta = \min\{t+r, 1\} - \max\{t-r,0\}$, then 
\begin{equation} \label{eq:membership_changing_bound}
\mathbb{P}\Big\{
\max_{s \in \mathcal{S}(t;r)} L\big(M^{(s)}, M^{(t)}\big) \geq 4\gamma\delta + \frac{3\log(n)}{n}
\Big\} \leq \frac{1}{n}
\end{equation}
for some universal constant $c$.
\end{lemma}
\begin{proof}
Let $t_-  = \min \mathcal{S}(t;r)$, $t_+ =  \max \mathcal{S}(t;r)$ and choose any $t', t'' \in \mathcal{S}(t;r)$ where $0 \leq t_- \leq t' \leq t'' \leq t_+ \leq 1$. 
For an $\tau > 0$ to be determined, consider the four events,
\begin{align*}
\mathcal{E}_1 &= \Big\{n \cdot L(M^{(t')}, M^{(t'')}) \geq n\gamma\delta + \tau \Big\},\\
\mathcal{E}_2 &= \Big\{\big(\text{\# of nodes that changed communities anytime between }t'\text{ and }t''\big) \geq n\gamma\delta +\tau\Big\},\\
\mathcal{E}_3 &= \Big\{\big(\text{\# of nodes that changed communities anytime between }t_-\text{ and }t_+\big) \geq n\gamma\delta +\tau\Big\},\\
\mathcal{E}_4 &= \Big\{\sum_{s=t_-}^{t_+} n \cdot L(M^{(s)}, M^{(s+1/T)})
\geq n\gamma\delta +\tau\Big\}.
\end{align*}
Observe that for simultaneously over such choice of $t'$ and $t''$, $\mathcal{E}_1 \Rightarrow
\mathcal{E}_2 \Rightarrow \mathcal{E}_3 \Rightarrow \mathcal{E}_4$, where the last event models the number of nodes that change communities between any two consecutive time points in $\mathcal{S}(t;r)$. 
Hence $\mathbb{P}(\mathcal{E}_1) \leq \mathbb{P}(\mathcal{E}_4)$, which implies that 
\begin{equation} \label{eq:lemmata:poisson_1}
\mathbb{P}\Big\{\max_{s \in \mathcal{S}(t;r)} L\big(M^{(s)}, M^{(t)}\big) \geq \gamma\delta + \tau/n\Big\} \leq 
\mathbb{P}\big(\mathcal{E}_4 \big).
\end{equation}
Hence, we focus on the upper-bounding the RHS. 

Let $\tilde{T} = |\mathcal{S}(t;r)| = \delta \cdot T$, i.e., the number of summands in the summation on the LHS of $\mathcal{E}_4$. This is also the number of non-overlapping intervals of length $1/T$ (plus one) that fit between $t_-$ and $t_+$.
Observe that since the nodes change communities according to a Poisson($\gamma$) process independently of one another, the probability a node changes communities in a time interval of $1/T$ is $1-\exp(-\gamma/T)$.
Consider two Binomial random variables $X$ and $Y$ defined as
\begin{align*}
X &\sim \text{Bernoulli}\big(n\cdot \tilde{T}, \; 1-\exp(-\gamma/T)\big)\\
Y &\sim \text{Bernoulli}\big(n\cdot \tilde{T}, \; \max\big\{\gamma/T,1\big\}\big),
\end{align*} 
which represents the number of success among $n \cdot \tilde{T}$ trials each with a probability $1-\exp(-\gamma/T)$ or  $\{\gamma/T,1\}$ of success respectively.
(Here, a ``success'' represents a node changing communities within a time interval of length $1/T$.)
Recalling that $\exp(-x) \geq 1-x$ and that $\delta = \tilde{T}/T$ by definition,
observe,
\begin{equation} \label{eq:lemmata:poisson_2}
\mathbb{P}\big(\mathcal{E}_4  \big) = \mathbb{P}\big(X \geq  n\gamma\delta + \tau\big) \leq \mathbb{P}\big(Y \geq  n\gamma\delta + \tau\big).
\end{equation}

Continuing, keeping in mind that $\mathbb{E}(X) \leq \mathbb{E}(Y) \leq n\gamma\delta$, we derive\footnote{Observe: if $\gamma/T > 1$, then $\mathbb{P}(Y \geq n\gamma\delta + \tau) = 0$
since the maximum value of $Y$ is $n\tilde{T}$, whereas $n\gamma\delta = n\gamma \tilde{T}/T > n \tilde{T}$.}
\begin{align}
\mathbb{P}\big(Y \geq  n\gamma\delta + \tau\big) &\overset{(i)}{\leq} 
\exp\Big\{
\frac{-\frac{1}{2}\tau^2}{\frac{\gamma}{T} \cdot (1-\frac{\gamma}{T} ) \cdot n \cdot \tilde{T} + \frac{1}{3}\tau}
\Big\} \nonumber \\
&\leq
\exp\Big\{
\frac{-\frac{1}{2}\tau^2}{n\gamma\delta + \frac{1}{3}\tau}
\Big\}  \label{eq:membership_changing_bound_1}
\end{align}
where $(i)$ holds via Bernstein's inequality (for example, Lemma 4.1.9 from \cite{de2012decoupling}).

Consider 
$\tau = 3n\gamma\delta + 3 \log(n)$. If $\log(n) > n\gamma\delta$, then
we have from \eqref{eq:membership_changing_bound_1} that
\[
\mathbb{P}\big(Y \geq  n\gamma\delta + \tau\big) \leq \exp\Big\{
\frac{-9/2 \cdot\log^2(n)}{3 \log(n)}\Big\} \leq 1/n.
\]
Otherwise, if $n\gamma\delta > \log(n)$, then we have from \eqref{eq:membership_changing_bound_1} that
\[
\mathbb{P}\big(Y \geq  n\gamma\delta + \tau\big) \leq \exp\Big\{
\frac{-9/2\cdot  (n\gamma\delta)^2}{3 n\gamma\delta}\Big\} \leq \exp(-n\gamma\delta) \leq 1/n.
\]
Hence, we are done.
 \end{proof}

\vspace{1em}
Next, we aim to bound $\sigma_{\min}\{(\tilde{M}^{(s)})^\top\tilde{M}^{(t)}\}$.

\begin{lemma} \label{lem:membership_minsig}
Given Assumption \ref{ass:cluster_size}, 
consider particular time indices $s,t\in [0,1]$.
Define $h = L(M^{(s)},M^{(t)})$.
Then, for any two community matrices $M^{(s)}$ and $M^{(t)}$
and their column-normalized versions $\tilde{M}^{(s)}$ and $\tilde{M}^{(t)}$,
\[
\sigma_{\min}\big\{(\tilde{M}^{(s)})^\top\tilde{M}^{(t)}\big\} \geq 1 - ch^{1/2},
\]
where $c = (2c_2K)^{1/2}+c_2^{3/2}K/2$.
\end{lemma}

\begin{proof}
Observe that for any permutation matrix $R \in \mathbb{Q}_K$,
\[
\sigma_{\min}\big\{(\tilde{M}^{(s)})^\top\tilde{M}^{(t)}\big\}  = \sigma_{\min}\big\{(\tilde{M}^{(s)})^\top\tilde{M}^{(t)}R\big\}.
\]
Hence, for notational convenience, let $M_0 = M^{(s)}$, $\Delta_0 = \Delta^{(s)}$  denote a diagonal matrix where the diagonal entries denote the column sum of $M_0$. Additionally, let
\[
M_1 = M^{(t)} R', \quad \text{such that }R' = \min_{R \in \mathbb{Q}_K}\|M^{(s)} - M^{(t)} R\|_0,
\]
and $\Delta_1$ denote the diagonal matrix where the diagonal entries denote the column sum of $M_1$.
Hence, $\tilde{M}_0 = M_0 (\Delta_0)^{-1/2}$
and $\tilde{M}_1 = M_1 (\Delta_1)^{-1/2}$.
Then,
\begin{align}
\sigma_{\min}\big\{(\tilde{M}^{(s)})^\top\tilde{M}^{(t)}\big\} &= 
\sigma_{\min}\big\{(\tilde{M}_0)^\top\tilde{M}_1\big\} =
\sigma_{\min}\big\{(\tilde{M}_0)^\top\tilde{M}_0+ (\tilde{M}_0)^\top(\tilde{M}_1 - \tilde{M}_0)\big\} \nonumber \\
&\overset{(i)}{\geq} 1 - \sigma_{\max}\big\{(\tilde{M}_0)^\top(\tilde{M}_1- \tilde{M}_0)\big\} \overset{(ii)}\geq 1 - \big\|\tilde{M}_1 - \tilde{M}_0\big\|_{\op}, \label{eq:singular_value_membership_lowerbound}
\end{align}
where $(i)$ holds since the spectral radius of $I+A$ for an identity matrix $I$ and arbitrary $A$ is contained within $1 \pm \|A\|_{\op}$ and $(ii)$ holds by submultiplicativity of the spectral norm. Since $\tilde{M}_0= M_0 \Delta_0^{-1/2}$ and $\tilde{M}_1 = M_1 \Delta_1^{-1/2}$, we additionally observe
\begin{align}
\big\|\tilde{M}_1 - \tilde{M}_0\big\|_{\op} &= \big\|M_1 \Delta_1^{-1/2} - M_0 \Delta_0^{-1/2} 
\pm M_0 \Delta_1^{-1/2}\big\|_{\op} \nonumber \\ 
&\leq \big\|(M_1 - M_0) \Delta_1^{-1/2}\big\|_{\op} + \big\| M_0 (\Delta_1^{-1/2}- \Delta_0^{-1/2})\big\|_{\op} \nonumber \\
&\leq  \big\|M_1 - M_0\big\|_{\op}\| \Delta_1^{-1/2}\|_{\op} + 
\| M_0 \|_{\op}\big\|\Delta_1^{-1/2}- \Delta_0^{-1/2}\big\|_{\op}\label{eq:lem:membership_minsig:1}
\end{align}

To bound $\big\|M_1 - M_0\big\|_{\op}$, observe that $\|M_1 - M_0\|_0 = 2nh$ thanks to our permutation of columns above via $R'$.
Rearrange the rows of $M_1-M_0$ such that the first $nh$ rows of $M_1-M_0$ have one 1 and one -1 in each row (and all remaining values are 0) and the remaining rows of $M_1-M_0$ are all 0's.
Then, consider the matrix $(M_1-M_0)(M_1-M_0)^\top$, where the top-left $nh\times nh$ submatrix has values $\{0, 1, 2\}$ in absolute value. Let this submatrix be called $E$. Then,
\[
\lambda_{\max}\big\{(M_1-M_0)(M_1-M_0)^\top\big\} = \lambda_{\max}(E) \overset{(i)}{\leq} 2nh,
\]
where $(i)$ is an upper-bound relying on the maximum value of $E$.
Therefore, we have shown that $\|M_1-M_0\|_{\op} \leq (2nh)^{1/2}$.

Let $n_{\min} = n/(c_2K)$ be defined as the smallest allowable community size, as specified by 
Assumption \ref{ass:cluster_size}.
To bound $\|\Delta_1^{-1/2} - \Delta_0^{-1/2}\|_{\op}$, consider a particular community $k\in\{1,\ldots K\}$. Observe that
\begin{align*}
n_{1,k}^{-1/2} - n_{0,k}^{-1/2}  &= \frac{1}{n_{1,k}^{1/2}} - \frac{1}{n_{0,k}^{1/2}}  =
\frac{n_{1,k}^{1/2} - n_{0,k}^{1/2}}{(n_{1,k}n_{0,k})^{1/2}} = \frac{n_{1,k} - n_{0,k}}{(n_{1,k} n_{0,k})^{1/2} (n_{1,k}^{1/2} + n_{0,k}^{1/2})}\leq \frac{nh}{2n_{\min}^{3/2}}.
\end{align*}
This means that $\|\Delta_1^{-1/2} - \Delta_0^{-1/2}\|_{\op} \leq nh/(2n_{\min}^{3/2})$.

Plugging our results into \eqref{eq:lem:membership_minsig:1}, we have
\[
\big\|\tilde{M}_1 - \tilde{M}_0\big\|_{\op}  \leq (2nh)^{1/2}\cdot \frac{1}{n_{\min}^{1/2}} + n_{\max}^{1/2} \cdot
\frac{nh}{2n_{\min}^{3/2}} 
\overset{(i)}{\leq} \Big\{(2c_2K)^{1/2} + \frac{c_2^{3/2}K}{2}\Big\} \cdot h^{1/2}
\]
where $(i)$ holds from Assumption \ref{ass:cluster_size} and recalling that $h\leq 1$. Plugging this into \eqref{eq:singular_value_membership_lowerbound}, we are done.
\end{proof}

\vspace{1em}
Next, we aim to bound the spectral difference between $\tilde{M}^{(s)}(\tilde{M}^{(s)})^\top$ and
$\tilde{M}^{(t)}(\tilde{M}^{(t)})^\top$

\begin{lemma} \label{lem:membership_projection_op}
For any two membership matrices $M^{(s)}$ and $M^{(t)}$,
\[
\Big\|\tilde{M}^{(s)}(\tilde{M}^{(s)})^\top - \tilde{M}^{(t)}(\tilde{M}^{(t)})^\top\Big\|_{\op} \leq 2c h^{1/2},
\]
where $c$ and $h$ are defined in Lemma \ref{lem:membership_minsig}.
\end{lemma}

\begin{proof}
For notational convenience, let $M_0 = M^{(s)}$ and $M_1 = M^{(t)}$.
We will invoke properties about the distance between two orthonormal matrices (see Lemma 1 from \cite{cai2018rate} for example). Specifically,
\[
\Big\|\tilde{M}_1(\tilde{M}_1)^\top - \tilde{M}_0(\tilde{M}_0)^\top\Big\|_{\op} \leq 2 \cdot
\Big\{1 - \sigma^2_{\min}(\tilde{M}_1^\top \tilde{M}_0)\Big\}^{1/2}.
\]
Hence, we can invoke Lemma \ref{lem:membership_minsig} to finish the proof,
\begin{align*}
\Big\|\tilde{M}_1(\tilde{M}_1)^\top - \tilde{M}_0(\tilde{M}_0)^\top\Big\|_{\op} &\leq 2\cdot \big\{1 - (1-ch^{1/2})^2\big\}^{1/2} \overset{(i)}{\leq}
2\cdot \big\{1 - (1 - c^2h)\big\}^{1/2}
= 2 ch^{1/2},
\end{align*}
where $(i)$ holds since if $a,b>0$, then $(a-b)^2 \leq |(a-b)(a+b)| = |a^2 - b^2|$.
\end{proof}

\vspace{1em}
\begin{lemma} \label{lem:probability_op}
Given Assumption \ref{ass:cluster_size}, for any membership matrix $M^{(t)}$, connectivity matrix $B^{(t)}$ and sparsity $\rho_n$, 
\[
\big\|Q^{(t)}\big\|_{\op} \leq c \rho_n n.
\]
for some constant $c$ that depends on $c_2$ and $K$.
\end{lemma}

\begin{proof}
Let $c$ be a constant that can vary from term to term, depending only on the constants $c_2$ and $K$. 
Defining $n_{\max} = c n$ as defined in Assumption \ref{ass:cluster_size} as the maximum cluster size, we have that
\[
\big\|Q^{(t)}\big\|_{\op} = \big\|\rho_n M^{(t)}B^{(t)}(M^{(t)})^\top\big\|_{\op} \leq c\rho_n n,
\]
via the submultiplicativity of the spectral norm and the fact that $\|B^{(t)}\|_{\op} \leq K$ since $B^{(t)} \in [0,1]^{K \times K}$.
\end{proof}

\vspace{1em}
Below, we upper-bound the probability that each community size stays within a certain size for a two-community model where each community is initialized to be the same size.
\begin{lemma} \label{lem:recursive_decomposition}
Assume a two-community model  (i.e., $K=2$)  following the model described in Section \ref{ss:consistency} (using Assumption \ref{ass:bernoulli} instead of Assumption \ref{ass:poisson}), where each community is initialized to have equal community sizes. 
%Additionally, assume that $\gamma/T < 1/4$. 
Then, with probability at least $1-1/T$, each community's size will stay within
\[
\Big[
\frac{n}{2} - c\cdot \max[\{n \gamma \log(T)\}^{1/2}, \log(T)], 
\frac{n}{2} + c\cdot \max[\{n \gamma \log(T)\}^{1/2},\log(T)]
\Big],
\]
for some universal constant $c$, for all $t \in \mathcal{T}$.
\end{lemma}

As a note, observe that since each node changes memberships with probability $\gamma/T$ for each discrete non-overlapping time interval of length $1/T$, each node will have $\gamma$ events between $t=0$ and $t=1$ on average. Hence, $n\gamma$ is the mean number of total membership changes across all nodes and all time.

\begin{proof}
We wish to bound the community size uniformly across all time $t \in \mathcal{T}\backslash\{1\}$. Let $N_t$ denote the number of nodes in Community 1 at time $t$. For $t \in \mathcal{T}$ where $t > 1/T$, let $t' = t-1/T$ and $\mathcal{F}_{t'}$ denote the filtration of the last time prior to $t$ where $F_{0} = \emptyset$. 
Observe for $t \in \mathcal{T}$, due to the two-community setup,
\begin{equation}\label{eq:lower_bound1}
\mathbb{E}\big(N_t | \mathcal{F}_{t'}\big) = N_{t'}\cdot \big(1-\frac{\gamma}{T}\big) + \big(n - N_{t'}\big)\cdot \frac{\gamma}{T},
\end{equation}
where $N_0 = n/2$.
Let $Z_t = N_t-n/2$ denote the size of Community 1 deviates from parity. Certainly, $Z_t$ is a symmetric random variable around 0 since both communities are initialized with equal sizes.
Our goal is show that $Z_t$ is concentrated near 0 for all $t \in \mathcal{T}$ with high probability
under the provided assumptions.

Towards this end, let $\alpha = 1-2\gamma/T$ and $\beta = \gamma/T$.
Observe that from \eqref{eq:lower_bound1} and the definition of $Z_t$,
\begin{equation}\label{eq:lower_bound2}
\mathbb{E}\big(Z_t | \mathcal{F}_{t'}\big) = \big(1- \frac{2\gamma}{T}\big) \cdot Z_{t'} = \alpha \cdot Z_{t'}.
\end{equation}
where for $Z_0 =0$.
We can think of $\alpha$ as a factor that shrinks $Z_{t'}$ towards 0 (i.e., equal community sizes).
Define 
\begin{equation}\label{eq:lower_bound3}
M_t = Z_t - \mathbb{E}(Z_t | \mathcal{F}_{t'}) = Z_t - \alpha Z_{t'}, \quad \text{for}\quad t \in \mathcal{T}.
\end{equation}
as the deviation of the expected size of Community 1 from its expectation at time $t$.
 Recalling the functional form of centered Bernoulli's, observe that from 
\eqref{eq:lower_bound1} and \eqref{eq:lower_bound2}, 
\begin{equation}\label{eq:lower_bound3b}
M_t | \mathcal{F}_{t'} \overset{d}{=} \sum_{i=1}^{n} \xi_{i,t}
\end{equation}
where
\[
\text{if } i \in \{1,\ldots,N_{t'}\}, \quad \text{then }\xi_{i,t} = \begin{cases}
\beta &\quad \text{with probability }1-\beta\\
-(1-\beta) &\quad \text{with probability }\beta
\end{cases},
\]
and
\[
\text{if } i \in \{N_{t'}+1, \ldots, n\}, \quad \text{then }
\xi_{i,t} = \begin{cases}
1-\beta &\quad \text{with probability }\beta\\
-\beta &\quad \text{with probability }1-\beta
\end{cases}.
\]
Without loss of generality, let $t_1 = t'= t - 1/T$, $t_2 = t - 2/T, \ldots, t_S = 1/T$ for $S = t(T-1)-1$.
Hence, $t_1 > t_2 > \ldots > t_S$, meaning $t_S$ is the earliest time, and $t_1$ is the latest time.
Then, building upon a recursive decomposition for \eqref{eq:lower_bound3},
\begin{equation}\label{eq:lower_bound4}
Z_t = M_t + \alpha M_{t_1} + \alpha^2 M_{t_2} + \ldots + \alpha^{S} M_{t_S},
\end{equation}
recalling that $M_{t_S} = M_{1/T} = 0$ by our definitions.

We seek a Chernoff-like argument. Observe that for any $c > 0$,
\begin{align}
\mathbb{E}\big(e^{cZ_t}\big) &= \mathbb{E}\big\{e^{c(M_t + \alpha M_{t_1} + \alpha^2 M_{t_2} + \ldots + \alpha^{S} M_{t_S})}\big\} \nonumber \\
&= \mathbb{E}\Big[\mathbb{E}\big\{e^{c(M_t + \alpha M_{t_1} + \alpha^2 M_{t_2} + \ldots + \alpha^{S} M_{t_S})} | \mathcal{F}_{t_1}\big\}\Big] \nonumber\\
&= \mathbb{E}\Big[\mathbb{E}\big\{e^{cM_t} | \mathcal{F}_{t_1}\big\} e^{\alpha M_{t_1} + \alpha^2 M_{t_2} + \ldots + \alpha^{S} M_{t_S}}\Big]. \label{eq:lower_bound5}
\end{align}

Analyzing the first term on the RHS of \eqref{eq:lower_bound5}, provided that $c<1$,
\begin{align}
\mathbb{E}\big(e^{cM_t} | \mathcal{F}_{t_1}\big) &\overset{(i)}{=} \prod_{i=1}^{n} \mathbb{E}e^{c \xi_{i,t}}\nonumber \\
&= \prod_{i=1}^{n}\Big\{ 1 + c \mathbb{E}\big(\xi_{i,t}\big) + \sum_{k=2}^{\infty} \mathbb{E}\big(\frac{1}{k!}c^k \xi^k_{i,t}\big)\Big\} \nonumber \\
&\overset{(ii)}{=}\prod_{i=1}^{n}\Big(1+ \sum_{k=2}^{\infty}c^k \beta\Big) = 
\prod_{i=1}^{n}\Big(1+ \frac{\beta c^2}{1-c}\Big) \overset{(iii)}{\leq} \exp\big(\frac{n\beta c^2}{1-c}\big). \label{eq:lower_bound6}
\end{align}
where $(i)$ holds from \eqref{eq:lower_bound3b},  $(ii)$ holds since $\mathbb{E}(\xi_{i,t}) = 0$ and $\mathbb{E}(|\xi_{i,t}|^k) \leq \beta = \gamma/T$, and $(iii)$ holds since $\exp(x) \geq 1+x$.
Combining \eqref{eq:lower_bound6} with \eqref{eq:lower_bound5}, we obtain
\begin{align}
\mathbb{E}\big(e^{cZ_t}\big) &\leq e^{\frac{n\beta c^2}{1-c}} \cdot \mathbb{E}\big\{
e^{c(\alpha M_{t_1} + \alpha^2 M_{t_2} + \ldots + \alpha^{S} M_{t_S})}
\big\} \nonumber \\ 
&\overset{(iv)}{=} 
e^{\frac{n\beta c^2}{1-c}} \cdot \mathbb{E}\Big\{\mathbb{E}\big(
e^{c\alpha M_{t_1}}|\mathcal{F}_{t_2} 
\big)
e^{c(\alpha^2 M_{t_2} + \ldots + \alpha^{S} M_{t_S})}
\Big\} \nonumber \\
&\overset{(v)}{\leq}
e^{\frac{n\beta c^2}{1-c}} \cdot \mathbb{E}\Big\{\mathbb{E}\big(
e^{c M_{t_1}}|\mathcal{F}_{t_2} 
\big)^\alpha
e^{c(\alpha^2 M_{t_2} + \ldots + \alpha^{S} M_{t_S})}
\Big\}  \nonumber \\
&\overset{(vi)}{\leq}
e^{\frac{n\beta c^2}{1-c}} \cdot \mathbb{E}\Big\{
\big(e^{\frac{n\beta c^2}{1-c}} \big)^\alpha
e^{c(\alpha^2 M_{t_2} + \ldots + \alpha^{S} M_{t_S})}
\Big\} \nonumber\\
&\leq
e^{\frac{(1+\alpha) n\beta c^2}{1-c}} \cdot \mathbb{E}\Big\{
e^{c(\alpha^2 M_{t_2} + \ldots + \alpha^{S} M_{t_S})}
\Big\}, \label{eq:lower_bound7}
\end{align}
where $(iv)$ holds by an argument analogous to \eqref{eq:lower_bound5},
$(v)$ holds by Jensen's inequality since $f(x) = x^\alpha$ is concave for $\alpha \in (0,1)$,
$(vi)$ holds by an argument analogous to \eqref{eq:lower_bound6}.
Repeating the argument for \eqref{eq:lower_bound7} a total for $S$ times (recalling that $\alpha \in (0,1)$) yields our desired inequality
\begin{equation} \label{eq:lower_bound8}
\mathbb{E}\big(e^{cZ_t}\big) \leq e^{\frac{(1+\alpha + \alpha^2 + \ldots + \alpha^S)n\beta c^2}{1-c}} \leq 
e^{\frac{Tn\beta c^2}{(1-c)}}
\end{equation}

Returning to our original goal of constructing a tail bound for $Z_t$, we then use Markov's inequality alongside \eqref{eq:lower_bound8} to yield the inequalities that for any $\tau > 0$,
\[
\mathbb{P}\big(Z_t \geq \tau\big) \leq \mathbb{P}\big(e^{cZ_t} \geq e^{c\tau}\big)
\leq \mathbb{E}\big(e^{cZ_t}\big)/e^{c\tau} \overset{(viii)}{\leq} \exp\Big(\frac{Tn\beta c^2}{1-c} - c\tau\Big)
\]
where $(viii)$ holds from \eqref{eq:lower_bound8}. Setting $c = \tau/(2Tn\beta+\tau)$ yields,
\[
\mathbb{P}\big(Z_t \geq \tau\big) \leq 
\exp\Big(\frac{-\tau^2}{4Tn\beta + 2\tau }\Big).
\]
By symmetry of $Z_t$ around 0, we equally obtain an equivalent upper-bound for $\mathbb{P}(-Z_t \geq \tau)$. This combines to form our desired bound,
\[
\mathbb{P}\big(|Z_t|\geq \tau\big) \leq 
2\exp\Big(\frac{-\tau^2}{4Tn\beta + 2\tau }\Big).
\]
Hence, by setting $\tau = c'\cdot\max\{(Tn\beta \log(T))^{1/2}, \log(T)\}$ for a universal $c'$,
we have 
\[
\mathbb{P}\big(|Z_t|\geq \tau\big) \leq \frac{1}{T^2}.
\]
Therefore, using a union bound, we are ensured with probability at least $1-1/T$, all $\{Z_t\}$'s are bounded by 
\[
c'\cdot\max[\{Tn\beta \log(T)\}^{1/2}, \log(T)]
= c' \cdot\max[\{n\gamma \log(T)\}^{1/2}, \log(T)]
\]
in magnitude simultaneously for all $t\in\mathcal{T}$.
\end{proof}

\vspace{1em}
Below, we upper-bound the probability the number of nodes that change membership across between any two consecutive time points is less than a particular threshold. The following lemma is different from Lemma \ref{lem:membership_changing_bound} for two main reasons: 1) 
Lemma \ref{lem:membership_changing_bound} handles the maximal difference between two membership matrices within a time interval, whereas the following lemma focuses on only two consecutive time points. 2) The following lemma will make an assumption about node's behavior within a time interval of $1/T$ that will simplify the proof.

\begin{lemma} \label{lem:tail_binomial}
Assume a two-community model (i.e., $K=2$) following the model described in Section \ref{sec:model} (using Assumption \ref{ass:bernoulli} instead of Assumption \ref{ass:poisson}). %and that if nodes that change memberships deterministically do not return to the original membership within a time interval of $1/T$.
Then, the probability that more than
\[
\frac{5n\gamma}{T} + 4\log(T)
\]
nodes change membership between any two (fixed) consecutive time points $s,t \in \mathcal{T}$ (i.e., $t-s = 1/T$) is at most $1/T^2$.
\end{lemma}
\begin{proof}

Consider the two events,
\begin{align*}
\mathcal{E}_1 &= \Big\{n \cdot L(M^{(s)}, M^{(t)}) \geq n\cdot \frac{\gamma}{T} + \tau \Big\}\\
\mathcal{E}_2 &= \Big\{(\# \text{ of nodes that change communities anytime between $s$ and $t$}) \geq n\cdot \frac{\gamma}{T} + \tau \Big\},
\end{align*}
where, recall, $n \cdot L(M^{(s)}, M^{(t)})$ is the number of nodes that change communities when comparing time $s$ to time $t$.
We are interested in bounding $(\mathcal{E}_1)$ for an appropriately chosen $\tau$. 
However, observe that $\mathcal{E}_1 \Rightarrow \mathcal{E}_2$, hence $\mathbb{P}(\mathcal{E}_1) \leq \mathbb{P}(\mathcal{E}_2)$. Therefore, we
are interested in bounding $\mathbb{P}(\mathcal{E}_2)$.

By Assumption \ref{ass:bernoulli}, each node changes memberships within a time interval of length $1/T$ independently of each other at rate $\gamma/T$. Hence,
\begin{equation} \label{eq:bernoulli_upper-tail_1}
\mathbb{P}(\mathcal{E}_2) = \mathbb{P}\Big(X \geq n\cdot \frac{\gamma}{T} + \tau\Big).
\end{equation}
Since there are only two communities and we assume that if nodes that change memberships deterministically do not return to the original membership within a time interval of $1/T$,
the $\text{Bernoulli}(\gamma/T)$ process of node membership changes in Assumption \ref{ass:bernoulli} allows us to model $X$ as a Bernoulli random variable with mean $n\gamma/T$. 

Therefore, to upper-bound the RHS of \eqref{eq:bernoulli_upper-tail_1}, we use Bernstein's inequality (for example, Lemma 4.1.9 from \cite{de2012decoupling}):
\begin{equation} \label{eq:bernoulli_upper-tail_2}
 \mathbb{P}\Big(X \geq n\cdot \frac{\gamma}{T} + \tau\Big) \leq \exp\Big(
 \frac{-\frac{1}{2}\tau^2}{n\gamma/T + \frac{1}{3}\tau}
 \Big).
\end{equation}
Consider $\tau = 4 n\gamma/T +4 \log(T)$. If $\log(T) > n\gamma/T$,  then we have
from \eqref{eq:bernoulli_upper-tail_2} that
\[
 \mathbb{P}\Big(X \geq n\cdot \frac{\gamma}{T} + \tau\Big) 
 \leq
 \exp\Big\{\frac{-16\log^2(T)}{9 \cdot 1/3 \cdot \log(T)}\Big\}
 \leq 1/T^2.
\]
Otherwise, if $n\gamma/T > \log(T)$, then  we have
from \eqref{eq:bernoulli_upper-tail_2} that
\[
 \mathbb{P}\Big(X \geq n\cdot \frac{\gamma}{T} + \tau\Big) 
 \leq
 \exp\Big\{\frac{-16(n\gamma/T)^2}{(1+8/3) \cdot n\gamma/T}\Big\}\leq 
  \exp\big(-2n\gamma/T\big) \leq   \exp\big(-2\log(n)\big) 
 \leq 1/T^2.
 \]
Putting everything together, we have shown that
\[
\mathbb{P}\Big\{n \cdot L(M^{(s)}, M^{(t)}) \geq 5 n \cdot \frac{\gamma}{T} + 4 \log(T)\Big\} \leq 1/T^2,
\]
and hence we are done.

\end{proof}

\section{Additional simulation}

\subsection{Simulation of homophilic networks}

The simulation investigated in Section \ref{sec:simulation} of the main text comprised a collection of both homophilic and heterophilic networks. 
Arguably, comparing KD-SoS to the ``PZ'' method proposed in \cite{pensky2019spectral} is unfair to the latter method, as it was not designed for such a setting. 
To investigate how the four methods in our Simulation setting perform in a more favorable setting than the PZ method, we simulated a separate scenario where the setup is identical to that in Section \ref{sec:simulation}, except that all the networks are homophilic. Specifically, 
\[
B^{(t)} = \begin{bmatrix}
0.62 & 0.22 & 0.46\\
0.22 & 0.62 & 0.46\\
0.46 & 0.46 & 0.85
\end{bmatrix} \quad \text{ for } t \in \mathcal{T}.
\]

The results are shown in Figure \ref{fig:simulation_stationary_all_simple}. 
We see that, in comparison to the original simulation setting with both homophilic and heterophilic networks shown in Figure \ref{fig:simulation_stationary_all-cleaned} in the main text, this simulation demonstrates that PZ outperforms all three other methods, including KD-SoS. 
This means that if all the networks were homophilic, summing the adjacency matrices within a certain bandwidth aggregates information more effectively than debiasing the sum of squared adjacency matrices.

However, despite these results, we still advocate for KD-SoS in practice, as it is challenging for practitioners to determine whether all the networks in their analysis are homophilic. 
KD-SoS can handle both homophilic and heterophilic networks without requiring this prior information.

\begin{figure}[tb]
  \centering
  \includegraphics[width=200px]{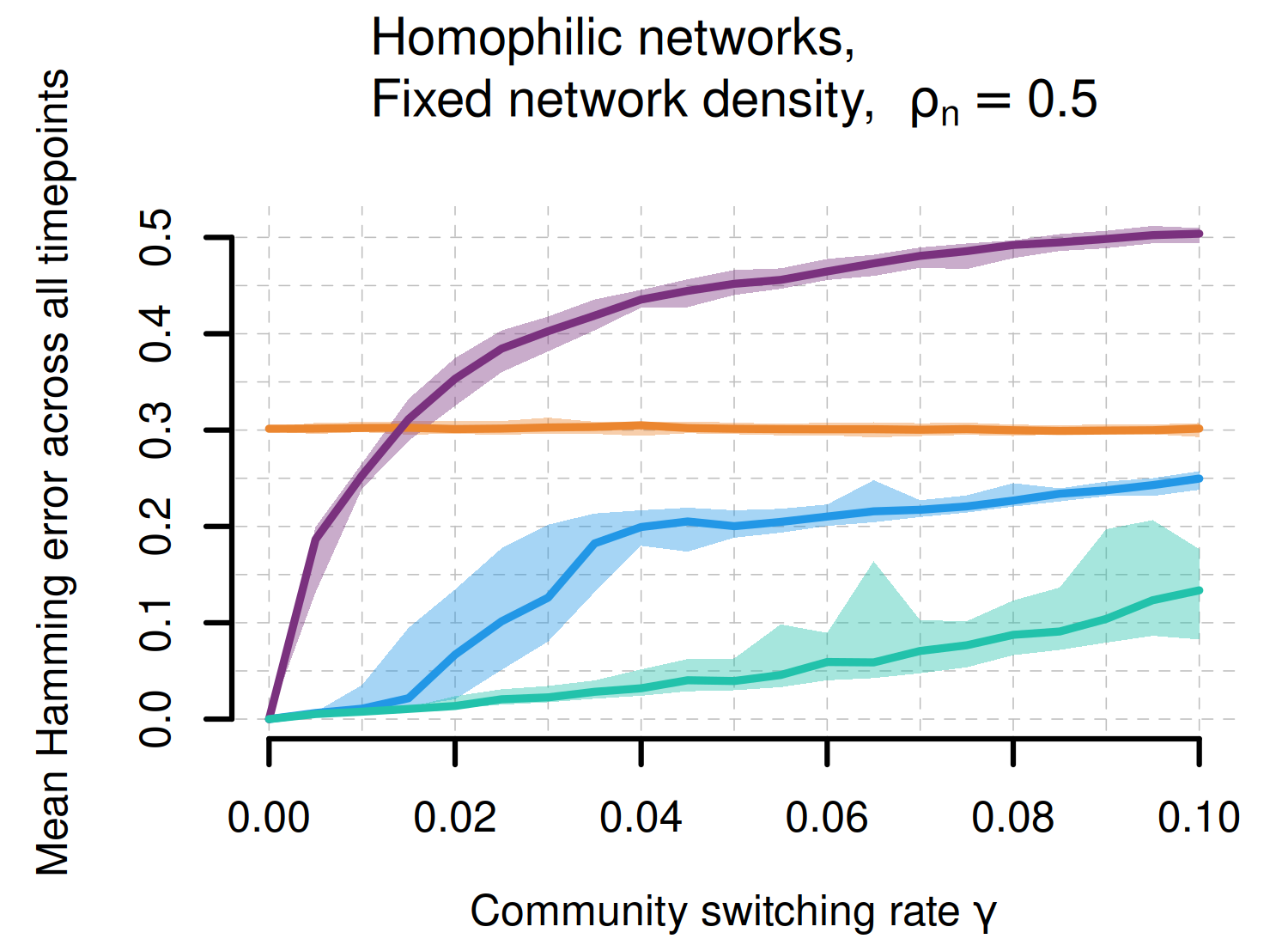}
  \quad
   \includegraphics[width=200px]{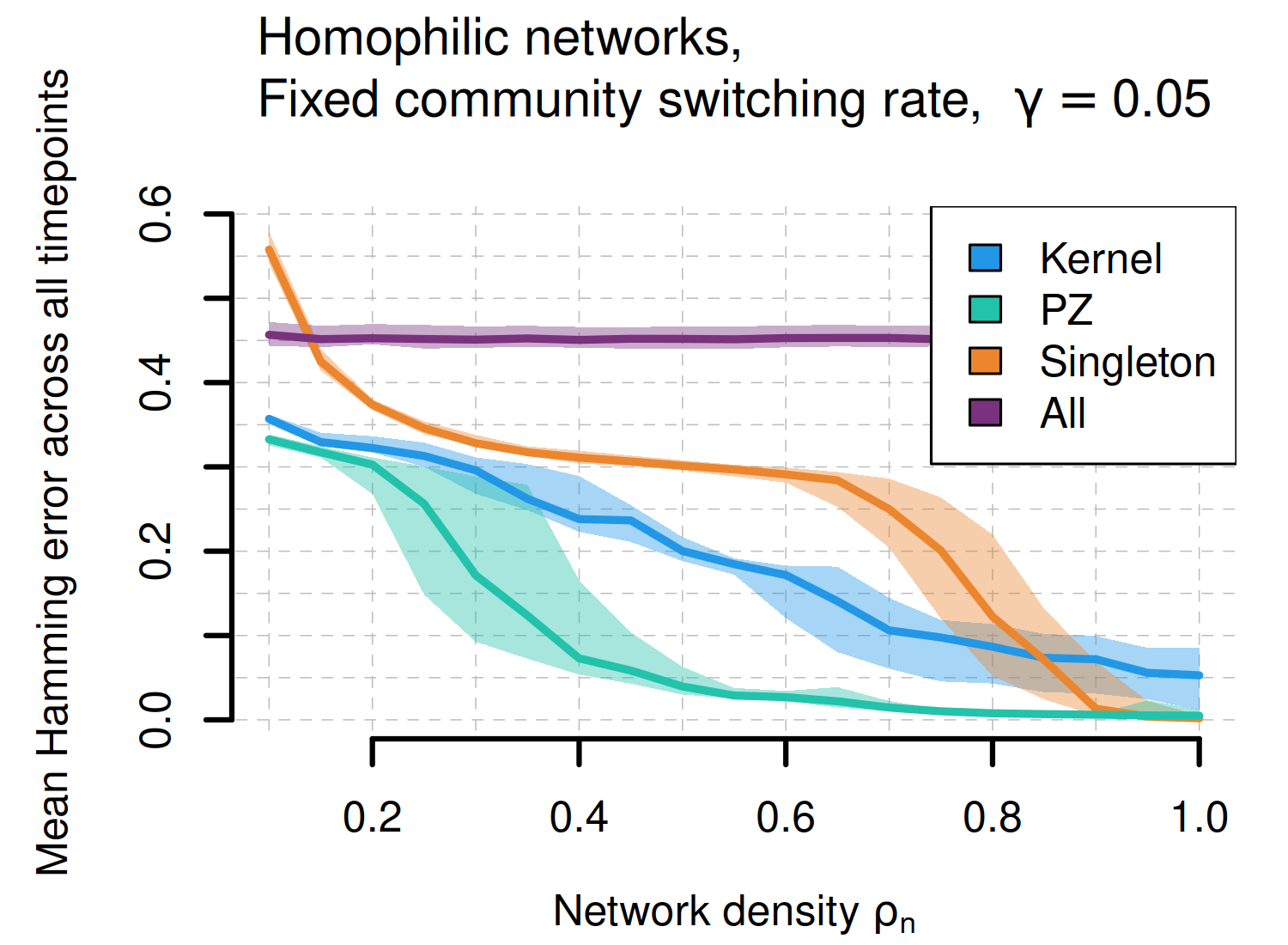}
  \caption
   { 
  Simulation of the ``pure'' homophilic setting, where the results are shown in the same layout of Figure \ref{fig:simulation_stationary_all-cleaned} in the main text.
   }
    \label{fig:simulation_stationary_all_simple}
\end{figure}

\subsection{Simulation of misspecified $K$}

While Section \ref{sec:tuning} in the main text documents a novel procedure to select the kernel bandwidth, another important question in practice is how to choose the number of communities. 
As mentioned in the Discussion section of the main text, this is a challenging goodness-of-fit statistical problem, where even questions about a single SBM network remain open, as noted in \cite{li2016network} and \cite{chen2018network}.
Nonetheless, we can investigate the empirical performance of KD-SoS with a misspecified number of communities $K$.

We construct a simulation setting using the setup described in Section \ref{sec:simulation} in the main text, and we focus on four specific values of the community switching rate $\gamma$ and network density $\rho_n$, specifically $(\gamma, \rho_n) = \{(0.01,0.3), (0.01,0.5), (0.05, 0.3), (0.05,0.5)\}$.
In each setting, we apply KD-SoS with $K\in\{2,3,4\}$ over 25 trials.

To evaluate the performance of KD-SoS with a misspecified $K$, we use the following procedure: First, for every estimated community, compute the Shannon entropy of the ``true'' community among the nodes in that estimated community.
Since there are three true communities, the distribution of true communities in every estimated community is a vector $(\hat{v}_1, \hat{v}_2, \hat{v}_3)$ where $\sum_{k=1}^3 \hat{v}_k = 1$ and $\hat{v}_k \geq 0$ for all $k \in \{1,\ldots,3\}$. The Shannon entropy is defined as
\[
-\sum_{k=1}^{3}\hat{v}_k \log(\hat{v}_k),
\]
where we define $0\log(0)=0$. 
A higher normalized Shannon entropy indicates that the community exhibits a more uniform distribution of true communities.
To score the overall clustering among all the estimated communities, we average (mean) the normalized Shannon entropy over all the estimated communities.
Given this metric, the ``best'' choice of $K$ would be one with the smallest normalized Shannon entropy, because it means that this choice of $K$ yields the most ``pure'' communities. 

Our results are shown in Figure \ref{fig:simulation_misspecified}. 
We observe that across three out of four simulation settings, the choice of $K=3$ (the appropriately specified number of communities) yields the best results, as it has the smallest normalized Shannon entropy. 
For $\rho_n=0.3$ and $\gamma=0.05$, all three choices of $K$ yield very similar normalized Shannon entropy. 
Additionally, we observe that it is often ``safer'' to specify too many communities than too few.
Conceptually, this is corroborated by theoretical results about estimating the number of communities in SBMs where, asymptotically, the probability of a method under-estimating the number of communities goes to zero, but it is challenging to bound the probability of a method over-estimating the number of communities (see \cite{chen2018network}).

\begin{figure}[tb]
  \centering
  \includegraphics[width=250px]{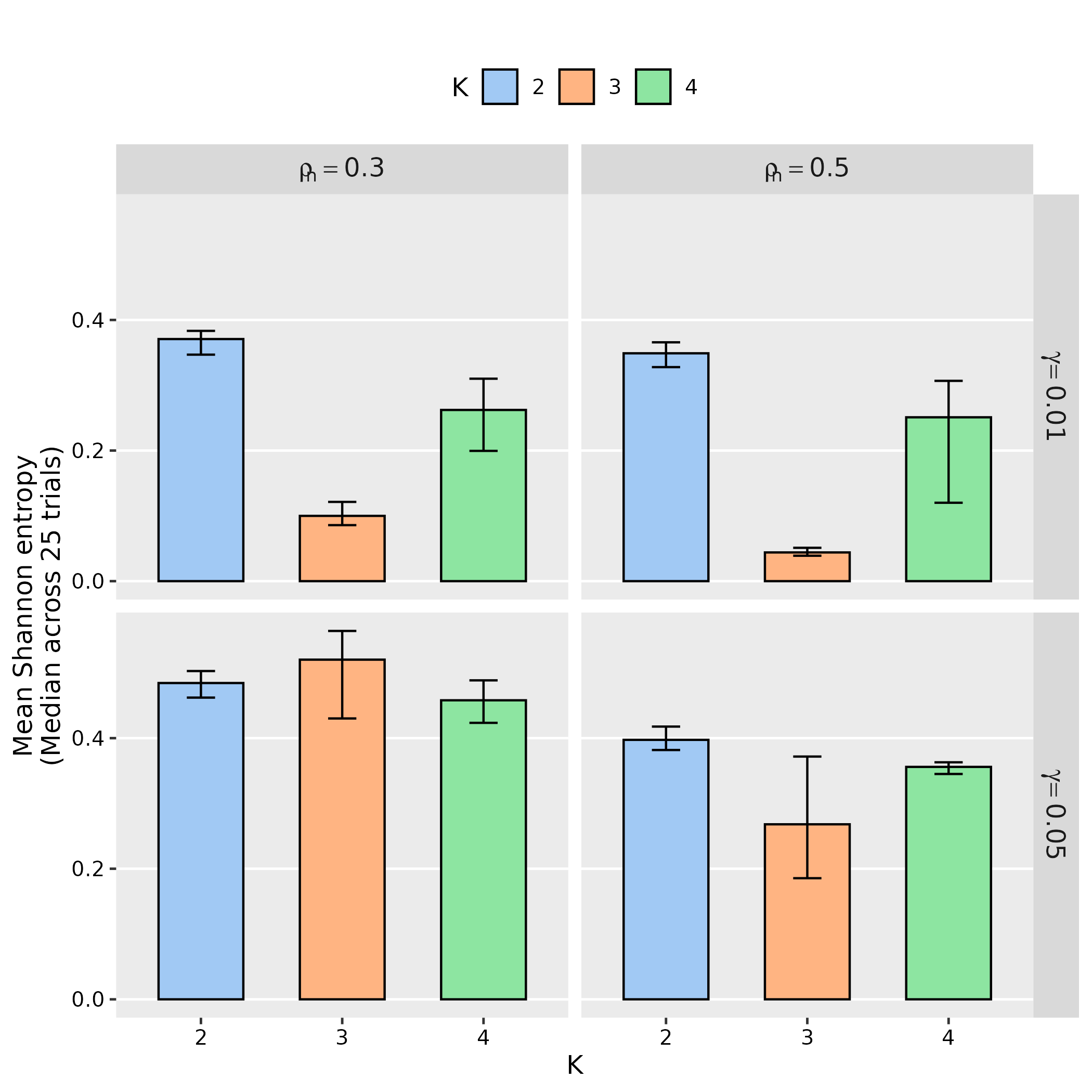}
  \caption
   { 
  Simulation of four different settings of network density $\rho_n$ and community switching rate $\gamma$ for the setting described in Section \ref{sec:simulation} in the main text with three ``true'' communities, where we apply KD-SoS with $K \in \{2,3,4\}$ (colored blue, orange, and green, respectively) even though there are only three true communities. The y-axis shows the median (over 25 trials) normalized Shannon entropy of the true communities within each estimated community, averaged across all the $K$ estimated communities. A smaller value on the y-axis denotes a more appropriate choice of $K$.
   }
    \label{fig:simulation_misspecified}
\end{figure}

\subsection{Simulation of non-stationary transition matrix}

The simulation in Section \ref{sec:simulation} in the main text had a stationary transition matrix dictating how nodes at time $t$ transition to a potentially different community at time $t+1$.
The simulation shown in the main text sets this transition matrix to the same value for all time $t \in \mathcal{T}$ for the sake of simplicity of exposition. 
We note that our theory does not require this to be the case. 
Here, we demonstrate a simulation where, even with a non-stationary transition matrix (i.e., the transition matrix changes as a function of time $t$ itself), KD-SoS still retains its excellent performance.

Towards this end, we construct a simulation setting using the setup described in Section \ref{sec:simulation} of the main text, except that we modify the transition matrix in \eqref{eq:simulation:stationary_markov} of the main text. Instead, in this simulation, the transition matrix is generated at random for every time $t$ according to the following procedure:
\begin{itemize}
    \item Initialize the the $K\times K$ transition matrix to have $1-\gamma$ along the diagonal.
    \item For every row $i \in \{1,\ldots,K\}$, sample a value $j \in \{1,\ldots,K\}\backslash \{i\}$ uniformly at random. Set entry $(i,j)$ of the transition matrix to be $\gamma$.
\end{itemize}
In this way, $100\cdot \gamma$ percent of the nodes in any community transition to a different community, and this receiving community can change from one time $t$ to the next.

We display our results in Figure \ref{fig:simulation_nonstationary} where we fix the network density $\rho_n=0.5$ and vary the community switching rate $\gamma$. This is analogous to Figure \ref{fig:simulation_stationary_all-cleaned} (left) in the main text, except the transition matrix is now non-stationary. Broadly speaking, the relative ordering of all four methods remains the same for all community switching rates $\gamma$ compared to the stationary setting shown in Figure \ref{fig:simulation_stationary_all-cleaned} (left).
Mainly, KD-SoS still outperforms the three other methods in this simulation setting, reinforcing the fact that our theory about KD-SoS's membership recovery does not depend on a stationary transition matrix.

\begin{figure}[tb]
  \centering
  \includegraphics[width=200px]{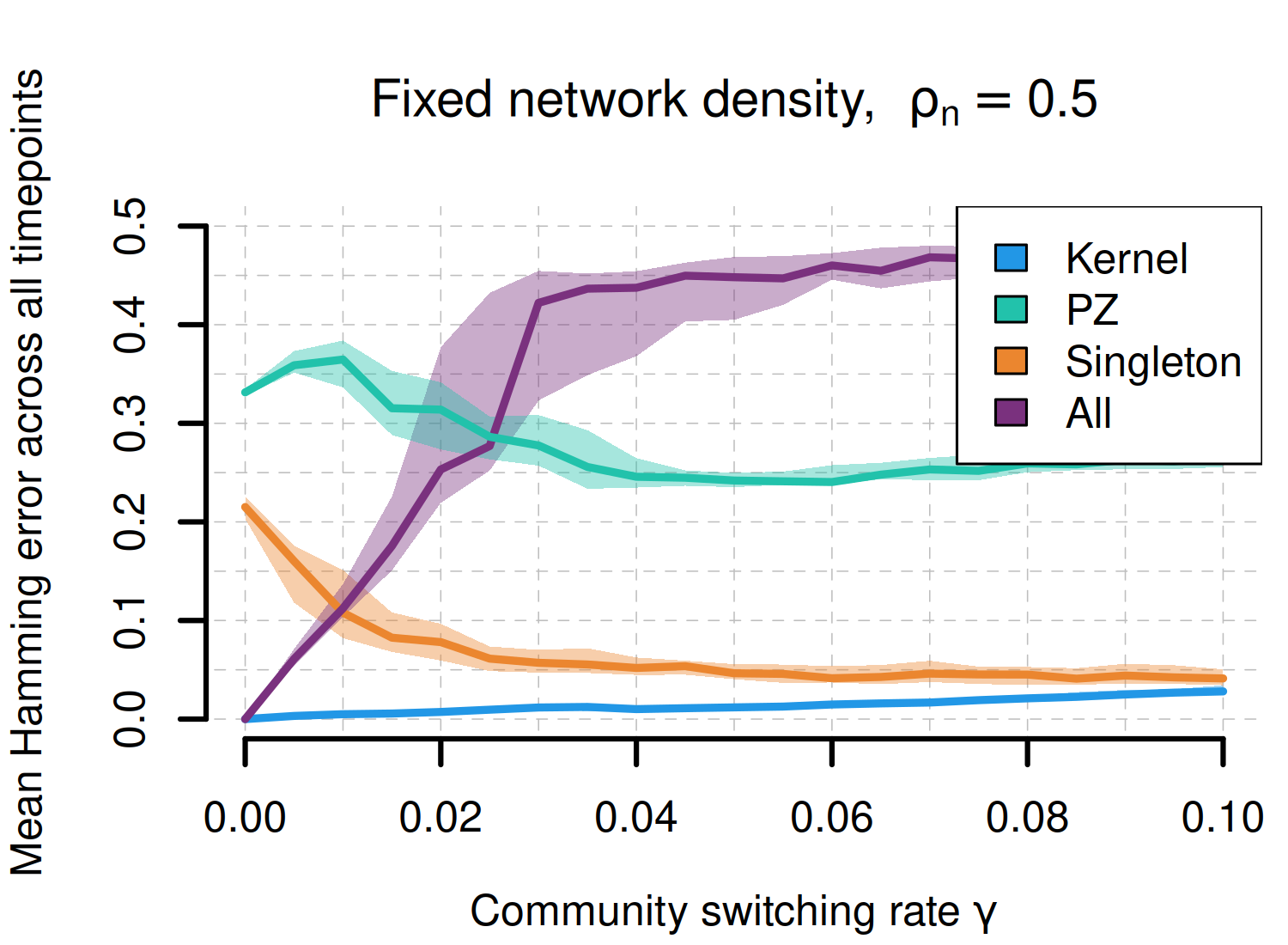}
  \caption
   { 
  Simulation of the simulation with a non-stationary transition matrix over time, where the results are shown in the exact layout of Figure \ref{fig:simulation_stationary_all-cleaned} in the main text.
   }
    \label{fig:simulation_nonstationary}
\end{figure}

\subsection{Simulation of changing network density over time}

Here, we investigate how the estimated bandwidth using the tuning procedure described in Section \ref{sec:tuning} in the main text varies with the network density $\rho_n$. 
We would expect that as the network density increases, the estimated bandwidth should increase. This is because a lower network density means there is less information in a network, necessitating a larger bandwidth to accumulate sufficient information across more networks.
Additionally, we investigate whether our bandwidth estimation procedure can handle more challenging settings where the network density $\rho_n$ varies over time $t$. Such a setting would require using our tuning procedure in a locally adaptive fashion.

To investigate both aspects, we construct a simulation setting using the setup described in Section \ref{sec:simulation} in the main text, except that we vary the network density $\rho_n$ with time $t$ and simulate 100 equally spaced time indices between $[0,1]$. Specifically, we vary $\rho_n$ varying between 0.25 and 1 based on a sinusoidal function of $t$, where $\rho_n(t) = 0.25$ for $t=0$ and $\rho_n(t) = 1$ for $t=1$ (Figure \ref{fig:simulation_dynamic_rho}, left). By having $\rho_n(t)$ vary with one and a half phases between $t\in[0,1]$, we are ensured that both sparse and dense matrices are equally affected by any potential boundary bias issue. Note that $\rho_n(t)=1$ does not imply the network is fully connected, see the construction of the probability matrix $Q^{(t)}$ in Equation \ref{eq:model:qt} in the main text.

To estimate a local bandwidth for each network at time $t$, we modify our tuning procedure described initially in Section \ref{sec:tuning} in the main text.
Specifically, our modified procedure is the following:
\begin{enumerate}
\item For each bandwidth $r \in \{r_1,\ldots,r_m\}$ at time $t\in[0,1]$, compute the score of the bandwidth $\theta_t(r)$ in the following way: Compute the leading eigenspaces of $\sum_{s \in \mathcal{S}}(A^{(s)})^2 - D^{(s)}$, where $\mathcal{S}$ is either $\mathcal{S}(t; c\cdot r)\backslash [0,t)$ or $\mathcal{S}(t;c \cdot r)\backslash (t,1]$ for $\mathcal{S}(t;c \cdot r)$ defined in \eqref{eq:z_box}. Then, compute the $\sin \Theta$ distance between these two eigenspaces via \eqref{eq:sin_theta}, denoted as $\theta_t(r)$.
\item Choose the optimal bandwidth with the smallest score, i.e., $\hat{r}_t = \argmin_{r \in \{r_1,\ldots,r_m\}} \theta_t(r)$.
\end{enumerate}

\begin{figure}[tb]
  \centering
  \includegraphics[width=200px]{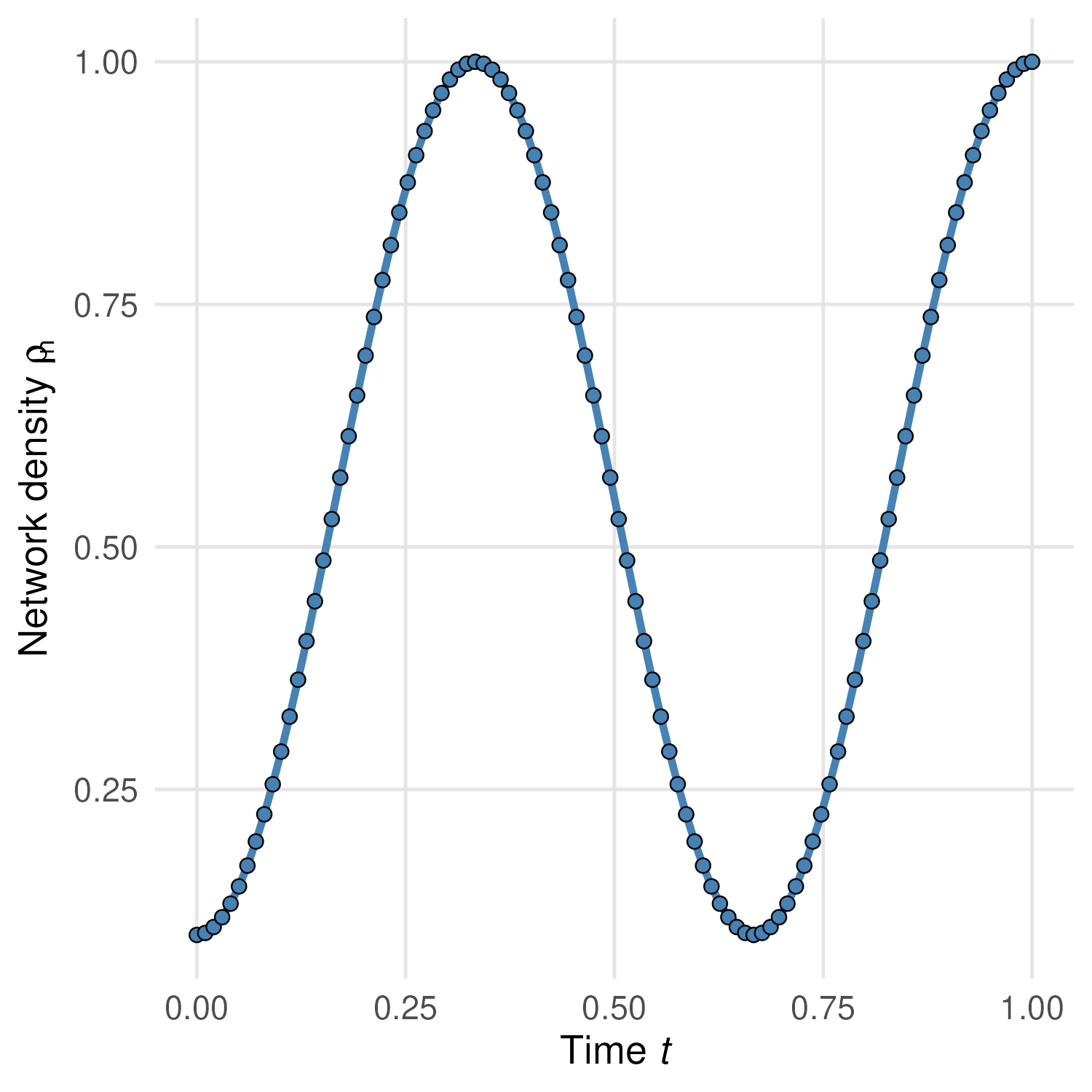}
  \quad
   \includegraphics[width=200px]{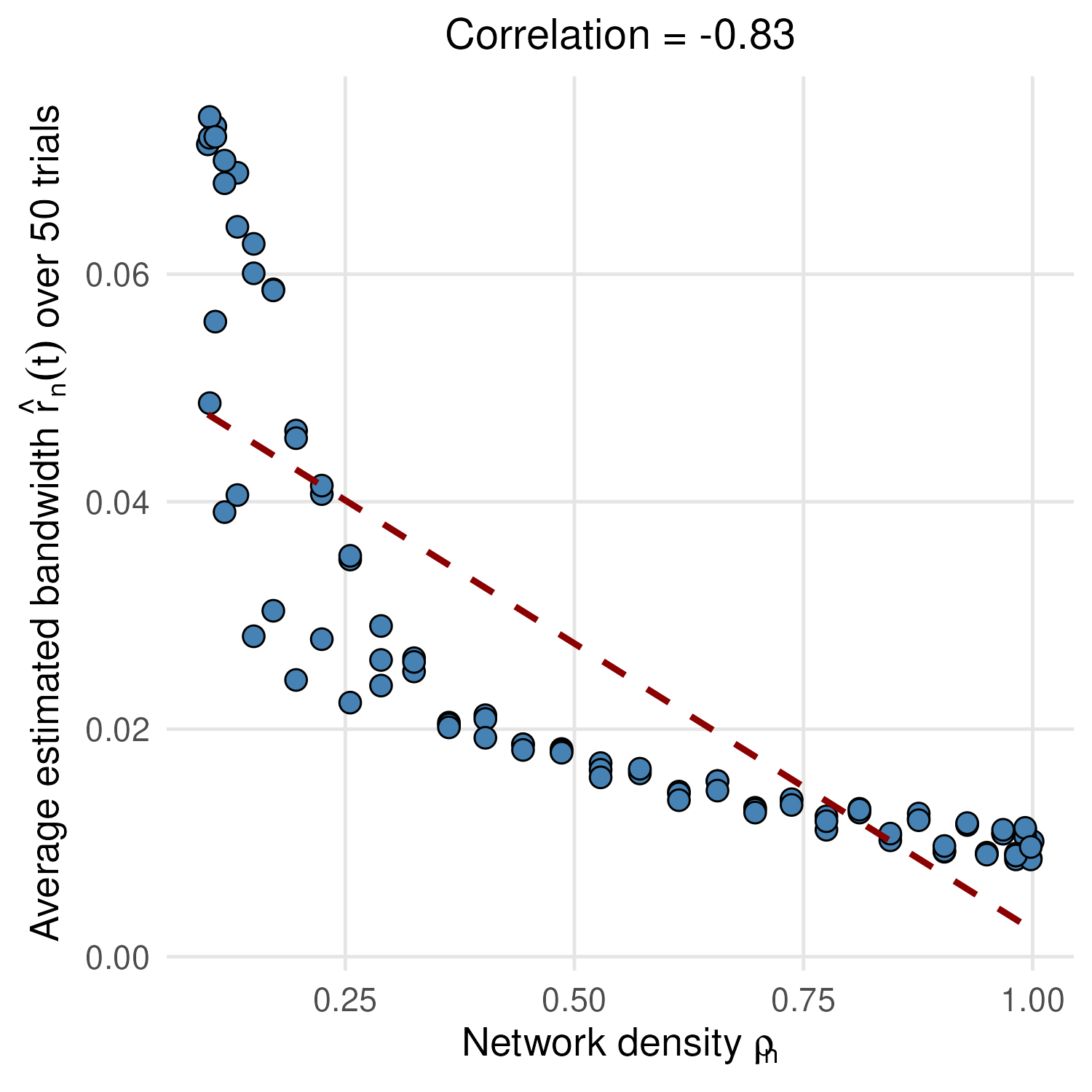}
  \caption
   { 
  Simulation of 100 networks where the network density $\rho_n$ varies with time $t$ (left). The estimated bandwidth $\hat{r}_t$ for each time index $t$, where we choose the bandwidth separately for each time index based on the smallest $\sin \Theta$ distance (right). 
  We observe that, in general, the bandwidth decreases as the network density increases.  
  The red dashed line denotes the linear regression line, which displays a correlation of $-0.83$.
   }
    \label{fig:simulation_dynamic_rho}
\end{figure}

We set $r_1,\ldots,r_m$ to vary from 0 to 0.1, equally spaced for $m=25$.
Our results are shown in Figure \ref{fig:simulation_dynamic_rho} (right).
We see that, averaged over 50 trials, the average bandwidth decreases with the network density $\rho_n$ as mathematically expected.
This demonstrates that our tuning procedure behaves appropriately and can be used even in settings where the network density varies with time.

\section{Additional details and plots of networks}
In this section, we provide preprocessing details and additional plots to display the results across all 12 networks.

\subsection{Preprocessing of networks}
The preprocessing consists of different steps: 1) preprocessing the scRNA-seq data via SAVER, 2) ordering the cells via pseudotime, and 3) constructing the 12 networks.

\begin{itemize}
\item Preprocessing the scRNA-seq data via SAVER:
Using the data from \cite{trevino2021chromatin}, we first extract
the cells labeled \texttt{In Glun trajectory} as well as in cell types \texttt{c8},
\texttt{c14}, \texttt{c2}, \texttt{c9}, \texttt{c5}, and \texttt{c7}, as labeled by the authors.
Additionally, we select genes that are marker genes for our selected cell types,
as well as the differentially expressed genes between glutamatergic neurons between 16 postconceptional weeks and 20-24 postconceptional weeks, both sets also labeled by the authors.

Using these selected cells and genes, we apply SAVER \citep{huang2018saver} to denoise the data using the default settings. We use this method over other existing denoising methods for scRNA-seq data since SAVER has been shown to validate and meaningfully retain correlations among genes experimentally.

\vspace{.5em}
\item Ordering the cells via pseudotime: To construct the pseudotime, we analyze the data based on the leading 10 principal components (after applying \texttt{Seurat::NormalizeData}, \texttt{Seurat::FindVariableFeatures}, \texttt{Seurat::ScaleData}, and \texttt{Seurat::RunPCA}). 
We then apply Slingshot \citep{street2018slingshot}
to the cells in this PCA embedding, based on ordering the cell types:  \texttt{c8},
followed by \texttt{c14}, followed by \texttt{c2}, followed by \texttt{c9} and \texttt{c5}, and finally followed by \texttt{c7}. (The authors provided this order.) This provides the appropriate ordering of the 18,160 cells.

\vspace{.5em}
\item Constructing the 12 networks: We now have the SAVER-denoised scRNA-seq data and the corresponding cell ordering.
Based on this ordering, we partition the 18,160 cells into 12 equally-sized bins.
For each bin, we compute the correlation matrix among all genes and convert this matrix into an adjacency matrix based on whether the correlation magnitude is above 0.75. 
Finally, once we have completed this for all 12 networks, we remove any genes whose median degree (across all 12 networks) is 0 or 1. This results in the 12 networks we analyze among the 993 genes.
\end{itemize}

\subsection{Selection of the number of latent dimensions $K$}

We show in Figure \ref{fig:appendix_K} the rationale for choosing $K=10$ in our analysis of the developing brain dataset.
Our diagnostic is based on the debiased sum of squared adjacency matrices, 
\[
\sum_{t=1}^{12}\Big[(A^{(t)})^2 - D^{(t)}\Big].
\]
We chose this matrix because it uniformly aggregates information across all time points, and we use it to gauge the appropriate number of latent dimensions before analyzing the time-varying dynamics.
We compute an eigen-decomposition of this matrix.
The scree plot in Figure \ref{fig:appendix_K}A demonstrates that $K=10$ has a visual ``elbow'' based on the decreasing eigenvalue.
Furthermore, with a target cumulative variance captured by the first number of latent dimensions. 
Empirically, we have found that capturing 90\% of the variance is a reasonable guideline for retaining biologically relevant information in our analysis.
We see in Figure \ref{fig:appendix_K}B that at $K=10$, this desired amount of variance is retained.

\begin{figure}[tb]
  \centering
  \includegraphics[width=350px]{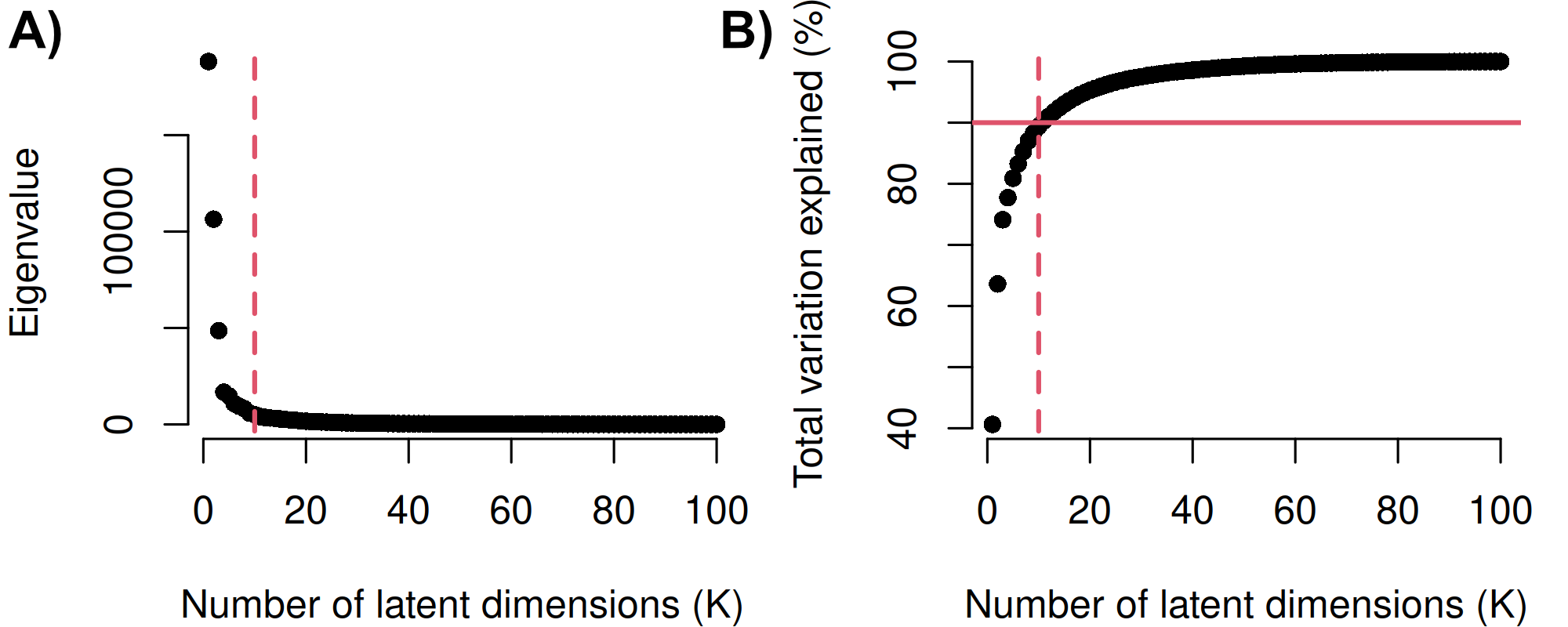}
  \caption
   { 
 A) Scree plot showing the value of eigenvalues of the sum of squared debiased adjacency matrices. 
 B) Cumulative variance captured by the first $K$ latent dimensions based on the eigen-decomposition of the sum of squared debiased adjacency matrices.
 The dashed red vertical line denotes $K=10$, the selected number of latent dimensions.
 The solid red horizontal line in (B) denotes the targeted 90\% of variance captured.
    }
    \label{fig:appendix_K}
\end{figure}

\subsection{Additional plots of results for developing brain}

In the following, we provide additional plots across all 12 networks, showing the communities within each network and how the gene memberships in one network relate to those in other networks.

In Figure \ref{fig:appendix_network-labeled}, we plot the gene memberships for each network, where the graphical layout is held fixed. We can visually observe that specific genes change memberships over time, but most genes do not often change memberships.

\begin{figure}[tb]
  \centering
  \includegraphics[width=300px]{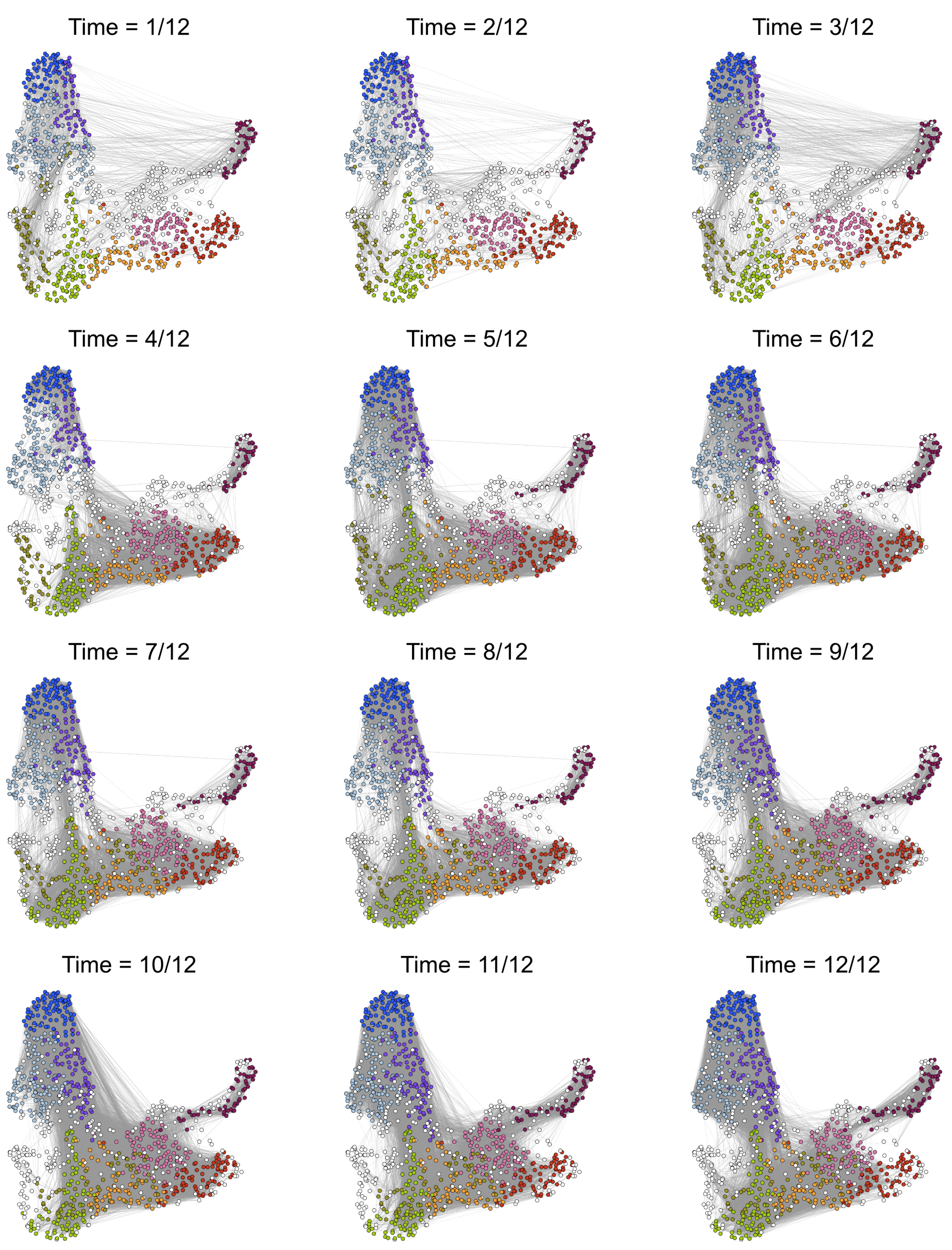}
  \caption
   { 
   Gene memberships across all 12 networks, where the graphical layout is fixed, and the gray lines denote edges between two correlated genes. Each gene is colored one of ten different colors (community 1 as burgundy, community 2 as red, community 3 as salmon, community 4 as orange, community 5 as lime, community 6 as olive, community 7 as purple, community 8 as purple, community 9 as blue, and community 10 as white).
    }
    \label{fig:appendix_network-labeled}
\end{figure}

In Figure \ref{fig:appendix_adjacency}, we plot each of the 12 networks as adjacency matrices (i.e., heatmaps), where the genes are reshuffled from one row/column to the next so that genes in each community are grouped together. We can see an obvious membership structure within each network and slightly varying community sizes across time.

\begin{figure}[tb]
  \centering
  \includegraphics[width=300px]{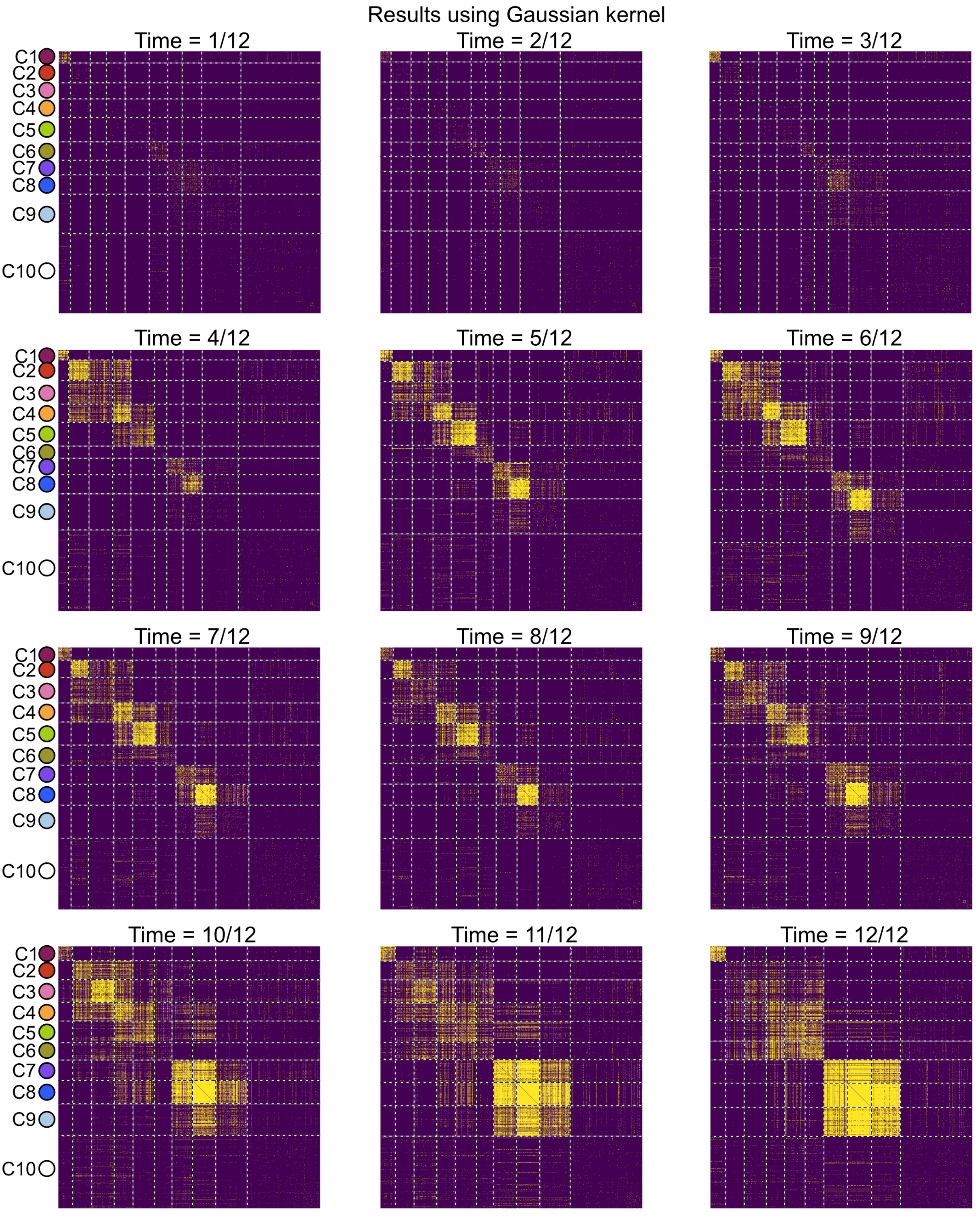}
  \caption
   { 
 Adjacency matrices for each of the 12 networks, where the genes are reshuffled in rows/columns from one plot to the next so that genes in each community are grouped together. The yellow color denotes an edge between two genes, while dark blue denotes the absence of an edge. The communities are separated visually by a white dotted line. The colors for each community are the same as in \ref{fig:appendix_network-labeled}.
    }
    \label{fig:appendix_adjacency}
\end{figure}

In Figure \ref{fig:appendix_connectivity}, we plot the connectivity within  and across communities, which better summarizes the adjacency matrices shown in 
Figure \ref{fig:appendix_adjacency}. Based on Sylvester's criterion, we can see that some of the 12 networks are indefinite (i.e., they contain negative eigenvalues) due to 2-by-2 submatrices along the diagonal that have negative eigenvalues.

\begin{figure}[tb]
  \centering
  \includegraphics[width=350px]{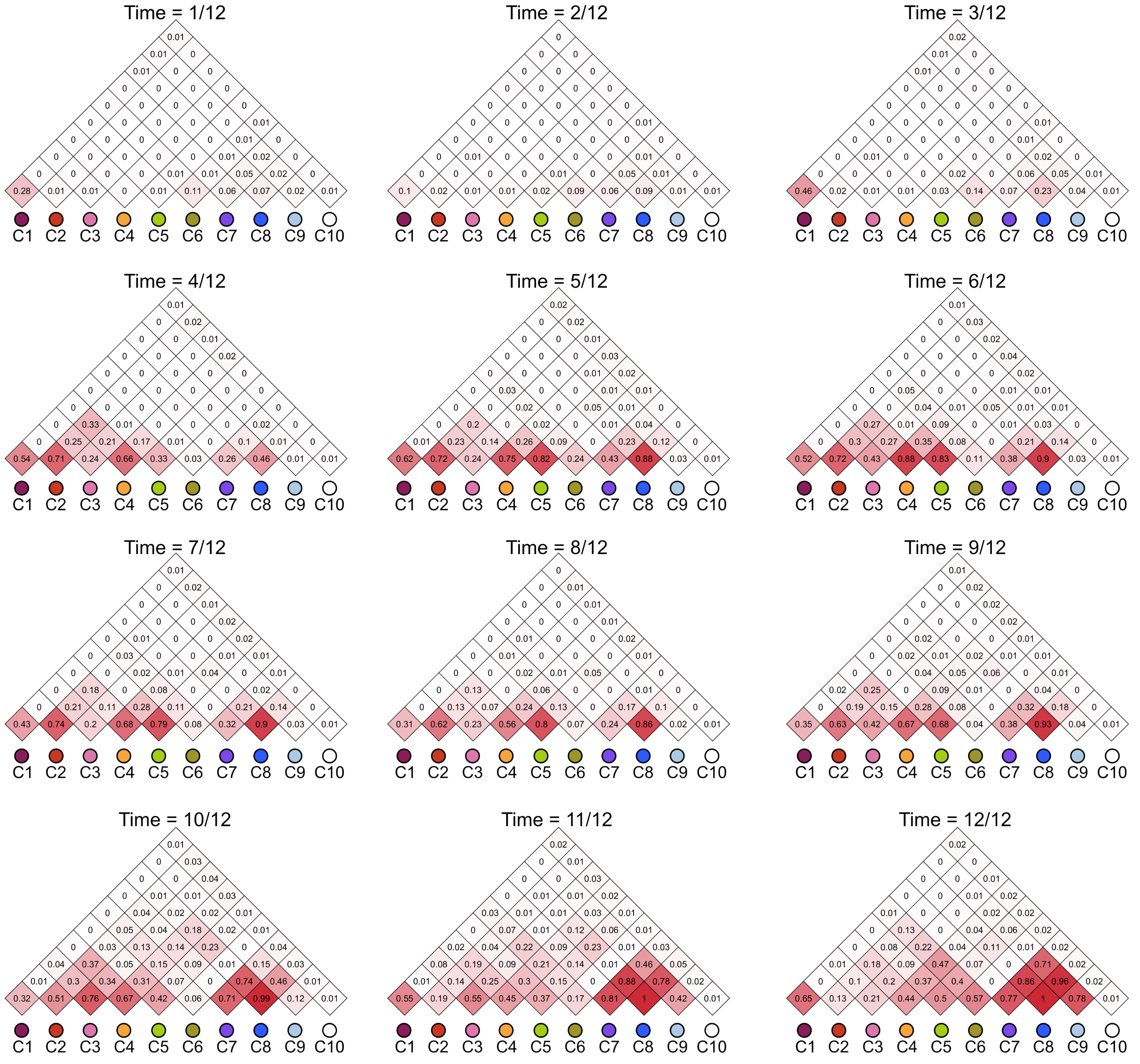}
  \caption
   { 
   Connectivity matrices as heatmaps for each of the 12 networks. The shown numbers denote the percentage of edges within or across communities (among all possible edges), and the colors range from white (i.e., connectivity of 0) to bright red (i.e., connectivity of 1). The colors for each community are the same as in \ref{fig:appendix_network-labeled}.
    }
    \label{fig:appendix_connectivity}
\end{figure}

Lastly, in Figure \ref{fig:appendix_alluvial}, we present the alluvial plots, illustrating how the membership structure evolves from one network to the next and how the 10-dimensional embedding effectively reveals the community structure within each network. This is an extended version of  Figure \ref{fig:alluvial} in the main text.

\begin{figure}[tb]
  \centering
  \includegraphics[width=350px]{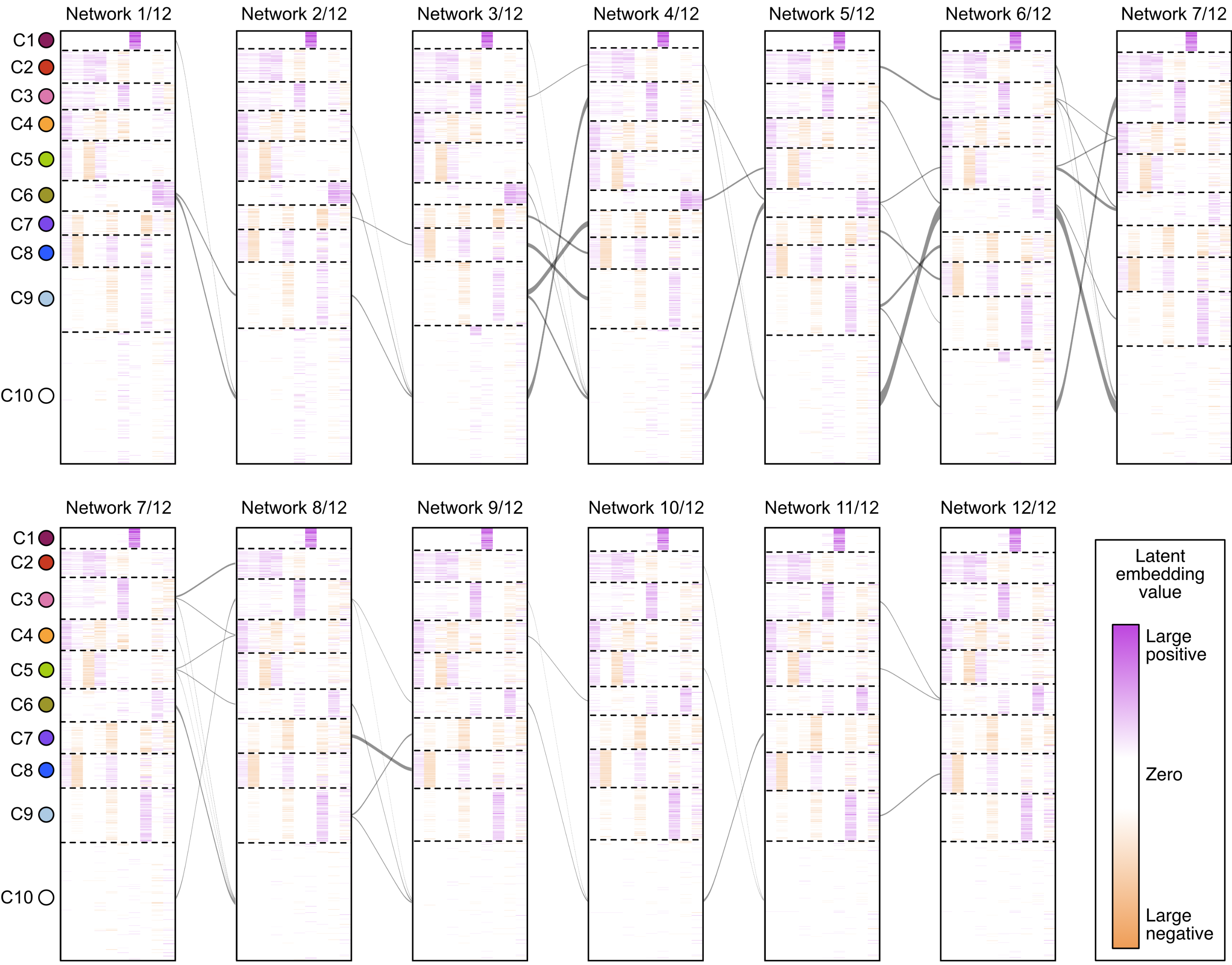}
  \caption
   { 
   Alluvial plots across all 12 networks. This is an extension of the main text's Figure  \ref{fig:alluvial}. The colors for each community are the same as in \ref{fig:appendix_network-labeled}.
      }
    \label{fig:appendix_alluvial}
\end{figure}

\subsection{Alternative analysis using a box kernel}
In Figure \ref{fig:appendix_adjacency_alternative}, we plot the estimated gene communities if we had used a box kernel. We observe that the 10 communities (which do not necessarily correspond one-to-one with the 10 estimated communities initially in Figure \ref{fig:appendix_adjacency}) remain unchanged across all 12 time points. 
That is, the estimated communities using a box kernel remain unchanged over time, even though KD-SoS allows for changes in membership. 
We suspect that this stems from the lack of smoothness in the box kernel.
Because networks are discretely included or excluded in the averaging of the box kernel, our tuning procedure is incentivized to average over all the networks in our data analysis, as the ``signal'' in our single-cell data is weaker than in our simulations.
This yields a non-changing gene community structure, which is biologically unrealistic \citep{fleck2022inferring,kamimoto2023dissecting}.

\begin{figure}[tb]
  \centering
  \includegraphics[width=300px]{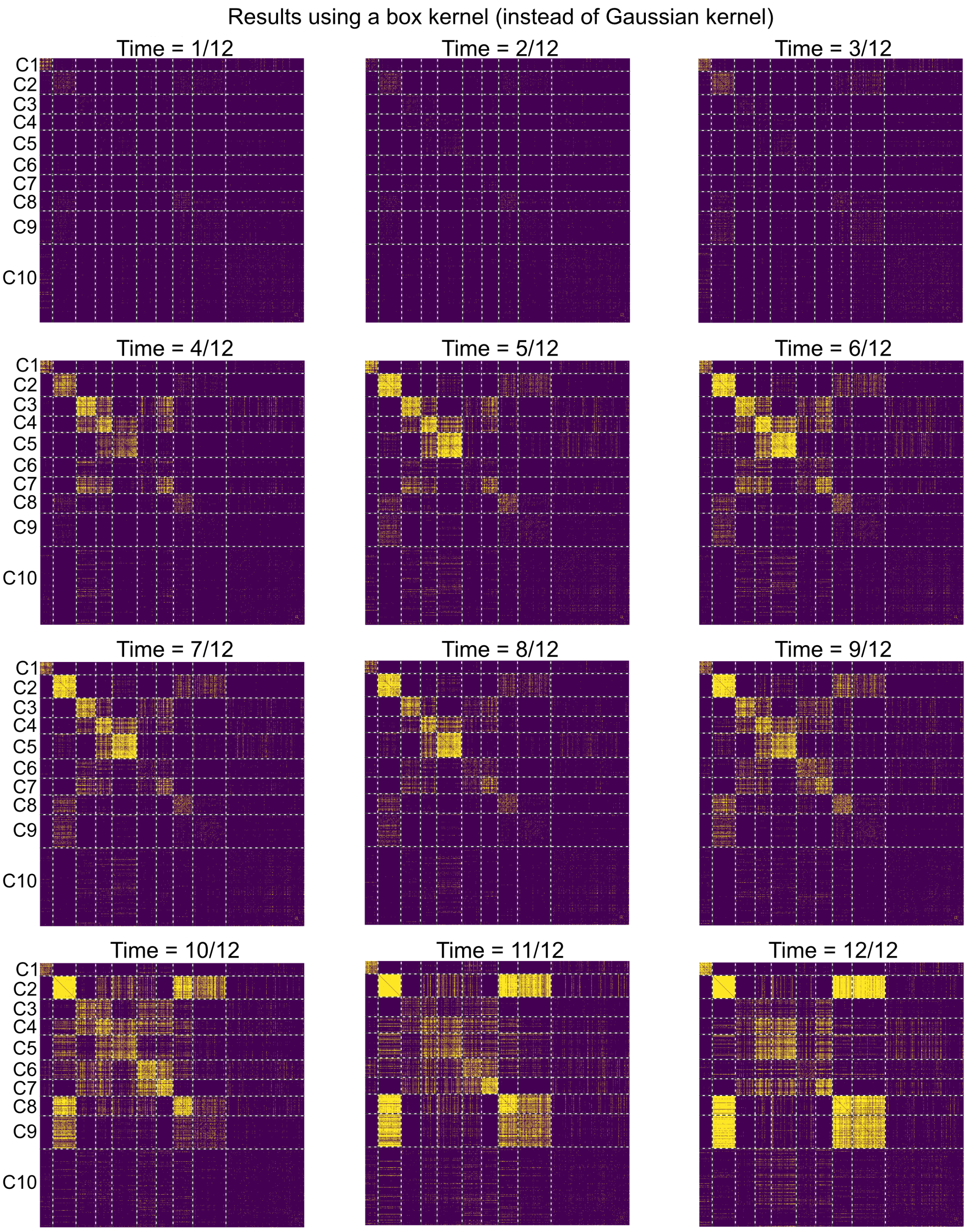}
  \caption
   { 
   Communities of the 12 networks, where the communities are estimated by KD-SoS using a box kernel. The communities do not change over the 12 time points. The figure is shown in the same format as Figure \ref{fig:appendix_adjacency}.
      }
    \label{fig:appendix_adjacency_alternative}
\end{figure}

\end{document}